\documentclass[reqno]{amsart}
\usepackage{amssymb,latexsym,color,}
\usepackage{amsmath}
\usepackage{amsthm}
\usepackage{graphicx}
\usepackage{palatino,mathpazo}
\usepackage{setspace}
\usepackage{titletoc}
\def\XXint#1#2#3{{\setbox0=\hbox{$#1{#2#3}{\int}$}
     \vcenter{\hbox{$#2#3$}}\kern-.5\wd0}}

\newcommand{\un}[1]{\underline{#1}}

\newcommand{\om}{\Omega}

\newcommand{\sscp}{\scriptscriptstyle}

\newcommand{\beq}{\begin{equation}}
\newcommand{\eeq}{\end{equation}}
\newcommand{\rife}[1]{(\ref{#1})}

\newtheorem{theorem}{Theorem}[section]
\newtheorem{definition}{Definition}[section]

\newtheorem{remark}[theorem]{Remark}

\newtheorem{lemma}[theorem]{Lemma}

\numberwithin{equation}{section}

\begin{document}
\title[Uniqueness for bubbling solutions of mean field equation]
{Uniqueness of bubbling solutions of mean field equations}

\author{Daniele Bartolucci$^{a,*}$}

\author{Aleks Jevnikar$^a$}

\address{$^A$ Department of Mathematics, University of Rome {\it "Tor Vergata"},  Via della ricerca scientifica n.1, 00133 Roma,
Italy.}

\author{Youngae Lee$^b$}
\address{$^B$ National Institute for Mathematical Sciences, 70 Yuseong-daero 1689 beon-gil, Yuseong-gu, Daejeon, 34047, Republic of Korea}

\author{ Wen Yang$^c$}
\address{$^C$ Department of Applied Mathematics, Hong Kong Polytechnic University, Hung Hom, Kowloon, Hong Kong}

\medskip

\email{bartoluc@mat.uniroma2.it (D. Bartolucci), jevnikar@mat.uniroma2.it (A. Jevnikar), \newline math.yangwen@gmail.com (W. Yang), youngaelee0531@gmail.com (Y. Lee)}


\thanks{$^*$ Corresponding author.}



\begin{abstract}
We prove the uniqueness of blow up solutions
of the mean field equation as $\rho_n\to 8\pi m$, $m\in\mathbb{N}$. If $u_{n,1}$ and $u_{n,2}$ are
two sequences of bubbling solutions with the same $\rho_n$ and the same (non degenerate) blow up set, then
$u_{n,1}=u_{n,2}$ for sufficiently large $n$. The proof of the uniqueness
requires a careful use of some sharp estimates for bubbling solutions of mean field
equations \cite{cl1} and a rather involved analysis of suitably defined Pohozaev-type identities as recently developed in \cite{ly2} in the context of the
Chern-Simons-Higgs equations. Moreover, motivated by the Onsager statistical description of two dimensional turbulence, we are bound to
obtain a refined version of an estimate about $\rho_n-8\pi m$ in case the first order evaluated in \cite{cl1} vanishes.
\end{abstract}

 \maketitle

\section{Introduction}Let $(M,\mathrm{d} s)$ be a compact Riemann surface with volume $|M|=1$ and $\rho_n>0$ be a sequence
satisfying $\lim_{n\to+\infty}\rho_n=8\pi m$ for some positive integer $m\ge 1$.
We denote by $\mathrm{d}\mu$ the volume form, by  $\Delta_{\sscp M}$ the Laplace-Beltrami operator on $(M,\mathrm{d} s)$,
and consider the following mean field type problem:
\begin{equation}
\label{mfd}\left \{
\begin{array}{l}
\Delta_{\sscp M} u_n+\rho_n\Big(\frac{h(x)e^{u_n(x)}}{\int_Mh e^{u_n }\mathrm{d}\mu }-1\Big)=0\ \ \textrm{in}\ \ M,\\   \int_Mu_n\mathrm{d}\mu=0,\quad u_n\in C^{\infty}(M),
\end{array}
\right.
\end{equation} where $h(x)=h_*(x)e^{-4\pi\sum_{j=1}^N\alpha_j G(x,p_j)}\ge 0$, $p_j$
are   distinct points,  $\alpha_j\in\mathbb{N}$,  $h_*>0$, $h_*\in C^{2,\sigma}(M)$,  and $G$ is the Green function, which satisfies,
\begin{equation*}
-\Delta_{\sscp M} G(x,p)=\delta_p-1~\mathrm{in}~M,~\ \ \mathrm{and}~\int_MG(x,p)\mathrm{d}\mu(x)=0.
\end{equation*}

The mean field equation \eqref{mfd} and the corresponding Dirichlet problem (see \eqref{mf} below) have attracted a lot of attention in recent years
because of their applications to several issues of interest in
Mathematics and Physics, such as Electroweak and Chern-Simons self-dual vortices \cite{sy2}, \cite{T0}, \cite{yang},
conformal metrics on surfaces with \cite{Troy} or without conical singularities \cite{KW},
statistical mechanics of two-dimensional turbulence  \cite{clmp2} and of self-gravitating systems \cite{w} and cosmic strings \cite{pot},
and more recently the theory of hyperelliptic curves \cite{cLin14} and of the Painlev\'e equations \cite{CKLin}.
These was some of the motivations behind the many efforts done to determine
existence \cite{bp},\cite{B5},\cite{BMal},\cite{BDeM},\cite{BdM2},\cite{BdMM}, \cite{bt},\cite{cama},\cite{cl2},\cite{cl4},\cite{DJLW},
\cite{dj},\cite{EGP},\cite{KMdP},  \cite{linwang}, \cite{Mal1}, \cite{Mal2}, \cite{NT},
multiplicity  \cite{BdMM}, \cite{dem2}, uniqueness \cite{bl},\cite{BLin3},\cite{BLT}, \cite{CCL}, \cite{GM1}, \cite{GM2}, \cite{GM3}, \cite{Lin1}, \cite{Lin7}, \cite{Lin8}, \cite{suz}, concentration-compactness and bubbling behavior
see \cite{bt}, \cite{bm}, \cite{yy}, \cite{ls}, and \cite{bcct},  \cite{bt2}, \cite{cl1}, \cite{cl3}, \cite{DeMSR}, \cite{KLin}, \cite{tar}, \cite{Za2}, and the structure of entire solutions \cite{barjga}, \cite{cli1}, \cite{FL}, \cite{PT}, \cite{Tar14}, of \eqref{mfd} and \eqref{mf}.

In spite of the many results at hand, and with few exceptions (see \cite{CCL}, \cite{suz} and more recently  \cite{B5}, \cite{BLin3}),
we still don't know much about the qualitative behaviour of the global bifurcation diagram of solutions of \eqref{mfd} and \eqref{mf}.
What one can infer from the above mentioned results in the sub-critical/critical regime is
that solutions of \eqref{mf} are unique if $\rho_n<8\pi$ and are unique whenever they exists if $\rho_n=8\pi$, see  \cite{suz} and \cite{BLin3}, \cite{CCL},
\cite{GM2}. The same question can be asked about \eqref{mfd} whose answer is still not well understood,
see \cite{Lin1}, \cite{Lin7} and \cite{GM3}. However, solutions of either \rife{mfd} or \rife{mf} are
expected to be generically non unique for $\rho_n>8\pi$, see \cite{BdM2}, \cite{BdMM}, \cite{dem2}.

Our aim here is to contribute in this direction by showing that blow up solutions of \eqref{mfd} and \eqref{mf} are unique for $n$ large enough.

\begin{definition}
Let $u_n$ be a sequence of solutions of  \eqref{mfd}. We say that $u_n$ blows up at the points $q_j\notin\{p_1,\cdots,p_N\}$, $j=1,\cdots,m$, if,
$
\frac{h(x)e^{u_n(x)}}{\int_Mh e^{u_n }\mathrm{d}\mu }\rightharpoonup 8\pi\sum\limits_{j=1}^m\delta_{q_j},
$
weakly in the sense of measure in $M$.
\end{definition}

Let $K(x)$ be the Gaussian curvature at  $x\in M$ and $R(x,y)$ denote the regular part (see section \ref{sec_pre} below),
of the Green function $G(x,y)$. For $\mathbf{q}=(q_1,\cdots,q_m)\in M\times\cdots\times M$, we denote by,
\beq\label{g*}
G_j^*(x)=8\pi R(x,q_j)+8\pi\sum^{1,\cdots,m}_{l\neq j}G(x,q_l),
\eeq
\beq\label{l_q}\ell(\mathbf{q})=\sum_{j=1}^m[\Delta_{\sscp M} \log h(q_j)+8m\pi-2K(q_j)]h(q_j)e^{G_j^*(q_j)},\ \ \textrm{and}
\eeq
\beq\label{f_q_j}
f_{\mathbf{q},j}(x)=8\pi\left[R(x,q_j)- R(q_j,q_j)+\sum^{1,\cdots,m}_{l\neq j}(G(x,q_l)-G(q_j,q_l))\right]+\log \frac{h(x)}{ h(q_j)}.
\eeq

We will denote by $B^{\sscp M}_r(q)$ the geodesic ball of radius $r$ centred at $q\in M$, while $U^{\sscp M}_r(q)$ will denote the pre image of the Euclidean ball of radius $r$, $B_r(q)\subset \mathbb{R}^2$, in a suitably defined isothermal coordinate system (see section \ref{sec_pre} below for further details). If $m\geq 2$ we fix a constant $r_0\in(0,\frac{1}{2})$ and a family of open sets $M_j$ satisfying, $M_l\cap M_j=\emptyset$ if $l\neq j$, $\bigcup_{j=1}^m\overline{M}_j=M$, $U^{\sscp M}_{2r_0}(q_j) \subseteq M_j$, $j=1,\cdots m$. Then, let us define,

\begin{equation}\label{D_q}D(\mathbf{q})=\lim_{r\to0}
 \sum_{j=1}^m h(q_j)e^{G_j^*(q_j)}\left(\int_{M_j\setminus  U^{\sscp M}_{{r_j}}(q_j)}
                 e^{\Phi_j(x,\mathbf{q})}  \mathrm{d}\mu(x)
-\frac{\pi}{{r_j^2}}\right),
\end{equation}
where $M_1=M$ if $m=1$, $r_j=r\sqrt{8h(q_j)e^{G_j^*(q_j)}} $ and,
\begin{equation}\label{P_q}
\Phi_j(x,\mathbf{q})=\sum_{l=1}^m 8\pi  G(x,q_l)-G_j^*(q_j)+\log h(x)-\log h(q_j).
\end{equation}
The quantity $D(\mathbf{q})$ was first introduced in \cite{CCL, clw}. For $(x_1,\cdots, x_m)\in M\times \cdots M$, we also define,
\begin{align}\label{f_q}
&f_m(x_1,x_2,\cdots,x_m)=\sum_{j=1}^{m}\big[\log(h(x_j))+4\pi R(x_j,x_j)\big]+4\pi\sum_{l\neq j}^{1,\cdots,m}G(x_l,x_j),
\end{align}
and let $D_{\sscp M}^2f_m$ be its Hessian tensor field on $M$.
Then we have,

\begin{theorem}\label{thm_concen}Let $u_{n}^{(1)}$ and $u_{n}^{(2)}$ be  two sequences of solutions of \eqref{mfd},
blowing up at the points $q_j\notin\{p_1,
\cdots,p_N\}$, $j=1,\cdots,m$, where $\mathbf{q}=(q_1,\cdots,q_m)$ is a critical point of $f_m$ and $\textrm{det}(D_{\sscp M}^2f_m(\mathbf{q}))\neq 0$.
Assume that $\rho^{(1)}_n=\rho_n=\rho^{(2)}_n$ and that, either,
\begin{enumerate}
\item $\ell(\mathbf{q})\neq0$, or,
\item  $\ell(\mathbf{q})=0$ and $D(\mathbf{q})\neq0$.
\end{enumerate}
Then there exists $n_0\ge1$ such that $u_{n}^{(1)}=u_{n}^{(2)}$ for all $n\ge n_0$.
\end{theorem}

The proof of Theorem \ref{thm_concen} is worked out by an adaptation of an argument recently proposed in \cite{ly2}. In that paper Lin and Yan
prove uniqueness for blow up solutions of the Chern-Simons-Higgs equation. In particular, it is claimed in \cite{ly2} that the
method adopted there does the job also in the case of the mean field equation \rife{mfd} and in fact our aim is to prove that claim. However it seems that
the adaptation of that argument to our problem is not straightforward.

First of all, the cornerstone of the proof is the description of the blow up behavior of solutions established in \cite{cl1}. In that case the leading order of the expansion of $\rho_n-8\pi m$ as well as of the reminder term of blow up solutions is proportional to $\ell(\mathbf{q})$, see section \ref{sec_pre} below. By means of these estimates, if $\ell(\mathbf{q})\neq 0$, we can prove that the difference of the blow up rates (which we denote by $\lambda_{n,j}^{(1)}-\lambda_{n,j}^{(2)}$) is small for large $n$, see Lemma \ref{lem_diff_tilu}. This is why the case $\ell(\mathbf{q})= 0$ is more subtle and this is why we are bound to derive an improved version of the estimate concerning $\rho_n-8\pi m$. A full generalization of the estimates in \cite{cl1} to the case $\ell(\mathbf{q})= 0$, that is, including the reminder term of blow up solutions, at least to our knowledge has been derived only in case $m=1$ and only for the Dirichlet problem, see \cite{CCL}.

\begin{remark} Far from being just a mathematical problem, the case $\ell(\mathbf{q})=0$ often arise in the study of geometric and physical problems, 
as for example in the Onsager statistical mechanical description of two dimensional turbulence, see \cite{clmp2} and more recently \cite{B5}. 
Motivated by this problem,  in the final part of this paper we will discuss the uniqueness result relative to the Dirichlet problem \rife{mf}, 
see Theorem \ref{thm_concenD} below. Indeed, inspired by a recent result \cite{B5}, we believe that, in the non degenerate setting of 
Theorem \ref{thm_concenD} and  for large enough $n$, 1-point blow up solutions could be parametrized by their Dirichlet energy. 
In particular, on domains of second kind \cite{clmp2}, \cite{CCL}, we believe that this fact would imply the existence of a full interval of strict convexity 
of the entropy, see \cite{B5}. We will discuss this problem in a forthcoming paper \cite{BJLY}. 
However it is crucial to the understanding of this application to establish uniqueness in case $\ell(\mathbf{q})=0$. A uniqueness result for $1$-point 
blow up solutions of the Gelfand problem $-\Delta v=\varepsilon e^{v}$ in $\om$, $v=0$ on $\partial \om$ 
was obtained in \cite{GG} in the simpler case where $\om$ is convex and symmetric with respect to both axis.
\end{remark}

Therefore we derive the following improvement of Theorem 1.1 in \cite{cl1}.

\begin{theorem}\label{thm_new}
Let $u_n$ be a sequence of solutions of  \eqref{mfd} which blows up at the points $q_j\notin\{p_1,\cdots,p_N\}$, $j=1,\cdots,m$, $\delta>0$ be a fixed constant and
$
\lambda_{n,j}=\max_{B^{\sscp M}_\delta(q_j)}\Big(u_n-\log ( \;\int\limits_M h e^{u_n})\Big)\
\textrm{for}\  j=1,\cdots, m.
$\\
Then, for any $n$ large enough, the following estimate holds,
\begin{equation}
\label{a5}\begin{aligned}
\rho_n-8\pi m=~&\frac{2\ell(\mathbf{q})e^{-\lambda_{n,1}}}{m h^2(q_1)e^{G_1^*(q_1)}}
\Big(\lambda_{n,1}+ \log\rho_n h^2(q_1)e^{G_1^*(q_1)} \delta^2 -2\Big)\\
&+ \frac{8e^{-\lambda_{n,1}} }{h^2(q_1)e^{G_1^*(q_1)}\pi m}\Big( D(\mathbf{q}) +O(\delta^\sigma)\Big)
+O(\lambda_{n,1}^2e^{-\frac{3}{2}\lambda_{n,1}})+O(e^{-(1+\frac{\sigma}{2})\lambda_{n,1}}),
\end{aligned}\end{equation}
where $\sigma$ is fixed by the assumption $h_*\in C^{2,\sigma}(M)$.
\end{theorem}

The proof of Theorem \ref{thm_new} relies on a careful improvement of an argument first proposed in \cite{cl1}. 
By using Theorem \ref{thm_new}, we succeed in showing that $\lambda_{n,j}^{(1)}-\lambda_{n,j}^{(2)}$ is asymptotically small if  $\ell(\mathbf{q})=0$ and
$D(\mathbf{q})\neq 0$ as well, see Lemma \ref{lem_diff_tilu}.

Then, as in \cite{ly2}, we analyze the asymptotic behavior of
$\zeta_n=\frac{u_{n}^{(1)}-u_{n}^{(2)}}{\|u_n^{(1)}-u_n^{(2)}\|_{L^\infty(M)}}$. Near each blow up point $q_j$,
and after a suitable scaling, $\zeta_n$ converges to an entire solution of the linearized problem associated to the Liouville equation:
\begin{equation}
\label{liouville}
\Delta v+ e^{v}=0\quad\textrm{in}\ \mathbb{R}^2.
\end{equation}
Solutions of \eqref{liouville} are completely classified \cite{cli1} and take the form,
\begin{equation}\label{17}
v\left( z\right)=v_{\mu,a}(z) = \log \frac{8e^{\mu}}{ ( 1+e^{\mu}  \vert z +a  \vert ^{2} ) ^{2}},  \quad \mu  \in \mathbb{R},\;a=(a_1,a_2) \in \mathbb{R}^2.
\end{equation}
The freedom in the choice of $\mu$ and $a$ is due to the well known invariance of equation \eqref{liouville} under dilations and translations.
The linearized operator $L$ relative to $v_{\sscp 0,0}$  is defined by,
\begin{equation}\label{entirelinear}L\phi:=\Delta \phi+\frac{8}{(1+|z|^2)^2}\phi\quad\textrm{in}\ \mathbb{R}^2.
\end{equation}
It is well known, see \cite[Proposition 1]{bp}, that the kernel of $L$ has real dimension $3$ with eigenfunctions $Y_0$, $Y_1$, $Y_2$,
where,
\begin{equation*}
\begin{aligned}  Y_0(z) = \frac{1-|z|^2}{1+ |z|^2}=\frac{\partial v_{\mu,a}}{\partial \mu}\Big|_{(\mu,a)=(0,0)},\
Y_1(z) = \frac{z_1}{1+ |z|^2}=-\frac{1}{4}\frac{\partial v_{\mu,a}}{\partial a_1}\Big|_{(\mu,a)=(0,0)},
\  Y_2(z) =\frac{z_2}{1+ |z|^2}=-\frac{1}{4}\frac{\partial v_{\mu,a}}{\partial a_2}\Big|_{(\mu,a)=(0,0)}.
  \end{aligned}
\end{equation*}
The second crucial point of the proof of Theorem \ref{thm_concen} is to show that, after scaling and for large $n$, $\zeta_n$
is orthogonal to $Y_0$, $Y_1$ and $Y_2$. As in \cite{ly2} this is done by a rather delicate analysis of various suitably defined Pohozaev-type identities.
However, compared with \cite{ly2}, we face a truly new  difficulty, since the difference of the blow up rates (that is
$\lambda_{n,j}^{(1)}-\lambda_{n,j}^{(2)}$)
in our case can be of the same order of $\frac{1}{\lambda_{n,j}^{(1)}}$, a situation which cannot occur in the Chern-Simons-Higgs problem discussed
in \cite{ly2}. In order to overcome this difficulty we have to carry out an higher order expansion of $u_{n}^{(1)}-u_{n}^{(2)}$
by using Green's representation formula. The leading order of that expansion has to be determined explicitly by using the explicit
form of entire solutions of \rife{liouville} (see Lemma \ref{lem_equalb0}). Besides, the main estimates relies on a series of subtle cancellations, 
see Lemma \ref{poho1_lhs} and Lemma \ref{poho1_rhs}.

\begin{remark}
The above argument can be adapted in a non trivial way to address the non degeneracy of the $m$-point bubbling solutions of the mean field equation \eqref{mfd}.
We will discuss this topic in another paper \cite{BJLY}.
\end{remark}

\begin{remark}
We point out that in Theorem \ref{thm_concen} we consider solutions of \eqref{mfd} blowing up at the points $q_j$ such that $q_j\notin\{p_1,\cdots,p_N\}$, $j=1,\cdots,m$, in order to avoid the effect of the vortex points $\{p_1,\cdots,p_N\}$ on the local profile of the bubbling solution. It would be interesting to carry out a refinement of the above argument suitable for treating the latter case as well and to address uniqueness of solutions blowing up possibly at the vortex points.

In this respect, let us mention the papers \cite{coll-sing1, coll-sing2, coll-sing3} where the case of bubbling solutions for collapsing vortex points is studied. In particular, in \cite{coll-sing1} the authors managed to derive the uniqueness of such bubbling solutions by exploiting a similar argument as in the present paper.
\end{remark}

\begin{remark}\label{remreg}
The assumption $\alpha_j\in\mathbb{N}$ is used to guarantee that $u\in C^{\infty}(M)$, which in turn allows a simplified discussion of
 the already very technical proof. However, since by assumption $q_j\notin \{p_1,\cdots,p_m\}$, then we may relax
that assumption and let $\alpha_j\in (-1,+\infty)$. Indeed, the sharp local estimates in \cite{cl1} still hold in this more general setting, just
with minor changes relative to the regularity class of $u_n$. In other words, Theorem \ref{thm_concen} still holds if we allow
$\alpha_j\in (-1,+\infty)$, $j=1,\cdots,m$.
\end{remark}

This paper is organized as follows.
In section \ref{sec_pre} we  review some known sharp estimates for blow up solutions of \eqref{mfd}.
In section \ref{sec_est} we analyse the limit behavior of $\zeta_n$ on each region $U^{\sscp M}_{r_0}(q_j)$ and
$M\setminus \cup_{j=1}^m U^{\sscp M}_{r_0}(q_j)$.
In section \ref{sec_poho} we  prove Theorem \ref{thm_concen} by the analysis of some suitably derived Pohozaev-type identities.
In section \ref{sec_Dir} we discuss the uniqueness of solutions of the Dirichlet problem.
In section \ref{apppendix} we prove Theorem \ref{thm_new}.

\section{Preliminary}\label{sec_pre}In this section we recall some sharp estimates for blow up solutions of \eqref{mfd}.
 Suppose that $u_n$ is a sequence of blow-up solutions of  \eqref{mfd} which blows up at $q_j\notin\{p_1,\cdots,p_N\}$, $j=1,\cdots, m$.
Let \begin{align}\notag
\tilde{u}_n(x)=u_n(x)-\log\left(\int_Mhe^{u_n}\mathrm{d}\mu\right).
\end{align}
Then it is easy to see that,
    \begin{equation}
\label{eq_tilu}
\Delta_{\sscp M} \tilde{u}_n+\rho_n(h(x)e^{\tilde{u}_n(x)}-1)=0\ \ \textrm{in}\ \ M,\ \  \textrm{and}
\end{equation}
\begin{equation}
\label{int_tilu}
\int_M he^{\tilde{u}_n}\mathrm{d}\mu=1.
\end{equation}
We denote by,
\begin{align}\notag
\lambda_{n}=\max_{M}\tilde{u}_n,\ \ \textrm{and}
\end{align}
\begin{align}\notag
\lambda_{n,j}=\max_{B^{\sscp M}_\delta(q_j)}\tilde{u}_n=\tilde{u}_n(x_{n,j})\
\textrm{for}\  j=1,\cdots, m,\ \textrm{where}\ \delta>0\ \textrm{is a fixed constant.}
\end{align}

Next, let us introduce some notations for local computations. We introduce a local isothermal coordinate system $\un{x}=T_j(x)\in \mathbb{R}^2$, such that $\un{q}_j=T_j(q_j)$, $\un{x}_{n,j}=T_j(x_{n,j})$ and
$\mathrm{d} s^2=e^{2\varphi_j(\un{x})}|\mathrm{d} \un{x}|^2$ with $\varphi_j(\un{x}_{n,j})=0$ and $\nabla\varphi_j(\un{x}_{n,j})=0$. It will be also useful to denote by $U^{\sscp M}_r(x_0)=T_j^{-1}(B_r(\un{x}_0))$, the pre-image of $B_r(\un{x}_0)$, where $\un{x}_0=T_j(x_0)$ and $B_{r}(\un{x}_0)\subset \mathbb{R}^2$ denotes the Euclidean ball of radius $r$ centred at $\un{x}_0\in \mathbb{R}^2$. Therefore, when evaluated in $U^{\sscp M}_{\delta}(x_{n,j})$, in local coordinates \eqref{eq_tilu} takes the form,
\begin{equation}\notag
\Delta \tilde{u}_n+\rho_ne^{2\varphi_j(\un{x})}(h(\un{x})e^{\tilde{u}_n(\un{x})}-1)=0\ \ \textrm{in}\ \
\un{x}\in B_{\delta}(\un{x}_{n,j}),
\end{equation}
where $\Delta=\sum\limits_{i=1}^2\frac{\partial^2}{\partial x_i^2}$ denotes the standard Laplacian in $\mathbb{R}^2$.

For later use we recall that  $r_0>0$ is defined as right after \eqref{f_q_j}
to guarantee  that,
\beq\label{190117}
U^{\sscp M}_{2r_0}(q_j) \subseteq M_j,\; j=1,\cdots m.
\eeq

\begin{remark}
To simplify the exposition we will use the expressions $\tilde{u}, h, R, G, K,...$ to denote those function in both global and local coordinates. It will be clear time to time which one of the functions involved is being used.
\end{remark}

\begin{remark}
We will often need to take back local estimates into globally defined quantities. Therefore we fix an atlas whose local maps are denoted by $\{T_a\}$, and whenever for some $k\geq 1$ we have
$g=g(\un{x}_1,\cdots,\un{x}_k)$ with $\un{x}_i=T_{a_i}(x_i)$,
$i=1,\cdots,k$, then we will denote by $T^{-1}_*(g(\un{x}_1,\cdots,\un{x}_k))=g(T_{a_1}(x_1),\cdots, T_{a_k}(x_k))$.
\end{remark}

It is well known that the conformal factor $\varphi_a$ is a solution of,
\begin{equation}\label{delr1}
-\Delta \varphi_a=e^{2\varphi_a}K,\quad \un{x}\in B_\delta(\un{x}_0).
\end{equation}

The regular part of the Green function $R(x,y)$ is defined in a local isothermal coordinate system $\un{x}=T_a(x)$ as follows. 
For $\un{y}=T_a(y)$ fixed, we can choose the conformal factor $e^{2\varphi_a(\un{x})}$ so that $\varphi_a(\un{y})=0$. 
Then $R(x,y)$ is defined to be the unique solution of
\begin{equation}\label{delr}
\Delta R(\un{x},\un{y})= e^{2\varphi_j (\un{x})},\;\un{x}\in B_\delta(\un{x}_0),\quad R(\un{x},\un{y})=G(\un{x},\un{y})+
\frac{1}{2\pi}\log(|\un{x}-\un{y}|),
\;\un{x}\in\partial B_\delta(\un{x}_0),
\end{equation}
and therefore it is not difficult to check that it also satisfies,
$${R(\un{x},\un{y})=\frac{1}{2\pi}\log|\un{x}-\un{y}|+G(\un{x},\un{y})}.$$

Next, let us define,
\begin{align}\label{Un}
U_{n,j}(\un{x})=\log\frac{e^{\lambda_{n,j}}}{(1+\frac{\rho_nh(\un{x}_{n,j})}{8}e^{\lambda_{n,j}}|\un{x}-\un{x}_{n,j,*}|^2)^2},\quad  \un{x}\in \mathbb{R}^2,
\end{align}
where the point $\un{x}_{n,j,*}$ is chosen to satisfy,
\[\nabla U_{n,j}(\un{x}_{n,j})=\nabla(\log h(\un{x}_{n,j})).\]
Then, it is not difficult to check that,
\begin{align}\label{diff_x_nj}
|\un{x}_{n,j}-\un{x}_{n,j,*}|=O(e^{-\lambda_{n,j}}).
\end{align}
Let us also define,
\begin{align}\label{eta}
\eta_{n,j}(\un{x})=\tilde{u}_n(\un{x})-U_{n,j}(\un{x})-(G_{j}^{*}(\un{x})-G_{j}^{*}(\un{x}_{n,j})), \quad \un{x}\in B_{\delta}(\un{x}_{n,j}).
\end{align}
It has been proved in \cite[Theorem 1.4]{cl1} that, for $x\in B_{\delta}(\un{x}_{n,j})$, it holds,
\begin{equation}\begin{aligned}\label{eta_ets}
  \eta_{n,j}(\un{x})=~&\frac{-8}{\rho_nh(\un{x}_{n,j})}[\Delta\log h(\un{x}_{n,j})+8\pi m -2K(\un{x}_{n,j})]e^{-\lambda_{n,j}}
  [\log(R_{n,j}|\un{x}-\un{x}_{n,j}|+2)]^2\\&+O(\log(R_{n,j}|\un{x}-\un{x}_{n,j}|+2)e^{-\lambda_{n,j}})+
  O(\lambda_{n,j} e^{-\lambda_{n,j}})=O(\lambda_{n,j}^2 e^{-\lambda_{n,j}}),\end{aligned}\end{equation}
where $R_{n,j}=\sqrt{\frac{\rho_nh(\un{x}_{n,j})e^{\lambda_{n,j}}}{8}}$. It has also been proved in \cite[Corollary 2.4]{cl1} that
one can find constants $c>0$ and $c_\delta>0$ such that,
\begin{equation}\label{info_lambda}
\begin{aligned}
  &|\lambda_n-\lambda_{n,j}|\le c\ \ \textrm{for}\ \ j=1,\cdots, m,\ \ \ \ \ \
 |\tilde{u}_n(x)+\lambda_n|\le c_\delta \ \ \textrm{for}\ \ x\in M\setminus \cup_{j=1}^m B^{\sscp M}_{\delta}(q_j).
  \end{aligned}
\end{equation}
Moreover, see  \cite[section 3]{cl1}, we have,
\begin{equation}
\label{relation_lambda}
e^{\lambda_{n,j}}h^2(x_{n,j})e^{G_j^*(x_{n,j})}=e^{\lambda_{n,1}}h^2(x_{n,1})e^{G_1^*(x_{n,1})}
(1+O(e^{-\frac{\lambda_{n,1}}{2}})),
\end{equation}
and in particular, see \cite[Theorem 1.4]{cl1}, the following estimate holds,
\begin{equation}
\begin{aligned}
\label{lambda_exp}
&\lambda_{n,j}+\int_M\tilde{u}_n\mathrm{d}\mu+
2\log\left(\frac{\rho_n h(x_{n,j})}{8}\right)+G_j^*(x_{n,j})\\&=-\frac{2}{\rho_n h(x_{n,j})}
(\Delta_{\sscp M} \log h(x_{n,j})+8\pi m -2K(x_{n,j}))\lambda_{n,j}^2e^{-\lambda_{n,j}}+O(\lambda_{n,j} e^{-\lambda_{n,j}}).
\end{aligned}
\end{equation}

Let us also recall, see \cite[Lemma 5.4]{cl1}, that,
\begin{equation}
\label{gradatq}
\nabla_{\sscp M}[\log h(x)+G_j^*(x)]\Big|_{x=x_{n,j}}=O(\lambda_{n,j} e^{-\lambda_{n,j}}),
\end{equation}
where $\nabla_{\sscp M}$ is a suitable gradient vector field on $M$,
which, together with the assumption $\det (D^2 f_m(\mathbf{q}))\neq 0$, shows that,
\begin{equation}
\label{diffxandq}
|\un{x}_{n,j}-{\un{q}}_j|=O(\lambda_{n,j} e^{-\lambda_{n,j}}).
\end{equation}

\begin{remark}
We remark that, since in any local isothermal coordinate system it holds $\Delta_{\sscp M}=e^{-2\varphi_j}\Delta$, then, in view of
\eqref{diffxandq} and $\varphi_j(\un{x}_{n,j})=0, \nabla \varphi_j (\un{x}_{n,j})=0$, we find that,
\begin{equation*}\begin{aligned}
\Delta_{\sscp M} \log h(x_{n,j})&{=e^{-2\varphi_j(\un{x}_{n,j})}\Delta \log h(\un{x}_{n,j}) }
 {=e^{-2\varphi_j(\un{q}_{j})}\Delta \log h(\un{q}_{j})+O(\lambda_{n,j} e^{-\lambda_{n,j}})= \Delta_{\sscp M} \log h(\un{q}_{j}) +O(\lambda_{n,j} e^{-\lambda_{n,j}}).}
\end{aligned}\end{equation*}
This fact will be often used in the many estimates involved.
\end{remark}

The local masses  corresponding to the blow up of $\tilde{u}_n$ at $q_j$, $1\le j\le m$, are defined as follows,
\begin{equation}
\label{rhonj}
\rho_{n,j}=\rho_n \int_{U^{\sscp M}_{\delta}(q_j)}he^{\tilde{u}_n}\mathrm{d}\mu,
\end{equation}
and we will use the following estimate proved in \cite[section 3]{cl1},
\begin{equation}\label{rhon8pi}
\rho_{n,j}-8\pi=\frac{16\pi}{\rho_nh(\un{x}_{n,j})}\{\Delta\log h(\un{x}_{n,j})+\rho_n-2K(\un{x}_{n,j})\}\lambda_{n,j} e^{-\lambda_{n,j}}+O(e^{-\lambda_{n,j}}),
\end{equation}
In particular, see \cite[Theorem 1.1]{cl1}, we have:
\begin{equation}\begin{aligned}\label{rhon}
 \rho_n-8\pi m  &=\frac{2}{m}\sum_{j=1}^m h^{-1}(\un{x}_{n,j})[\Delta\log h(\un{x}_{n,j})+8\pi m -2K(\un{x}_{n,j})]
\lambda_{n,j} e^{-\lambda_{n,j}}+O(e^{-\lambda_{n,j}})
\\&=\frac{2}{m}\frac{\lambda_{n,1}e^{-\lambda_{n,1}}}{h^2(\un{x}_{n,1})
e^{G_1^*(\un{x}_{n,1})}}\sum_{j=1}^m[\Delta\log h(\un{x}_{n,j})+8\pi m -2K(\un{x}_{n,j})]h(\un{x}_{n,j})e^{G_j^*(\un{x}_{n,j})}
 +O(e^{-\lambda_{n,1}})
\\ &=\frac{2}{m}\frac{\lambda_{n,1}e^{-\lambda_{n,1}}}{h^2({x}_{n,1})
e^{G_1^*({x}_{n,1})}}\ell(\mathbf{q})
+O(e^{-\lambda_{n,1}}).
\end{aligned}
\end{equation}

The asymptotic behavior of $\tilde{u}_n$ on $M\setminus \cup_{j=1}^m U^{\sscp M}_{\delta}(q_j)$, is well described in terms of the auxiliary function,
\begin{equation}\label{wnx}
w_n(x)=\tilde{u}_n(x)-\sum_{j=1}^m\rho_{n,j} G(x,x_{n,j})-\int_M\tilde{u}_n \mathrm{d}\mu.
\end{equation}
which satisfies, see \cite[Lemma 5.3]{cl1},
\begin{equation}\label{est_wn}
w_n=o(e^{-\frac{\lambda_{n,j}}{2}})\ \ \textrm{on} \  \  C^1(M\setminus \cup_{j=1}^m U^{\sscp M}_{\delta}(q_j)).
\end{equation}

\section{Uniqueness of the blow up solutions with mass concentration}\label{sec_est}
To prove Theorem \ref{thm_concen} we argue by contradiction and assume
that \eqref{mfd} has two different solutions $u_n^{(1)}$ and $u_n^{(2)}$, with $\rho_n^{(1)}=\rho_n=\rho_n^{(2)}$,
which blow  up at $q_j$, $j=1,\cdots,m$. We will use $x_{n,j}^{(i)}$, $\lambda_n^{(i)}$,
$\lambda_{n,j}^{(i)}$, $\tilde{u}_n^{(i)}$, $R_{n,j}^{(i)}$, $U_{n,j}^{(i)}$, $x_{n,j,*}^{(i)}$, $w_n^{(i)}$, $\rho_{n,j}^{(i)}$
to denote $x_{n,j}$, $\lambda_n$, $\lambda_{n,j}$, $\tilde{u}_n$, $R_{n,j}$, $U_{n,j}$, $x_{n,j,*}$, $w_n$, $\rho_{n,j}$, as defined
in section 2, corresponding to $u_n^{(i)}$, $i=1,2$, respectively.

Our first result is an estimate about $|\lambda_{n,j}^{(1)}-\lambda_{n,j}^{(2)}|$ and $\|\tilde{u}_n^{(1)}-\tilde{u}_n^{(2)}\|_{L^\infty(M)}$.

\begin{lemma}\label{lem_diff_tilu}
$(i)$ $|\lambda_{n,j}^{(1)}-\lambda_{n,j}^{(2)}|=O\left(\sum_{i=1}^2\frac{1}{\lambda_{n,1}^{(i)}}\right)$ for all $1\le j\le m$.

$(ii)$ there exists a constant $c>1$ such that:
\begin{equation*}\begin{aligned}\frac{1}{c}|\lambda_{n,1}^{(1)}-\lambda_{n,1}^{(2)}|+O (\sum_{i=1}^2\lambda_{n,1}^{(i)}e^{-\frac{\lambda_{n,1}^{(i)}}{2}} )&\le \|\tilde{u}_n^{(1)}-\tilde{u}_n^{(2)}\|_{L^\infty(M)} \le c|\lambda_{n,1}^{(1)}-\lambda_{n,1}^{(2)}|+O (\sum_{i=1}^2\lambda_{n,1}^{(i)}e^{-\frac{\lambda_{n,1}^{(i)}}{2}} ).\end{aligned}\end{equation*}
\end{lemma}

\begin{proof} (i) In view of \eqref{eta} and \eqref{eta_ets}, we see that,  for $\un{x}\in B_{\delta}(\un{q}_j)$, it holds,
\begin{equation}\begin{aligned}\label{tiluest}
 \tilde{u}_n^{(1)}(\un{x})-\tilde{u}_n^{(2)}(\un{x}) &=U_{n,j}^{(1)}(\un{x})-U_{n,j}^{(2)}(\un{x})+
G_j^*(\un{x}_{n,j}^{(2)})-G_j^*(\un{x}_{n,j}^{(1)})+\eta_{n,j}^{(1)}(\un{x})-\eta_{n,j}^{(2)}(\un{x})
\\&=U_{n,j}^{(1)}(\un{x})-U_{n,j}^{(2)}(\un{x})+G_j^*(\un{x}_{n,j}^{(2)})-G_j^*(\un{x}_{n,j}^{(1)}) +
O(\sum_{i=1}^2(\lambda_{n,j}^{(i)})^2 e^{-\lambda_{n,j}^{(i)}}).
\end{aligned}
\end{equation}
By the definition of $U_{n,j}^{(i)}$, we find,
\begin{equation}\label{diffU}
U_{n,j}^{(1)}(\un{x})-U_{n,j}^{(2)}(\un{x})=
2\log\frac{(1+\frac{\rho_nh(\un{x}_{n,j}^{(2)})}{8}e^{\lambda_{n,j}^{(2)}}|\un{x}-\un{x}_{n,j,*}^{(2)}|^2)}{
(1+\frac{\rho_nh(\un{x}_{n,j}^{(1)})}{8}e^{\lambda_{n,j}^{(1)}}|\un{x}-\un{x}_{n,j,*}^{(1)}|^2)}
+\lambda_{n,j}^{(1)}-\lambda_{n,j}^{(2)},
\end{equation}
while, by \eqref{diff_x_nj} and \eqref{diffxandq}, we also have,
\begin{equation}\label{diff_x1_x_2}
|\un{x}_{n,j}^{(1)}-\un{x}_{n,j}^{(2)}|=O(\sum_{i=1}^2\lambda_{n,j}^{(i)} e^{-\lambda_{n,j}^{(i)}}), \quad\textrm{and}\ \ \
|\un{x}_{n,j,*}^{(1)}-\un{x}_{n,j,*}^{(2)}|
=O(\sum_{i=1}^2\lambda_{n,j}^{(i)} e^{-\lambda_{n,j}^{(i)}})\ \ \textrm{for any} \ \ 1\le j\le m.
\end{equation}
{At this point we conclude the proof of Lemma \ref{lem_diff_tilu} by considering two distinct cases:}

{Case 1. $\ell(\mathbf{q})\neq0$:  }
From \eqref{rhon} we have,
\begin{equation}\label{lmnj}
  \begin{aligned}\frac{2}{m}\frac{\lambda_{n,1}^{(1)}e^{-\lambda_{n,1}^{(1)}}}{h^2(x_{n,1}^{(1)})
  e^{G_1^*(x_{n,1}^{(1)})}}\ell(\mathbf{q})
+O(e^{-\lambda_{n,1}^{(1)}})=\frac{2}{m}\frac{\lambda_{n,1}^{(2)}e^{-\lambda_{n,1}^{(2)}}}{h^2(x_{n,1}^{(2)})
  e^{G_1^*(x_{n,1}^{(2)})}}\ell(\mathbf{q})
+O(e^{-\lambda_{n,1}^{(2)}}),
  \end{aligned}
\end{equation}
that is, dividing by  $\frac{2}{m}\frac{\lambda_{n,1}^{(2)}e^{-\lambda_{n,1}^{(2)}}}{h^2(x_{n,1}^{(1)})
  e^{G_1^*(x_{n,1}^{(1)})}}\ell(\mathbf{q})$, and in view of  \eqref{diff_x1_x_2},
\begin{equation*}
  \frac{\lambda_{n,1}^{(1)}}{\lambda_{n,1}^{(2)}}e^{-(\lambda_{n,1}^{(1)}-\lambda_{n,1}^{(2)})}
  =1+O\left(\sum_{i=1}^2\frac{1}{\lambda_{n,1}^{(i)}}\right),
\end{equation*} which in turn implies that,
\begin{equation}
\label{for_dif}
-(\lambda_{n,1}^{(1)}-\lambda_{n,1}^{(2)})+\log\frac{\lambda_{n,1}^{(1)}}{\lambda_{n,1}^{(2)}}
=\log\left(1+O\left(\sum_{i=1}^2\frac{1}{\lambda_{n,1}^{(i)}}\right)\right)
=O\left(\sum_{i=1}^2\frac{1}{\lambda_{n,1}^{(i)}}\right).
\end{equation}
Since, for some $\theta\in (0,1)$, we have $\log\frac{\lambda_{n,j}^{(1)}}{\lambda_{n,j}^{(2)}}=\frac{\lambda_{n,j}^{(1)}
-\lambda_{n,j}^{(2)}}{\theta\lambda_{n,j}^{(1)}+(1-\theta)\lambda_{n,j}^{(2)}}
=o(\lambda_{n,j}^{(1)}-\lambda_{n,j}^{(2)})$, then \eqref{for_dif}
implies that,
\begin{equation}
\label{difflam}
|\lambda_{n,1}^{(1)}-\lambda_{n,1}^{(2)}|=O\left(\sum_{i=1}^2\frac{1}{\lambda_{n,1}^{(i)}}\right).
\end{equation}
As a consequence, by using also \eqref{relation_lambda}, we conclude that
\begin{equation}\label{difflamj}
|\lambda_{n,j}^{(1)}-\lambda_{n,j}^{(2)}|=O\left(\sum_{i=1}^2\frac{1}{\lambda_{n,1}^{(i)}}\right)\quad\textrm{for all}\ \ 1\le j\le m,
\end{equation}
{whenever $\ell(\mathbf{q})\neq0$, as claimed.}

{Case 2. $\ell(\mathbf{q})=0$ and $D(\mathbf{q})\neq0$:   In view of \eqref{a5},  we have,}
\begin{equation}
\label{ad5}\begin{aligned}
 &{  \frac{8(e^{-\lambda_{n,1}^{(1)}}-e^{-\lambda_{n,1}^{(2)}}) }{h^2(q_1)e^{G_1^*(q_1)}\pi m}\Big( D(\mathbf{q}) +O(\delta^\sigma)\Big)}
{ =   O(\sum_{i=1}^2(\lambda_{n,1}^{(i)})^2e^{-\frac{3}{2}\lambda_{n,1}^{(i)}})+O(\sum_{i=1}^2( e^{-(1+\frac{\sigma}{2})\lambda_{n,1}^{(i)}}),}
\end{aligned}\end{equation}
{and then the same argument used in Case 1 above shows that if $\ell(\mathbf{q})=0$ and $D(\mathbf{q})\neq0$,
 then \eqref{difflamj} holds as well.}

{(ii)} Next, in view of \eqref{diff_x1_x_2} and \eqref{difflamj}, we see that,
\begin{equation*}\begin{aligned}
  & h(\un{x}_{n,j}^{(2)})e^{\lambda_{n,j}^{(2)}}|\un{x}-\un{x}_{n,j,*}^{(2)}|^2
  -h(\un{x}_{n,j}^{(1)})e^{\lambda_{n,j}^{(1)}}|\un{x}-\un{x}_{n,j,*}^{(1)}|^2
  \\ &=O(e^{\lambda_{n,j}^{(1)}})\Bigg\{(|\un{x}-\un{x}_{n,j,*}^{(1)}||\un{x}_{n,j,*}^{(1)}-\un{x}_{n,j,*}^{(2)}|+
  |\un{x}_{n,j,*}^{(1)}-\un{x}_{n,j,*}^{(2)}|^2
    + |\un{x}-\un{x}_{n,j,*}^{(1)}|^2(|\lambda_{n,j}^{(1)}-\lambda_{n,j}^{(2)}|+|\un{x}_{n,j}^{(1)}-\un{x}_{n,j}^{(2)}|))\Bigg\},
\end{aligned}\end{equation*}
which, together with \eqref{diffU}, \eqref{diff_x1_x_2}, allows us to conclude that,
\begin{equation}
\label{diff_U}
\|U_{n,j}^{(1)}-U_{n,j}^{(2)}\|_{L^{\infty}(B_{\delta}(\un{q}_j))}=O(1)\left(|\lambda_{n,1}^{(1)}-\lambda_{n,1}^{(2)}|+\sum_{i=1}^2\lambda_{n,j}^{(i)} e^{-\frac{\lambda_{n,j}^{(i)}}{2}}\right).
\end{equation}
From \eqref{tiluest}, \eqref{diff_x1_x_2}, and \eqref{diff_U}, we finally obtain that,
\begin{equation}\label{diff_on_small ball}
\|\tilde{u}_n^{(1)}-\tilde{u}_n^{(2)}\|_{L^{\infty}(B_{\delta}(\un{q}_j))}=
O(1)\left(|\lambda_{n,1}^{(1)}-\lambda_{n,1}^{(2)}|+\sum_{i=1}^2\lambda_{n,j}^{(i)} e^{-\frac{\lambda_{n,j}^{(i)}}{2}}\right) \ \ \textrm{for all} \ \ 1\le j\le m.
\end{equation}

Next we estimate $\tilde{u}_n^{(1)}-\tilde{u}_n^{(2)}$ in $M\setminus \cup_{j=1}^m U^{\sscp M}_{{\delta}}(q_j)$.
By the Green's representation formula, we see that, for $x\in M\setminus \cup_{j=1}^m U^{\sscp M}_{{\delta}}(q_j)$, it holds,
\begingroup
\abovedisplayskip=3pt
\belowdisplayskip=3pt
\begin{equation*}
\begin{aligned}
 \tilde{u}_n^{(1)}(x)-\tilde{u}_n^{(2)}(x)-\int_M(\tilde{u}_n^{(1)}-\tilde{u}_n^{(2)})\mathrm{d}\mu
 &=\rho_n\int_M G(y,x) h(y) (e^{\tilde{u}_n^{(1)}(y)}-e^{\tilde{u}_n^{(2)}(y)})\mathrm{d}\mu(y)
\\&=\rho_n\sum_{j=1}^m\int_{U^{\sscp M}_{\frac{{\delta}}{4}}(x_{n,j}^{(1)})}
(G(y,x)-G(x_{n,j}^{(1)},x))h(y) (e^{\tilde{u}_n^{(1)}(y)}-e^{\tilde{u}_n^{(2)}(y)})\mathrm{d}\mu(y)
\\&\quad+\sum_{j=1}^mG(x_{n,j}^{(1)},x)\int_{U^{\sscp M}_{\frac{{\delta}}{4}}(x_{n,j}^{(1)})}
\rho_nh(y) (e^{\tilde{u}_n^{(1)}(y)}-e^{\tilde{u}_n^{(2)}(y)})
\mathrm{d}\mu(y)
\\&\quad+\rho_n\int_{M\setminus \cup_{j=1}^m U^{\sscp M}_{\frac{{\delta}}{4}}(x_{n,j}^{(1)})}G(y,x)h(y) (e^{\tilde{u}_n^{(1)}(y)}-e^{\tilde{u}_n^{(2)}(y)})
\mathrm{d}\mu(y).
\end{aligned}
\end{equation*}
\endgroup
In view of \eqref{rhonj} and \eqref{info_lambda}, we find that, for $x\in M\setminus \cup_{j=1}^m U^{\sscp M}_{ {\delta} }(q_j)$,
\begingroup
\abovedisplayskip=3pt
\belowdisplayskip=3pt
\begin{equation}\begin{aligned}\label{green_rep}
 \tilde{u}_n^{(1)}(x)-\tilde{u}_n^{(2)}(x)-\int_M(\tilde{u}_n^{(1)}-\tilde{u}_n^{(2)})\mathrm{d}\mu
 &=\rho_n\sum_{j=1}^m\int_{U^{\sscp M}_{\frac{{\delta}}{4}}(x_{n,j}^{(1)})}(G(y,x)-G(x_{n,j}^{(1)},x))h(y) (e^{\tilde{u}_n^{(1)}(y)}-e^{\tilde{u}_n^{(2)}(y)})\mathrm{d}\mu(y)
\\&\quad+\sum_{j=1}^mG(x_{n,j}^{(1)},x)(\rho_{n,j}^{(1)}-\rho_{n,j}^{(2)})+O(\sum_{i=1}^2e^{-\lambda_{n,1}^{(i)}}).
\end{aligned}
\end{equation}
\endgroup
We also have, from \eqref{rhon8pi}, \eqref{diff_x1_x_2}, and \eqref{difflamj},
\begingroup
\abovedisplayskip=3pt
\belowdisplayskip=3pt
\begin{equation}
\begin{aligned}\label{drhon8pi}
\rho_{n,j}^{(1)}-\rho_{n,j}^{(2)} =~&\frac{16\pi}{\rho_nh(\un{x}_{n,j}^{(1)})}\{\Delta \log h(\un{x}_{n,j}^{(1)})+\rho_n-2K(\un{x}_{n,j}^{(1)})\}
\lambda_{n,j}^{(1)} e^{-\lambda_{n,j}^{(1)}}\\
&-\frac{16\pi}{\rho_nh(\un{x}_{n,j}^{(2)})}\{\Delta\log h(\un{x}_{n,j}^{(2)})+\rho_n-2K(\un{x}_{n,j}^{(2)})\}\lambda_{n,j}^{(2)}e^{-\lambda_{n,j}^{(2)}}
+O(\sum_{i=1}^2e^{-\lambda_{n,j}^{(i)}})
= O(\sum_{i=1}^2e^{-\lambda_{n,j}^{(i)}}).
\end{aligned}\end{equation}
\endgroup
By using \eqref{green_rep},  \eqref{drhon8pi}, and \eqref{eta_ets}, we have,
\begingroup
\abovedisplayskip=3pt
\belowdisplayskip=3pt
for $x\in M\setminus \cup_{j=1}^m U^{\sscp M}_{ \delta }(q_j)$,
\begin{equation}\begin{aligned}\label{result_green_rep}
&\tilde{u}_n^{(1)}(x)-\tilde{u}_n^{(2)}(x)-\int_M(\tilde{u}_n^{(1)}-\tilde{u}_n^{(2)})\mathrm{d}\mu\\
&=\rho_n\sum_{j=1}^m\int_{U^{\sscp M}_{\frac{{\delta}}{4}}(x_{n,j}^{(1)})}(G(y,x)-G(x_{n,j}^{(1)},x))h(y)
(e^{\tilde{u}_n^{(1)}}-e^{\tilde{u}_n^{(2)}})\mathrm{d}\mu(y)+O(\sum_{i=1}^2e^{-\lambda_{n,j}^{(i)}})\\
&=\sum_{j=1}^m\int_{B_{\frac{\delta}{4}}(\un{x}_{n,j}^{(1)})}O(1)\left(\sum_{i=1}^2\frac{|\un{y}-\un{x}_{n,j}^{(1)}|e^{\lambda_{n,j}^{(i)}}}
{(1+e^{\lambda_{n,j}^{(i)}}|\un{y}-\un{x}_{n,j,*}^{(i)}|^2)^2}\right)\mathrm{d}\un{y}
+O(\sum_{i=1}^2e^{-\lambda_{n,j}^{(i)}}) =O(\sum_{i=1}^2e^{-\frac{\lambda_{n,j}^{(i)}}{2}}).
\end{aligned}
\end{equation}
\endgroup
Therefore, we see from \eqref{lambda_exp}, \eqref{diff_x1_x_2}, and  \eqref{difflam} that,
\begin{equation}\begin{aligned}\label{dlambda_exp}
&\int_M(\tilde{u}_n^{(1)}-\tilde{u}_n^{(2)})\mathrm{d}\mu= -(\lambda_{n,j}^{(1)}-\lambda_{n,j}^{(2)}) +
O\left(\sum_{i=1}^2\lambda_{n,j}^{(i)} e^{-\lambda_{n,j}^{(i)}}\right),
\end{aligned}\end{equation}
which, together with \eqref{result_green_rep}, shows that,
\begin{equation}
\begin{aligned}\label{1result_green_rep}
&\tilde{u}_n^{(1)}(x)-\tilde{u}_n^{(2)}(x)
= -(\lambda_{n,j}^{(1)}-\lambda_{n,j}^{(2)}) +O\left( \sum_{i=1}^2e^{-\frac{\lambda_{n,j}^{(i)}}{2}}\right),
\end{aligned}
\end{equation}
for $x\in M\setminus\cup_{j=1}^m U^{\sscp M}_{{\delta}}(q_j)$. Clearly \eqref{1result_green_rep} and \eqref{diff_on_small ball}
prove $(ii)$, and so the proof of Lemma \ref{lem_diff_tilu} is completed.
\end{proof}

\bigskip

Let us define,
\begin{equation*}
\label{def_zeta}
\zeta_n(x)=\frac{\tilde{u}_n^{(1)}(x)-\tilde{u}_n^{(2)}(x)}{\|\tilde{u}_n^{(1)}-\tilde{u}_n^{(2)}\|_{L^{\infty}(M)}}, \;x\in M,
\end{equation*}
\begin{equation}
\label{fnx}
f_n^*(x)=\frac{\rho_nh(x)}{\|\tilde{u}_n^{(1)}-\tilde{u}_n^{(2)}\|_{L^{\infty}(M)}}\left(e^{\tilde{u}_n^{(1)}(x)}-e^{\tilde{u}_n^{(2)}(x)}\right),  \;x\in M, \ \ \textrm{and}
\end{equation}
\begin{equation*}
\label{cnx}
c_n(x)=\frac{e^{\tilde{u}_n^{(1)}(x)}-e^{\tilde{u}_n^{(2)}(x)}}{\tilde{u}_n^{(1)}(x)-\tilde{u}_n^{(2)}(x)}=
e^{\tilde{u}_n^{(1)}}(1+O(\|\tilde{u}_n^{(1)}-\tilde{u}_n^{(2)}\|_{L^\infty(M)})),  \;x\in M.
\end{equation*}

Clearly $\zeta_n$ satisfies,
\begin{equation}\begin{aligned}\label{eq_zeta}
 \Delta_{\sscp M} \zeta_n+f_n^*(x)
=\Delta_{\sscp M} \zeta_n+\rho_nh(x)c_n(x)\zeta_n(x)=0, \; \;x\in M.
\end{aligned}
\end{equation}

Finally, let us define $\widehat{\zeta_n}$ to be the local coordinate expression of $\zeta_n$ for $\un{x}\in B_\delta(\un{x}_{n,j}^{(1)})$ and,
\begin{equation*}\label{zetanj}
  \zeta_{n,j}(z)=\widehat{\zeta_n}\left(e^{-\frac{\lambda_{n,j}^{(1)}}{2}}z+\un{x}_{n,j}^{(1)}\right),\;
  |z|<\delta e^{\frac{\lambda_{n,j}^{(1)}}{2}}\ \ \textrm{for}\ \ 1\le j\le m.
\end{equation*}

Our next result is about the limit of $\zeta_{n,j}$.

\begin{lemma}
  \label{lem_est_zetanj} There exists constants $b_{j,0}$, $b_{j,1}$, and $b_{j,2}$ such that,
\begin{equation*}
  \zeta_{n,j}(z)\to b_{j,0}\psi_{j,0}(z)+b_{j,1}\psi_{j,1}(z)+b_{j,2}\psi_{j,2}(z)\quad\textrm{in}\ \ C^0_{{\mathrm{loc}}}(\mathbb{R}^2),
\end{equation*}where
\begin{equation*}
  \psi_{j,0}(z)=\frac{1-\pi m h({\un{q}_j})|z|^2}{1+\pi m h({\un{q}_j})|z|^2},\    \psi_{j,1}(z)=\frac{\sqrt{\pi mh({\un{q}_j})}z_1}
  {1+\pi m h({\un{q}_j})|z|^2},\
  \psi_{j,2}(z)=\frac{\sqrt{\pi mh({\un{q}_j})}z_2}{1+\pi m h({\un{q}_j})|z|^2}.
\end{equation*}
\end{lemma}
\begin{proof}
By Lemma \ref{lem_diff_tilu} and  \eqref{eta}, we have, for $\un{x}\in B_{\delta}(\un{x}_{n,j}^{(1)})$,
\begin{equation}\label{est_cn}\begin{aligned}
c_n(\un{x})&=e^{\tilde{u}_n^{(1)}(\un{x})}\left(1+O(\sum_{i=1}^2\frac{1}{\lambda_{n,j}^{(i)}})\right)
 =e^{U_{n,j}^{(1)}(\un{x})+\eta_{n,j}^{(1)}(\un{x})+G_j^*(\un{x})-G_j^*(\un{x}_{n,j}^{(1)})}\left(1+O(\sum_{i=1}^2\frac{1}{\lambda_{n,j}^{(i)}})\right).\end{aligned}
\end{equation}
Then, in view of \eqref{eta_ets} and \eqref{diff_x_nj}, we see that \eqref{est_cn} implies that,
\begin{equation}
\begin{aligned}
\label{cnx_exp}
 e^{-\lambda_{n,j}^{(1)}} c_n (e^{-\frac{\lambda_{n,j}^{(1)}}{2}}z+\un{x}_{n,j}^{(1)} )
 &= \frac{e^{G_j^*(e^{-\frac{\lambda_{n,j}^{(1)}}{2}}z+\un{x}_{n,j}^{(1)}) -G_j^*(\un{x}_{n,j}^{(1)})}}
{(1+\frac{\rho_nh(\un{x}_{n,j}^{(1)})}{8}|z+O(  e^{-\frac{\lambda_{n,j}^{(1)}}{2}})|^2)^2}
\left(1+O(\sum_{i=1}^2\frac{1}{\lambda_{n,j}^{(i)}})\right)
&\to \frac{1}{(1+m\pi h(\un{q}_j)|z|^2)^2}\ \ \textrm{in}\ \ C^{0}_{\textrm{loc}}(\mathbb{R}^2).
\end{aligned}
\end{equation}
Since $|\zeta_{n,j}|\le 1$, and in view of \eqref{eq_zeta} and \eqref{cnx_exp}, we also find,
$\zeta_{n,j}(z)\to \zeta_j(z)$ in $C^0_{\textrm{loc}}(\mathbb{R}^2)$, where,
\begin{equation*}
\label{eq_zetaj}
\Delta \zeta_j+\frac{8\pi m h(\un{q}_j)}{(1+\pi m h(\un{q}_j)|z|^2)^2}\zeta_j(z)=0\ \ \textrm{in}\ \ \mathbb{R}^2, \ \ |\zeta_j|\le 1.
\end{equation*}
By  \cite[Proposition 1]{bp}, $\zeta_j=b_{j,0}\psi_{j,0} +b_{j,1}\psi_{j,1} +b_{j,2}\psi_{j,2}$ which proves Lemma \ref{lem_est_zetanj}.
\end{proof}

Next, let us set,
\begin{equation*}
\label{hj}
h_j(\un{x})=h(\un{x})e^{2\varphi_j(\un{x})},\;\un{x}\in B_\delta(\un{x}_{n,j}^{(1)}).
\end{equation*}
For any subset $A\subseteq M$, we denote by,
\begin{equation*}
\label{1A}1_A(x)=\left \{
\begin{array}{l}
1\quad \textrm{if}\ \ x\in A, \\ 0\quad \textrm{if}\ \ x\notin A,
\end{array}
\right.
\end{equation*}
while, for any $r>0$, we also denote by,
\begin{equation}
\label{La}\left \{
\begin{array}{l}
\Lambda_{n,j,r}^{-}=re^{-\lambda_{n,j}^{(1)}/2}, \\ \Lambda_{n,j,r}^{+}=re^{\lambda_{n,j}^{(1)}/2}.
\end{array}
\right.
\end{equation}
Next we prove an estimate which will be needed in section \ref{sec_poho}.
\begin{lemma}
\label{lem_out}
\begin{equation}\label{exp_zeta}\begin{aligned}
 \zeta_n(x)-\int_M \zeta_n\mathrm{d}\mu &=\sum_{j=1}^m A_{n,j}G(x_{n,j}^{(1)},x)   +
\sum_{j=1}^me^{-\frac{1}{2}\lambda_{n,j}^{(1)}}T_*^{-1}\left(\sum_{h=1}^2\partial_{\un{y}_h} G(\un{y},\un{x})\Big|_{\un{y}={\un{x}}_{n,j}^{(1)}}\right) 
\frac{b_{j,h}4\sqrt{8}}{\sqrt{\rho_n h(q_j)}}\int_{\mathbb{R}^2}\frac{|z|^2}{(1+|z|^2)^3}\mathrm{d}z\\
& +o(e^{-\frac{1}{2}\lambda_{n,1}^{(1)}})\ \textrm{in}\ C^1(M\setminus\cup_{j=1}^m U^{\sscp M}_{\theta}(x_{n,j}^{(1)})),
\end{aligned}
\end{equation}
where $\theta>0$ is a suitable small constant, $\partial_{\un{y}_h} G(\un{y},\un{x})=\frac{\partial G(\un{y},\un{x})}{\partial \un{y}_h}$,
${\un{y}=(\un{y}_1,\un{y}_2)}$, and,
\begin{equation*}\label{A_nj}
A_{n,j}=\int_{M_j}f_n^*(y)\mathrm{d}\mu(y).
\end{equation*}
Moreover, {there is a constant $C>0$, which do not depend by $R>0$, which satisfies,}
\begin{equation}
\label{bdd_of_Zeta}
\begin{aligned}
&\left|\zeta_n(x)-\int_M\zeta_n\mathrm{d}\mu-\sum_{j=1}^m A_{n,j}G(x_{n,j}^{(1)},x)\right| \le {C}\sum_{j=1}^m e^{-\frac{\lambda_{n,j}^{(1)}}{2}}\left(\frac{1_{U_{2r_0}(x_{n,j}^{(1)})}(x)}
{(|T_j(x)-T_j(x_{n,j}^{(1)})|)} +1_{M\setminus U_{2r_0}(x_{n,j}^{(1)})}(x)\right),
\end{aligned}\end{equation}
 {for} $
x\in M\setminus\cup_{j=1}^m U_{\Lambda_{n,j,R}^{-}}(x_{n,j}^{(1)})$, where $r_0$ is fixed as in {\eqref{190117}.}
\end{lemma}
\begin{proof}
By the Green representation formula we find that,
\begin{equation}
\label{green_zeta}
\begin{aligned}
\zeta_n(x)-\int_M \zeta_n\mathrm{d}\mu&=\int_MG(y,x)f_n^*(y)\mathrm{d}\mu(y)\\&=\sum_{j=1}^m A_{n,j}G(x_{n,j}^{(1)},x)
+\sum_{j=1}^m\int_{M_j}(G(y,x)-G(x_{n,j}^{(1)},x))f_n^*(y)\mathrm{d}\mu(y).
\end{aligned}
\end{equation}
For $x\in M\setminus \cup_{j=1}^m U^{\sscp M}_{\theta}(x_{n,j}^{(1)})$, let $\un{x}=T(x)$ denote any suitable local isothermal coordinate system. Then we see from \eqref{eta_ets}, \eqref{info_lambda}, and \eqref{diff_x1_x_2} that,
\begin{equation}\label{on_Mj}\begin{aligned}
 \int_{M_j}(G(y,x)-G(x_{n,j}^{(1)},x))f_n^*(y)\mathrm{d}\mu(y)
 &=\int_{B_{r}(\un{x}_{n,j}^{(1)})}<\partial_{\un{y}} G(\un{y},\un{x})\Big|_{\un{y}=\un{x}_{n,j}^{(1)}},\un{y}-\un{x}_{n,j}^{(1)}>
f_n^*(\un{y})e^{2\varphi_j}\mathrm{d}\un{y}
\\&\quad+O(1)\left(\int_{B_{r}(\un{x}_{n,j}^{(1)})}\frac{|{\un{y}} -\un{x}_{n,j}^{(1)}|^2e^{\lambda_{n,j}^{(1)}}}
{(1+e^{\lambda_{n,j}^{(1)}}|\un{y}-\un{x}_{n,j,*}^{(1)}|^2)^2}
\mathrm{d}\un{y}\right)+O(e^{-\lambda_{n,j}^{(1)}})
\\&= \int_{B_{r}(\un{x}_{n,j}^{(1)})}<\partial_{\un{y}}  G(\un{y},\un{x})\Big|_{\un{y}=
\un{x}_{n,j}^{(1)}},\un{y}-\un{x}_{n,j}^{(1)}>f_n^*(\un{y})e^{2\varphi_j}\mathrm{d}\un{y}
+O(\lambda_{n,j}^{(1)} e^{-\lambda_{n,j}^{(1)}}),
\end{aligned}
\end{equation}
for a suitable $r>0$. Next, by Lemma \ref{lem_diff_tilu} we find that,
\begin{equation}
\label{dv}\frac{e^{\tilde{u}_n^{(1)}}-e^{\tilde{u}_n^{(2)}}}
{\|\tilde{u}_n^{(1)}-\tilde{u}_n^{(2)}\|_{L^\infty(M)}}=e^{\tilde{u}_n^{(1)}}\zeta_n\left(1+O\left(\frac{1}
{\lambda_{n,j}^{(1)}}\right)\right).
\end{equation}
At this point, setting $\delta_n=e^{-\frac{\lambda_{n,j}^{(1)}}{2}}$, and using \eqref{eta_ets}, \eqref{diff_x1_x_2}, then after scaling we see that, for $x\in M\setminus \cup_{j=1}^m U^{\sscp M}_{\theta}(x_{n,j}^{(1)})$, it holds,
{\allowdisplaybreaks
\begin{align*}
&\int_{B_{r}(\un{x}_{n,j}^{(1)})}<\partial_{\un{y}} G(\un{y},\un{x})\Big|_{\un{y}=\un{x}_{n,j}^{(1)}},\un{y}-\un{x}_{n,j}^{(1)}>
f_n^*(\un{y})e^{2\varphi_j}\mathrm{d}\un{y}\\
&=\delta_n^3\int_{B_{\Lambda_{n,j,r}^{+}}(0)}<\partial_{\un{y}} G(\un{y},\un{x})\Big|_{\un{y}=\un{x}_{n,j}^{(1)}},z>
\rho_nh_j(\delta_n z+\un{x}_{n,j}^{(1)})
  e^{U_{n,j}^{(1)}(\delta_n z+\un{x}_{n,j}^{(1)})+G_j^*(\delta_n z+
\un{x}_{n,j}^{(1)})-G_j^*(\un{x}_{n,j}^{(1)})}\zeta_{n,j}   (1+O ( (\lambda_{n,j}^{(1)})^{-1})) \mathrm{d} z
\\&=\delta_n\int_{B_{\Lambda_{n,j,r}^{+}}(0)}\frac{<\partial_{\un{y}} G(\un{y},\un{x})\Big|_{\un{y}=\un{x}_{n,j}^{(1)}},z>\rho_nh(\un{x}_{n,j}^{(1)})
\zeta_{n,j}(z)}{(1+\frac{\rho_nh(\un{x}_{n,j}^{(1)})}{8}|z+O( \delta_n)|^2)^2}
\mathrm{d} z+o(\delta_n).
\end{align*}}

In view of Lemma \ref{lem_est_zetanj}, we see that, for  $x\in M\setminus \cup_{j=1}^m U^{\sscp M}_{\theta}(x_{n,j}^{(1)})$,  $\un{x}=T(x)$, it holds,
\begin{equation}
\label{on_Mj2}
\begin{aligned}
&\int_{B_{{r}}(\un{x}_{n,j}^{(1)})}<\partial_{\un{y}} G(\un{y},\un{x})\Big|_{\un{y}=\un{x}_{n,j}^{(1)}},\un{y}-\un{x}_{n,j}^{(1)}>
 f_n^*({\un{x}})e^{2\varphi_j}
 \mathrm{d}\un{y}
  \\&=e^{-\frac{\lambda_{n,j}^{(1)}}{2}}\sum_{h=1}^2\partial_{\un{y}_h} G(\un{y},\un{x})\Big|_{\un{y}=\un{x}_{n,j}^{(1)}} b_{j,h}
  \frac{4\sqrt{8}}{\sqrt{\rho_n h(\un{x}_{n,j}^{(1)})}}\int_{\mathbb{R}^2}\frac{|z|^2}{(1+|z|^2)^3}  \mathrm{d} z+
  o(e^{-\frac{\lambda_{n,j}^{(1)}}{2}}).
 \end{aligned}
 \end{equation}
From \eqref{green_zeta}-\eqref{on_Mj2}, we see that  the estimate \eqref{exp_zeta} holds in $C^0(M\setminus \cup_{j=1}^m U^{\sscp M}_{\theta}(x_{n,j}^{(1)}))$. The proof of the fact that  \eqref{exp_zeta} holds in $C^1(M\setminus \cup_{j=1}^m U^{\sscp M}_{\theta}(x_{n,j}^{(1)}))$ is similar and we skip it here to avoid repetitions.

From \eqref{dv}, \eqref{info_lambda}, \eqref{eta_ets}, and a suitable scaling, we see that there exist {$C>0$, independent of $R>0$,} such that, for $\un{x}\in B_{2 r_0}(\un{x}_{n,j}^{(1)})\setminus  B_{\Lambda_{n,j,R}^{-}}(\un{x}_{n,j}^{(1)})$, $x=T^{-1}_j(\un{x})$, it holds,
{\allowdisplaybreaks
\begin{align}
\label{bdd_zetan_ab}
&\Big|\zeta_n(x)-\int_M\zeta_n\mathrm{d}\mu-\sum_{j=1}^m A_{n,j}G(x_{n,j}^{(1)},x)\Big|
\nonumber\\&\le\Big|\sum_{j=1}^m\int_{{U_{3r_0}(x_{n,j}^{(1)})}}(G(y,x)-G(x_{n,j}^{(1)},x))f_n^*(y)\mathrm{d}\mu(y)\Big|+O(e^{-\lambda_{n,j}^{(1)}})
\nonumber\\&\leq \sum_{j=1}^m\Big|\frac{1}{2\pi}\int_{{B_{3r_0}(\un{x}_{n,j}^{(1)})}}\log\frac{|\un{x}-\un{x}_{n,j}^{(1)}|}{|\un{x}-\un{y}|}f_n^*(\un{y})
e^{2\varphi_j}\mathrm{d}\un{y} \Big| +O\Big(\int_{B_{3r_0}(\un{x}_{n,j}^{(1)})}\frac{|\un{y}-\un{x}_{n,j}^{(1)}|e^{\lambda_{n,j}^{(1)}}}{(1+e^{\lambda_{n,j}^{(1)}}
|\un{y}-\un{x}_{n,j,*}^{(1)}|^2)^2}\mathrm{d}\un{y}\Big)
+O(e^{-\lambda_{n,j}^{(1)}})
\nonumber\\&\le \sum_{j=1}^m{O(1)}\Big(\int_{{B_{\Lambda_{n,j,3r_0}^+}(0)}}\frac{\Big|\log |\un{x}-\un{x}_{n,j}^{(1)}|-
\log |\un{x}-e^{-\frac{\lambda_{n,j}^{(1)}}{2}}z-\un{x}_{n,j}^{(1)} |\Big|}{(1+|z|^2)^2}\mathrm{d} z \Big)+O(e^{-\frac{\lambda_{n,j}^{(1)}}{2}})
\nonumber\\&\le {O(1)\Big(\int_{\frac{|\un{x}-\un{x}_{n,j}^{(1)}|e^{ \frac{\lambda_{n,j}^{(1)}}{2}}}{2}\le |z|\le 2|\un{x}-\un{x}_{n,j}^{(1)}|
e^{ \frac{\lambda_{n,j}^{(1)}}{2}}}\frac{ \Big|\log (e^{ \frac{\lambda_{n,j}^{(1)}}{2}}|\un{x}-\un{x}_{n,j}^{(1)}|)|+ |\log |e^{ \frac{\lambda_{n,j}^{(1)}}{2}}(\un{x}-\un{x}_{n,j}^{(1)})-z|\Big|}{  (1+|z|^2)^2}\mathrm{d} z \Big)}
\nonumber\\&+\sum_{j=1}^m{O(1)}\Big(\int_{{B_{\Lambda_{n,j,3r_0}^+}(0)}}
\frac{ e^{-\frac{\lambda_{n,j}^{(1)}}{2}}|z|}{|\un{x}-\un{x}_{n,j}^{(1)}|(1+|z|^2)^2}\mathrm{d} z \Big)+O(e^{-\frac{\lambda_{n,j}^{(1)}}{2}})
\nonumber\\& {\le  O(1)\Big(\frac{e^{-\frac{\lambda_{n,j}^{(1)}}{2}}}{|\un{x}-\un{x}_{n,j}^{(1)}|}\Big)
+O(1)\Big( \log |z| |z|^{-2}\Big|_{|z|=  3| \un{x}-\un{x}_{n,j}^{(1)} |e^{\frac{\lambda_{n,j}^{(1)}}{2}}}\Big)
+O(e^{-\frac{\lambda_{n,j}^{(1)}}{2}})}
 {\le C}\Big(\frac{e^{-\frac{\lambda_{n,j}^{(1)}}{2}}}{|\un{x}-\un{x}_{n,j}^{(1)}|}\Big).
\end{align}}
By \eqref{green_zeta}, \eqref{dv}, \eqref{eta_ets}, and \eqref{info_lambda}, we also see that,
for $x\in M\setminus \cup_{j=1}^mU^{\sscp M}_{2r_0}(x_{n,j}^{(1)})$, it holds,
\begin{equation}\label{bdd_zetan_ab2}\begin{aligned}
   &\left|\zeta_n(x)-\int_M\zeta_n\mathrm{d}\mu-\sum_{j=1}^m A_{n,j}G(x_{n,j}^{(1)},x)\right|
 =O\left(\sum_{j=1}^m\int_{B_{r_0}(\un{x}_{n,j}^{(1)})}\frac{|\un{y}-\un{x}_{n,j}^{(1)}|e^{\lambda_{n,j}^{(1)}}}
   {(1+e^{\lambda_{n,j}^{(1)}}|\un{y}-\un{x}_{n,j,*}^{(1)}|^2)^2}\mathrm{d}\un{y}\right)+O(e^{-\lambda_{n,1}^{(1)}})
  =O ( e^{-\frac{\lambda_{n,1}^{(1)}}{2}}). \end{aligned}
 \end{equation}
By \eqref{bdd_zetan_ab} and \eqref{bdd_zetan_ab2} we obtain \eqref{bdd_of_Zeta}, which concludes the proof of Lemma \ref{lem_out}.
\end{proof}
From now on, to simplify the notations, we will set\[\overline{f}(z)=f(e^{-\frac{\lambda_{n,j}^{(1)}}{2}}z+\un{x}_{n,j}^{(1)}),\ \
 |z|< \delta e^{\frac{\lambda_{n,j}^{(1)}}{2}}\ \ \textrm{ for any function} \ \   f:B_{\delta}(\un{x}_{n,j}^{(1)})\to \mathbb{R}.\]
Our next aim is to obtain a detailed description of the asymptotic behavior of $\zeta_n$ on $U^{\sscp M}_{2c}(q_j)$ and on $M\setminus \cup_{j=1}^m U^{\sscp M}_c(q_j)$ for a suitable small $c>0$. This task has been already worked out in \cite{ly2} for the Chern-Simons-Higgs equation and we will
follow that approach here. However, as mentioned in the introduction, our case is in some respect more involved, since
if  $|\lambda_{n,j}^{(1)}-\lambda_{n,j}^{(2)}|$ is not asymptotically small enough, then the argument in \cite{ly2} does not work.
To overcome this difficulty, we have to use the Green representation formula and carry out a rather delicate set of estimates.

\begin{lemma}\label{lem_equalb0}
There is a constant $b_0$, such that $b_{j,0}=b_0$ for $j=1,\cdots, m$. Moreover,  for any $c>0$ small enough, we have,
\begin{equation*}
  \zeta_n(x)=-b_0+o(1)\ \ \textrm{for any}\ \ x\in M\setminus \cup_{j=1}^m U^{\sscp M}_c(q_j).
\end{equation*}\end{lemma}
\begin{proof}
Let us recall that,
\[{\Delta_M} \zeta_n+\rho_nhc_n\zeta_n=
{\Delta_M} \zeta_n+\frac{\rho_nh(x)}{\|\tilde{u}_n^{(1)}-\tilde{u}_n^{(2)}\|_{L^{\infty}(M)}}
\left(e^{\tilde{u}_n^{(1)}(x)}-e^{\tilde{u}_n^{(2)}(x)}\right)=0\ \ \textrm{  in}\ \  M.
\]
By \eqref{info_lambda} and Lemma \ref{lem_diff_tilu}, we have  $c_n\to 0$ in $C_{\textrm{loc}}(M\setminus\{q_1\cdots q_m\})$.\\
Since  $\|\zeta_n\|_{L^\infty(M)}\le 1$, we see that  $\zeta_n\to\zeta_0$ in $C_{\textrm{loc}}(M\setminus\{q_1,\cdots,q_m\})$, where,
\begin{equation}\label{eq_zetan}
  {\Delta_M} \zeta_0=0\ \ \textrm{ in}\ \ \   M\setminus\{q_1\cdots q_m\}.
\end{equation}
Moreover, since $\|\zeta_n\|_{L^\infty(M)}\le 1$, then we have $\|\zeta_0\|_{L^\infty(M)}\le 1$. Therefore $\zeta_0  $ is smooth near $q_i$, $i=1,\cdots, m$,
and we can extend \eqref{eq_zetan} to $M$. Then  $\zeta_0\equiv -b_0$ in $M$, where $b_0$ is a constant and in particular we find,
\begin{equation}\label{limit_zetan0}
  \zeta_n\to -b_0\ \ \ \textrm{in}\ \ C_{\textrm{loc}}(M\setminus\{q_1\cdots q_m\}).
\end{equation}

At this point, we consider the following two cases separately:

\textit{Case 1.} $|\lambda_{n,j}^{(1)}-\lambda_{n,j}^{(2)}|\le o\left(\frac{1}{\lambda_{n,j}^{(1)}}\right)$.\\
In this situation, we can follow the argument adopted in \cite{ly2}. We sketch the proof here for readers convenience.\\
Let $\un{x}=T_j(x)\in B_{\delta}(\un{x}_{n,j})$,
$\psi_{n,j}(\un{x})=\frac{1-\frac{\rho_n}{8}h(\un{x}_{n,j}^{(1)})|\un{x}-\un{x}_{n,j}^{(1)}|^2e^{\lambda_{n,j}^{(1)}}}
{1+\frac{\rho_n}{8}h(\un{x}_{n,j}^{(1)})|\un{x}-\un{x}_{n,j}^{(1)}|^2e^{\lambda_{n,j}^{(1)}}}$ and let us fix $d\in (0,\delta)$. Then,
in view of {\eqref{diff_x_nj}}, we find,
\begin{equation*}\label{withpsi}\begin{aligned}
&\int_{\partial B_d(\un{x}_{n,j}^{(1)})}\left(\psi_{n,j}\frac{\partial\zeta_n}{\partial\nu}-\zeta_n\frac{\partial \psi_{n,j}}{\partial \nu}\right)\mathrm{d}\sigma=
\int_{B_d(\un{x}_{n,j}^{(1)})}\left(\psi_{n,j}\Delta\zeta_n-\zeta_n\Delta \psi_{n,j}\right)\mathrm{d} \un{x}
\\&=
\int_{B_d(\un{x}_{n,j}^{(1)})}\Bigg\{-\rho_n\zeta_n\psi_{n,j}h_j\left(\frac{e^{\tilde{u}_n^{(1)}}-e^{\tilde{u}_n^{(2)}}}{\tilde{u}_n^{(1)}-\tilde{u}_n^{(2)}}
\right)
+\rho_n\zeta_n\psi_{n,j}h_j(\un{x}_{n,j}^{(1)})e^{U_{n,j}^{(1)}}\Big(\frac{1+\frac{\rho_n}{8}h(\un{x}_{n,j}^{(1)})|\un{x}-\un{x}_{n,j,*}^{(1)}|^2
e^{\lambda_{n,j}^{(1)}}}{1+\frac{\rho_n}{8}h(\un{x}_{n,j}^{(1)})|\un{x}-\un{x}_{n,j}^{(1)}|^2e^{\lambda_{n,j}^{(1)}}}\Big)^{2}\Bigg\}
\mathrm{d} \un{x}\\&=\int_{B_d(\un{x}_{n,j}^{(1)})}\rho_n\zeta_n\psi_{n,j}\Bigg\{-h_je^{\tilde{u}_n^{(1)}}
(1+O(|\tilde{u}_n^{(1)}-\tilde{u}_n^{(2)}|)) +h(\un{x}_{n,j}^{(1)})e^{U_{n,j}^{(1)}}(1+O({ e^{-\frac{\lambda_{n,j}^{(1)}}{2}}}))\Bigg\}
\mathrm{d} \un{x}
\\&=\int_{B_d(\un{x}_{n,j}^{(1)})}\rho_n\zeta_n\psi_{n,j}
\Bigg\{-h_je^{U_{n,j}^{(1)}+\eta_{n,j}^{(1)}+G_j^*(\un{x})-G_j^*(\un{x}_{n,j}^{(1)})}
(1+O(|\tilde{u}_n^{(1)}-\tilde{u}_n^{(2)}|)) +h(\un{x}_{n,j}^{(1)})e^{U_{n,j}^{(1)}}(1+
O({e^{-\frac{\lambda_{n,j}^{(1)}}{2}}}))\Bigg\}
\mathrm{d} \un{x}.
\end{aligned}
\end{equation*}
Therefore, by a suitable scaling and by using \eqref{eta_ets}, we see that,
\begin{equation*}\label{withpsi1}
\begin{aligned}
&\int_{\partial B_d(\un{x}_{n,j}^{(1)})}\left(\psi_{n,j}\frac{\partial\zeta_n}{\partial\nu}-\zeta_n\frac{\partial \psi_{n,j}}{\partial \nu}\right)\mathrm{d}\sigma
 =\int_{B_{\Lambda_{n,j,d}^{+}}(0)}\rho_n\overline{\zeta_n}(z)\,\overline{\psi_{n,j}}(z)
\frac{ O(1)(e^{-\frac{\lambda_{n,j}^{(1)}}{2}}|z|+|\overline{\tilde{u}_n^{(1)}}-\overline{\tilde{u}_n^{(2)}}|+
{e^{-\frac{\lambda_{n,j}^{(1)}}{2}}}) }{(1+\frac{\rho_nh(\un{x}_{n,j}^{(1)})}{8} |z+
e^{\frac{\lambda_{n,j}^{(1)}}{2}}(\un{x}_{n,j}^{(1)}-\un{x}_{n,j,*}^{(1)})|^2)^2}\mathrm{d} z.
\end{aligned}
\end{equation*}
In view of Lemma \ref{lem_diff_tilu}$(ii)$ and since {we are concerned with the case}
$|\lambda_{n,j}^{(1)}-\lambda_{n,j}^{(2)}|\le o\left(\frac{1}{\lambda_{n,j}^{(1)}}\right)$, then we obtain,
\begin{equation}\label{withpsi2}
\begin{aligned}
&\int_{\partial B_d(\un{x}_{n,j}^{(1)})}\left(\psi_{n,j}\frac{\partial\zeta_n}{\partial\nu}-\zeta_n\frac{\partial \psi_{n,j}}{\partial \nu}\right)
\mathrm{d}\sigma=o\left(\frac{1}{\lambda_{n,j}^{(1)}}\right).
\end{aligned}
\end{equation}
Let $\zeta_{n,j}^*(r)=\int_0^{2\pi}\zeta_n(r,\theta)\mathrm{d}\theta$, where $r=|x-x_{n,j}^{(1)}|$. Then \eqref{withpsi2} yields,
\begin{equation*}\label{radial_eq}
  (\zeta_{n,j}^*)'(r)\psi_{n,j}(r)-\zeta_{n,j}^*(r)\psi_{n,j}'(r)=\frac{o\left(\frac{1}{\lambda_{n,j}^{(1)}}\right)}{r},\,\forall r\in (\Lambda_{n,j,R}^{-},
  {\delta}].
\end{equation*}
For any $R>0$ large enough and for any $r\in (\Lambda_{n,j,R}^{-}, {\delta})$, we also obtain that,
\[\psi_{n,j}(r)=-1+O\left(\frac{e^{-\lambda_{n,j}^{(1)}}}{r^2}\right),\ \ \psi_{n,j}'(r)=
O\left(\frac{{e^{-\lambda_{n,j}^{(1)}}}}{r^3}\right).\]
and so we conclude that,
\begin{equation}\label{derivative_zeta}
  (\zeta_{n,j}^*)'(r)=\frac{o\Big(\frac{1}{ \lambda_{n,j}^{(1)}}\Big)}{r}+O\left(\frac{{e^{-\lambda_{n,j}^{(1)}}}}{r^3}\right)
  \ \ \textrm{for all } \ \ r\in (\Lambda_{n,j,R}^{-}, {\delta}).
\end{equation}
Integrating \eqref{derivative_zeta}, we obtain,
\begin{equation}\label{compare_zeta}\begin{aligned}
  \zeta_{n,j}^*(r) =\zeta_{n,j}^*(\Lambda_{n,j,R}^{-})+o(1)+o\Big(\frac{1}{ \lambda_{n,j}^{(1)}}\Big)R+O({R^{-2}})  \ \ \textrm{for all } \ \ r\in (\Lambda_{n,j,R}^{-}, {\delta}).
\end{aligned}\end{equation}
By using Lemma \ref{lem_est_zetanj}, we find,
\[\zeta_{n,j}^*(\Lambda_{n,j,R}^{-})=-2\pi b_{j,0}+o_R(1)+o_n(1),\]
where $\lim_{R\to+\infty}o_R(1)=0$ and $\lim_{n\to+\infty}o_n(1)=0$ and then \eqref{compare_zeta} shows that,
\begin{equation}\label{outside_zeta}
  \zeta_{n,j}^*(r)=-2\pi b_{j,0}+o_R(1)+o_n(1)(1+O(R)),\ \ \textrm{for all } \ \ r\in (\Lambda_{n,j,R}^{-},{\delta}),
\end{equation} where $\lim_{n\to+\infty}o_n(1)=0$.
In view of \eqref{limit_zetan0}, we see that,
\begin{equation*}\label{compare_b}\zeta_{n,j}^*=-2\pi b_0+o_n(1) \ \ \textrm{in}\ \ C_{\textrm{loc}}(M\setminus\{q_1\cdots q_m\}),
\end{equation*} which implies that $b_{j,0}=b_0$ for $j=1,\cdots, m$, whenever
$|\lambda_{n,j}^{(1)}-\lambda_{n,j}^{(2)}|=o\left(\frac{1}{\lambda_{n,j}^{(1)}}\right)$, as claimed.

\noindent \textit{Case 2.} $\frac{1}{C \lambda_{n,j}^{(1)}}{\le} |\lambda_{n,j}^{(1)}-\lambda_{n,j}^{(2)}|\le \frac{C}{\lambda_{n,j}^{(1)}}$ for some constant $C>1$: In this  case, the argument in \cite{ly2} as outlined above does not yield the desired result.
Indeed,  since  $|\lambda_{n,j}^{(1)}-\lambda_{n,j}^{(2)}|$ is "not small enough", then
$\zeta_{n,j}^*(r)-\zeta_{n,j}^*(\Lambda_{n,j,R}^{-})$ is not as small as we would need, see \eqref{compare_zeta}.
So we adopt a different approach based on the Green representation formula.\\{Fix $d\in(0,\delta)$, and let}
  $\Lambda_{n,j,R}^{-}\le |\un{x}_1-\un{x}_{n,j}^{(1)}|\le |\un{x}_2-\un{x}_{n,j}^{(1)}|\le d$, then,
\begin{equation}
\begin{aligned}\label{diff_Znx}
 \zeta_n(x_1)-\zeta_n(x_2)
 &=\rho_n\int_M (G(x_1,y)-G(x_2,y))h(y) \left(\frac{e^{\tilde{u}_n^{(1)}(y)}-e^{\tilde{u}_n^{(2)}(y)}}{
\|\tilde{u}_n^{(1)}-\tilde{u}_n^{(2)}\|_{L^{\infty}(M)}}\right)\mathrm{d}\mu(y)
\\&=\frac{1}{2\pi}\sum \limits_{j=1}^m\int_{B_{2d}(\un{x}_{n,j}^{(1)})}
\log \frac{|\un{x}_2-\un{y}|}{|\un{x}_1-\un{y}|}
\rho_n h_j(\un{y})e^{\tilde{u}_n^{(1)}(\un{y})} \left(\frac{ 1-e^{\tilde{u}_n^{(2)} -\tilde{u}_n^{(1)} }}
{\|\tilde{u}_n^{(1)}-\tilde{u}_n^{(2)}\|_{L^\infty(M)}} \right)\mathrm{d}\un{y} +O(|\un{x}_1-\un{x}_2|).
\end{aligned}
\end{equation}
By the usual scaling, $\un{y}=\delta_n z+\un{x}_{n,j}^{(1)}$, where $\delta_n=e^{-\frac{\lambda_{n,j}^{(1)}}{2}}$, we see that,
\begin{equation}
\begin{aligned}\label{diff_Znx3}
&\int_{B_{2d}(\un{x}_{n,j}^{(1)})}\log \frac{|\un{x}_2-\un{y}|}{|\un{x}_1-\un{y}|}
\rho_n h_j(\un{y})e^{\tilde{u}_n^{(1)}(\un{y})}\left(\frac{ 1-e^{\tilde{u}_n^{(2)} -\tilde{u}_n^{(1)} }}
{\|\tilde{u}_n^{(1)}-\tilde{u}_n^{(2)}\|_{L^\infty(M)}} \right)\mathrm{d}\un{y}
\\&= \int_{B_{2\Lambda_{n,j,d}^{+}}(0)}\log \frac{|\un{x}_2-\un{x}_{n,j}^{(1)}-\delta_nz|}
{|\un{x}_1-\un{x}_{n,j}^{(1)}-\delta_nz|}
\rho_n \overline{h}_j(z)e^{\overline{\tilde{u}_n^{(1)}}(z)} e^{-\lambda_{n,j}^{(1)}}
\left(\frac{ 1-e^{\overline{\tilde{u}_n^{(2)}} -\overline{\tilde{u}_n^{(1)}} }}
{\|\tilde{u}_n^{(1)}-\tilde{u}_n^{(2)}\|_{L^\infty(M)}} \right)\mathrm{d} z
\\&=\int_{B_{2\Lambda_{n,j,d}^{+}}(0)}\log \frac{|\delta_n^{-1}(\un{x}_2-\un{x}_{n,j}^{(1)})-z|}
{|\delta_n^{-1}(\un{x}_1-\un{x}_{n,j}^{(1)})-z|}
\rho_n \overline{h}_j(z)e^{\overline{\tilde{u}_n^{(1)}}(z)} e^{-\lambda_{n,j}^{(1)}}
 \left(\frac{ 1-e^{\overline{\tilde{u}_n^{(2)}} -\overline{\tilde{u}_n^{(1)}} }}{\|\tilde{u}_n^{(1)}-\tilde{u}_n^{(2)}\|_{L^\infty(M)}} \right)
\mathrm{d} z.
\end{aligned}
\end{equation}
Fix $\alpha\in(0,\frac{1}{2})$. We will use the following inequality  (see  \cite[Theorem 4.1]{cfl}): let $g:\mathbb{R}^2\to \mathbb{R}$ satisfies $\int\limits_{\mathbb{R}^2}g^2(1+|z|)^{2+\alpha}dz<+\infty$. Then
there exists a constant $c >0$, independent of $\un{x}\in\mathbb{R}^2\setminus B_2(0)$ and $g$, such that,
\begin{equation}\label{refer_cfl}\begin{aligned}&\left|\int_{\mathbb{R}^2} (\log|\un{x}-z|-\log|\un{x}|)g(z)\mathrm{d}z\right| \le c  |\un{x}|^{-\frac{\alpha}{2}}(\log|\un{x}|+1)
\|g(z)(1+|z|)^{1+\frac{\alpha}{2}}\|_{L^2(\mathbb{R}^2)}.\end{aligned}\end{equation}
In view of \eqref{diff_Znx3} and \eqref{refer_cfl}, we find that,
\begingroup
\abovedisplayskip=3pt
\belowdisplayskip=3pt
\begin{equation}\begin{aligned}\label{diff_Znx2}
                   & \int_{B_{2d}(\un{x}_{n,j}^{(1)})}\log \frac{|\un{x}_2-\un{y}|}{|\un{x}_1-\un{y}|}
                   \rho_n h_j(\un{y})e^{\tilde{u}_n^{(1)}(\un{y})}\left(\frac{ 1-e^{ \tilde{u}_n^{(2)}  - \tilde{u}_n^{(1)}  }}
                   {\|\tilde{u}_n^{(1)}-\tilde{u}_n^{(2)}\|_{L^\infty(M)}} \right)\mathrm{d}\un{y} \\&=
                   \int_{B_{2\Lambda_{n,j,d}^{+}}(0)}\left(\log \frac{{\delta_n^{-1}}|\un{x}_2-\un{x}_{n,j}^{(1)}|}
                   {{\delta_n^{-1}}|\un{x}_1-\un{x}_{n,j}^{(1)}|}+
                   \sum_{i=1}^2(-1)^{{i}}\log \frac{|{\delta_n^{-1}}
                   (\un{x}_i-\un{x}_{n,j}^{(1)})-z|}{|{\delta_n^{-1}}(\un{x}_i-\un{x}_{n,j}^{(1)})|}\right)
                 \rho_n \overline{h}_j(z)e^{\overline{\tilde{u}_n^{(1)}}(z)} e^{-\lambda_{n,j}^{(1)}}
                   \left(\frac{ 1-e^{\overline{\tilde{u}_n^{(2)}} -
                   \overline{\tilde{u}_n^{(1)}} }}{\|\tilde{u}_n^{(1)}-\tilde{u}_n^{(2)}\|_{L^\infty(M)}} \right)\mathrm{d} z  \\&=
                   \log \frac{|\un{x}_2-\un{x}_{n,j}^{(1)}|}{|\un{x}_1-\un{x}_{n,j}^{(1)}|} \int_{B_{2\Lambda_{n,j,d}^{+}}(0)}
                   \rho_n \overline{h}_j(z)e^{\overline{\tilde{u}_n^{(1)}}(z)} e^{-\lambda_{n,j}^{(1)}}
                   \left(\frac{ 1-e^{\overline{\tilde{u}_n^{(2)}} -\overline{\tilde{u}_n^{(1)}} }}{\|\tilde{u}_n^{(1)}-\tilde{u}_n^{(2)}\|_{L^\infty(M)}} \right)
                   \mathrm{d} z
                   \\&\quad+O(\sum_{i=1}^2|e^{\frac{\lambda_{n,j}^{(1)}}{2}}(\un{x}_i-\un{x}_{n,j}^{(1)})|^{-\frac{\alpha}{2}}
                   \log |e^{\frac{\lambda_{n,j}^{(1)}}{2}}(\un{x}_i-\un{x}_{n,j}^{(1)})|). \end{aligned}
\end{equation}
\endgroup

From \eqref{diff_x_nj}-\eqref{eta_ets}, {we} also see that,
\begingroup
\abovedisplayskip=3pt
\belowdisplayskip=3pt
\begin{equation}
\begin{aligned}\label{integration_est1}
&\int_{B_{2\Lambda_{n,j,d}^{+}}(0)}\rho_n \overline{h}_j(z)e^{\overline{\tilde{u}_n^{(1)}}(z)}
e^{-\lambda_{n,j}^{(1)}}\left(\frac{ 1-e^{\overline{\tilde{u}_n^{(2)}} -\overline{\tilde{u}_n^{(1)}} }}
{\|\tilde{u}_n^{(1)}-\tilde{u}_n^{(2)}\|_{L^\infty(M)}} \right)\mathrm{d} z
\\&=\int_{B_{2\Lambda_{n,j,d}^{+}}(0)}\frac{\rho_n h_j(\delta_n z+\un{x}_{n,j}^{(1)})
e^{\overline{\eta_{n,j}^{(1)}}+G_j^*(\delta_n z+\un{x}_{n,j}^{(1)})-G_j^*(\un{x}_{n,j}^{(1)})}  }
{(1+\frac{\rho_nh(\un{x}_{n,j}^{(1)})}{8} |z+e^{\frac{\lambda_{n,j}^{(1)}}{2}}(\un{x}_{n,j}^{(1)}-\un{x}_{n,j,*}^{(1)})|^2)^2}
\left(\frac{ 1-e^{\overline{\tilde{u}_n^{(2)}} -\overline{\tilde{u}_n^{(1)}} }}{\|\tilde{u}_n^{(1)}-\tilde{u}_n^{(2)}\|_{L^\infty(M)}} \right)\mathrm{d} z
\\&=\int_{B_{2\Lambda_{n,j,d}^{+}}(0)}\frac{\rho_n h(\un{x}_{n,j}^{(1)})(1+O( \delta_n)+
O(\delta_n|z|)) }
{(1+\frac{\rho_nh(\un{x}_{n,j}^{(1)})}{8} |z|^2)^2}
\left(\frac{ 1-e^{\overline{\tilde{u}_n^{(2)}} -\overline{\tilde{u}_n^{(1)}} }}
{\|\tilde{u}_n^{(1)}-\tilde{u}_n^{(2)}\|_{L^\infty(M)}} \right)\mathrm{d} z
\\&=\int_{B_{2\Lambda_{n,j,d}^{+}}(0)}\rho_n h(\un{x}_{n,j}^{(1)})
\frac{\left(  \frac{ \overline{\tilde{u}_n^{(1)}} -\overline{\tilde{u}_n^{(2)}}-\frac{(\overline{\tilde{u}_n^{(1)}} -\overline{\tilde{u}_n^{(2)}})^2}{2} }
{\|\tilde{u}_n^{(1)}-\tilde{u}_n^{(2)}\|_{L^\infty(M)}}+O(\|\tilde{u}_n^{(1)}-\tilde{u}_n^{(2)}\|_{L^\infty(M)}^2)\right)}
{(1+\frac{\rho_nh(\un{x}_{n,j}^{(1)})}{8} |z|^2)^2} \mathrm{d} z +O( \delta_n).
\end{aligned}
\end{equation}
\endgroup
On the other side, from \eqref{eta} {and \eqref{eta_ets},} we find that,
\begin{equation*}\label{difftildun_lamnj}\begin{aligned}
 \tilde{u}_n^{(1)}(x)-\tilde{u}_n^{(2)}(x) &=U_{n,j}^{(1)}(\un{x})-U_{n,j}^{(2)}(\un{x})+G_j^*(\un{x}_{n,j}^{(2)})-
G_j^*(\un{x}_{n,j}^{(1)})+\eta_{n,j}^{(1)}(\un{x})-\eta_{n,j}^{(2)}(\un{x})
\\&=\lambda_{n,j}^{(1)}-\lambda_{n,j}^{(2)}+2\log\frac{(1+\frac{\rho_nh(\un{x}_{n,j}^{(2)})}{8}e^{\lambda_{n,j}^{(2)}}|\un{x}-\un{x}_{n,j,*}^{(2)}|^2)}
{(1+\frac{\rho_nh(\un{x}_{n,j}^{(1)})}{8}e^{\lambda_{n,j}^{(1)}}|\un{x}-\un{x}_{n,j,*}^{(1)}|^2)}
 +O\left((\lambda_{n,j}^{(1)})^2e^{-\lambda_{n,j}^{(1)}}\right),                                      \end{aligned}
\end{equation*}
which in turn implies that,
\begin{equation}\label{difftildun_lamnj_sc}\begin{aligned}
  \overline{\tilde{u}_n^{(1)}}(z) -\overline{\tilde{u}_n^{(2)}}(z)
 &=\lambda_{n,j}^{(1)}-\lambda_{n,j}^{(2)}+2\log\Big(\frac{ 1+\frac{\rho_nh(\un{x}_{n,j}^{(2)})}{8}e^{\lambda_{n,j}^{(2)}}
|\delta_n z+\un{x}_{n,j}^{(1)}-\un{x}_{n,j,*}^{(2)}|^2 }
{ 1+\frac{\rho_nh(\un{x}_{n,j}^{(1)})}{8}e^{\lambda_{n,j}^{(1)}}
|\delta_n z+\un{x}_{n,j}^{(1)}-\un{x}_{n,j,*}^{(1)}|^2 }\Big)
 +O\Big((\lambda_{n,j}^{(1)})^2e^{-\lambda_{n,j}^{(1)}}\Big).                                     \end{aligned}
\end{equation}
By \eqref{diff_x_nj} and \eqref{diff_x1_x_2}, we also see that,
{\allowdisplaybreaks\begin{align*}
& 2\log\Big(\frac{ 1+\frac{\rho_nh(\un{x}_{n,j}^{(2)})}{8}e^{\lambda_{n,j}^{(2)}}|
\delta_n z+\un{x}_{n,j}^{(1)}-\un{x}_{n,j,*}^{(2)}|^2 }{ 1+\frac{\rho_nh(\un{x}_{n,j}^{(1)})}{8}e^{\lambda_{n,j}^{(1)}}|
\delta_n z+\un{x}_{n,j}^{(1)}-\un{x}_{n,j,*}^{(1)}|^2 }\Big)
= 2\log\Big(\frac{ 1+\frac{\rho_nh(\un{x}_{n,j}^{(1)})}{8}e^{\lambda_{n,j}^{(2)}-\lambda_{n,j}^{(1)}} |z|^2 }
{ 1+\frac{\rho_nh(\un{x}_{n,j}^{(1)})}{8} |z|^2 }\Big)+O(\lambda_{n,j}^{(1)} e^{-\frac{\lambda_{n,j}^{(1)}}{2}})
\\&=   \Big( \frac{\frac{\rho_nh(\un{x}_{n,j}^{(1)})}{4} |z|^2  }{ 1+\frac{\rho_nh(\un{x}_{n,j}^{(1)})}{8} |z|^2 }\Big)(\lambda_{n,j}^{(2)}-\lambda_{n,j}^{(1)})
 +  \frac{\frac{\rho_nh(\un{x}_{n,j}^{(1)})}{8} |z|^2  }{ (1+\frac{\rho_nh(\un{x}_{n,j}^{(1)})}{8} |z|^2)^2 } (\lambda_{n,j}^{(1)}-\lambda_{n,j}^{(2)})^2
  +O(|\lambda_{n,j}^{(1)}-\lambda_{n,j}^{(2)}|^3)  +O(\lambda_{n,j}^{(1)} e^{-\frac{\lambda_{n,j}^{(1)}}{2}}),
\end{align*}}
which, together with \eqref{difftildun_lamnj_sc}, allows us to conclude that,
\begin{equation*}\label{difftildun_lamnj_sc2}\begin{aligned}
&\overline{\tilde{u}_n^{(1)}}(z) -\overline{\tilde{u}_n^{(2)}}(z)
\\&=\Big( \frac{ 1-\frac{\rho_nh(\un{x}_{n,j}^{(1)})}{8} |z|^2 }{ 1+\frac{\rho_nh(\un{x}_{n,j}^{(1)})}{8} |z|^2 }\Big)
(\lambda_{n,j}^{(1)}-\lambda_{n,j}^{(2)})
 +  \frac{\frac{\rho_nh(\un{x}_{n,j}^{(1)})}{8} |z|^2  }{ (1+\frac{\rho_nh(\un{x}_{n,j}^{(1)})}{8} |z|^2)^2 }
 (\lambda_{n,j}^{(1)}-\lambda_{n,j}^{(2)})^2
  +O(|\lambda_{n,j}^{(1)}-\lambda_{n,j}^{(2)}|^3)  +O(\lambda_{n,j}^{(1)} e^{-\frac{\lambda_{n,j}^{(1)}}{2}}),                                     \end{aligned}
\end{equation*}
and thus,
\begingroup
\abovedisplayskip=3pt
\belowdisplayskip=3pt
\begin{equation}\label{difftildun_lamnj_sc3}
\begin{aligned}&  \overline{\tilde{u}_n^{(1)}}(z) -\overline{\tilde{u}_n^{(2)}}(z)-\frac{(\overline{\tilde{u}_n^{(1)}} -\overline{\tilde{u}_n^{(2)}})^2}{2}
\\& =\Big( \frac{ 1-\frac{\rho_nh(\un{x}_{n,j}^{(1)})}{8} |z|^2 }{ 1+\frac{\rho_nh(\un{x}_{n,j}^{(1)})}{8} |z|^2 }\Big)
(\lambda_{n,j}^{(1)}-\lambda_{n,j}^{(2)})
\\&\quad +\Big(  \frac{\frac{\rho_nh(\un{x}_{n,j}^{(1)})}{8} |z|^2  }{ (1+\frac{\rho_nh(\un{x}_{n,j}^{(1)})}{8} |z|^2)^2 } -
\frac{ (1-\frac{\rho_nh(\un{x}_{n,j}^{(1)})}{8} |z|^2 )^2}{ 2(1+\frac{\rho_nh(\un{x}_{n,j}^{(1)})}{8} |z|^2)^2 }\Big)
(\lambda_{n,j}^{(1)}-\lambda_{n,j}^{(2)})^2
 +O(|\lambda_{n,j}^{(1)}-\lambda_{n,j}^{(2)}|^3)  +O(\lambda_{n,j}^{(1)} e^{-\frac{\lambda_{n,j}^{(1)}}{2}}).
\end{aligned}
\end{equation}
\endgroup
\begingroup
\abovedisplayskip=3pt
\belowdisplayskip=3pt
From \eqref{integration_est1}-\eqref{difftildun_lamnj_sc3}, we deduce that,
 \begin{equation}\begin{aligned}\label{integration_estimation}
&\int_{B_{2\Lambda_{n,j,d}^{+}}(0)}\rho_n \overline{h}_j(z)e^{\overline{\tilde{u}_n^{(1)}}(z)}
e^{-\lambda_{n,j}^{(1)}}\Big(\frac{ 1-e^{\overline{\tilde{u}_n^{(2)}} -\overline{\tilde{u}_n^{(1)}} }}
{\|\tilde{u}_n^{(1)}-\tilde{u}_n^{(2)}\|_{L^\infty(M)}} \Big)\mathrm{d} z
\\&=\int_{B_{2\Lambda_{n,j,d}^{+}}(0)}\frac{\rho_n h(\un{x}_{n,j}^{(1)})}{(1+\frac{\rho_nh(\un{x}_{n,j}^{(1)})}{8} |z|^2)^2}
  \Bigg\{\Big( \frac{ 1-\frac{\rho_nh(\un{x}_{n,j}^{(1)})}{8} |z|^2 }{ 1+\frac{\rho_nh(\un{x}_{n,j}^{(1)})}{8} |z|^2 }\Big)
  \frac{(\lambda_{n,j}^{(1)}-\lambda_{n,j}^{(2)})}{\|\tilde{u}_n^{(1)}-\tilde{u}_n^{(2)}\|_{L^\infty(M)}}
\\&\quad +\Big(  \frac{\frac{\rho_nh(\un{x}_{n,j}^{(1)})}{8} |z|^2  }{ (1+\frac{\rho_nh(\un{x}_{n,j}^{(1)})}{8} |z|^2)^2 } -\frac{ (1-\frac{\rho_nh(\un{x}_{n,j}^{(1)})}{8} |z|^2 )^2}{ 2(1+\frac{\rho_nh(\un{x}_{n,j}^{(1)})}{8} |z|^2)^2 }\Big)
\frac{(\lambda_{n,j}^{(1)}-\lambda_{n,j}^{(2)})^2}{\|\tilde{u}_n^{(1)}-\tilde{u}_n^{(2)}\|_{L^\infty(M)}} \Bigg\}\mathrm{d} z
\\&\quad +O\Big(\frac{|\lambda_{n,j}^{(1)}-\lambda_{n,j}^{(2)}|^3+\lambda_{n,j}^{(1)} e^{-\frac{\lambda_{n,j}^{(1)}}{2}}}{\|\tilde{u}_n^{(1)}-\tilde{u}_n^{(2)}\|_{L^\infty(M)}}\Big)+O( e^{-\frac{\lambda_{n,j}^{(1)}}{2}})+O(\|\tilde{u}_n^{(1)}-\tilde{u}_n^{(2)}\|_{L^\infty(M)}^2).
\end{aligned}
\end{equation}
\endgroup
At this point we note that, for any fixed $t>0$, it holds,
   \begin{equation}\begin{aligned}\label{entire_int}
&\int_{B_t(0)} \frac{8}{(1+  |z|^2)^2}
   \Big( \frac{ 1-  |z|^2 }{ 1+ |z|^2 }\Big) \mathrm{d} z
   =\frac{8\pi t^2}{(t^2+1)^2},\ \ \textrm{and}
\end{aligned}
\end{equation}   \begin{equation}\begin{aligned}\label{entire_int2}
&\int_{B_t(0)} \frac{8}{(1+  |z|^2)^2}
 \Big(  \frac{  |z|^2  }{ (1+  |z|^2)^2 } -  \frac{ (1-  |z|^2 )^2}{ 2(1+  |z|^2)^2 }\Big)\mathrm{d} z
=\frac{4\pi t^2(t^2-1)}{(t^2+1)^3}.
\end{aligned}
\end{equation}
Since $\Lambda_{n,j,d}^{+}=d e^{\frac{\lambda_{n,j}^{(1)}}{2}}$, then \eqref{integration_estimation}-\eqref{entire_int2} imply that,
\begin{equation*}\begin{aligned}\label{result1}
&\int_{B_{2\Lambda_{n,j,d}^{+}}(0)}\rho_n \overline{h}_j(z)e^{\overline{\tilde{u}_n^{(1)}}(z)} e^{-\lambda_{n,j}^{(1)}}\Big(\frac{ 1-e^{\overline{\tilde{u}_n^{(2)}} -\overline{\tilde{u}_n^{(1)}} }}{\|\tilde{u}_n^{(1)}-\tilde{u}_n^{(2)}\|_{L^\infty(M)}} \Big)\mathrm{d} z
\\&=O\Big(\frac{(\lambda_{n,j}^{(1)}-\lambda_{n,j}^{(2)})e^{-\lambda_{n,j}^{(1)}}}{\|\tilde{u}_n^{(1)}-\tilde{u}_n^{(2)}\|_{L^\infty(M)}} \Big)
  +O\Big(\frac{|\lambda_{n,j}^{(1)}-\lambda_{n,j}^{(2)}|^3+\lambda_{n,j}^{(1)} e^{-\frac{\lambda_{n,j}^{(1)}}{2}}}{\|\tilde{u}_n^{(1)}-\tilde{u}_n^{(2)}\|_{L^\infty(M)}}\Big) +O( e^{-\frac{\lambda_{n,j}^{(1)}}{2}})+O(\|\tilde{u}_n^{(1)}-\tilde{u}_n^{(2)}\|_{L^\infty(M)}^2).
\end{aligned}
\end{equation*}
At this point,   from Lemma \ref{lem_diff_tilu} and our assumption  $\frac{1}{C \lambda_{n,j}^{(1)}}{\le} |\lambda_{n,j}^{(1)}-\lambda_{n,j}^{(2)}|\le \frac{C}{\lambda_{n,j}^{(1)}}$,   we can find a constant $c_0>1$ such that,
\begin{equation}\label{diff_tilu}
 \frac{1}{C_0C \lambda_{n,j}^{(1)}}\le \frac{|\lambda_{n,j}^{(1)}-\lambda_{n,j}^{(2)}|}{c_0}\le\|\tilde{u}_n^{(1)}-\tilde{u}_n^{(2)}\|_{L^{\infty}(M)}\le c_0|\lambda_{n,j}^{(1)}-\lambda_{n,j}^{(2)}|\le  \frac{C_0C}{ \lambda_{n,j}^{(1)}}.
\end{equation}
By  \eqref{diff_tilu},
we obtain,
\begin{equation}\begin{aligned}\label{result2}
&\int_{B_{2\Lambda_{n,j,d}^{+}}(0)}\rho_n \overline{h}_j(z)e^{\overline{\tilde{u}_n^{(1)}}(z)} e^{-\lambda_{n,j}^{(1)}}\Big(\frac{ 1-e^{\overline{\tilde{u}_n^{(2)}} -\overline{\tilde{u}_n^{(1)}} }}{\|\tilde{u}_n^{(1)}-\tilde{u}_n^{(2)}\|_{L^\infty(M)}} \Big)\mathrm{d} z
 =o\Big(\frac{1}{\lambda_{n,j}^{(1)}}\Big).
\end{aligned}
\end{equation}
As a consequence, for $\Lambda_{n,j,R}^{-}\le |\un{x}_1-\un{x}_{n,j}^{(1)}|\le |\un{x}_2-\un{x}_{n,j}^{(1)}|\le d$,
and by using \eqref{diff_Znx}-\eqref{diff_Znx2} and \eqref{result2}, we find that,
\begingroup
\abovedisplayskip=3pt
\belowdisplayskip=3pt
\begin{equation}\label{result3}\begin{aligned}
                 \zeta_n(\un{x}_1)-\zeta_n(\un{x}_2)
                                     &= \log \frac{|\un{x}_2-\un{x}_{n,j}^{(1)}|}{|\un{x}_1-\un{x}_{n,j}^{(1)}|}
                                       \frac{1}{2\pi} \int_{B_{2\Lambda_{n,j,d}^{+}}(0)}
                   \rho_n \overline{h}_j(z)e^{\overline{\tilde{u}_n^{(1)}}(z)} e^{-\lambda_{n,j}^{(1)}}\left(\frac{ 1-e^{\overline{\tilde{u}_n^{(2)}}
                   -\overline{\tilde{u}_n^{(1)}} }}{\|\tilde{u}_n^{(1)}-\tilde{u}_n^{(2)}\|_{L^\infty(M)}} \right)\mathrm{d} z
                   \\&+O(|\un{x}_1-\un{x}_2|)  +O(\sum_{i=1}^2|e^{\frac{\lambda_{n,j}^{(1)}}{2}}(\un{x}_i-\un{x}_{n,j}^{(1)})|^{-\frac{\alpha}{2}}\log
                   |e^{\frac{\lambda_{n,j}^{(1)}}{2}}(\un{x}_i-\un{x}_{n,j}^{(1)})|)\\&= O(|\un{x}_1-\un{x}_2|)+o(1)+O(R^{-\frac{\alpha}{2}}\log R). \end{aligned}
\end{equation}
\endgroup
Finally, by fixing a small constant $r\in (0,d)$, and putting $|\un{x}_1-\un{x}_{n,j}^{(1)}|=Re^{-\frac{\lambda_{n,j}^{(1)}}{2}}$ and
$|\un{x}_2-\un{x}_{n,j}^{(1)}|=r$,
then Lemma \ref{lem_est_zetanj} and  \eqref{limit_zetan0} imply that,
\begin{equation} \label{result4}
\zeta_n(\un{x}_1)=-b_{j,0}+o_R(1)+o_n(1),\ \ \ \zeta_n(\un{x}_2)=-b_0+o_n(1),
\end{equation}
where $\lim_{R\to+\infty}o_R(1)=0$ and $\lim_{n\to+\infty}o_n(1)=0$. As a consequence, since $R>0$ and $r>0$ are arbitrary,
we see that \eqref{result3}-\eqref{result4} imply $b_{j,0}=b_0$ for $j=1,\cdots, m$,  in Case 2 as well.
This fact concludes the proof of Lemma \ref{lem_equalb0}.\end{proof}

\section{Estimates via Pohozaev type identities}\label{sec_poho}
From now on, for a given function $f(y,x)$, we shall use $\partial $ and $D$ to denote the partial derivatives with respect to $y$ and $x$ respectively.
{With a small abuse of notation, for a function $f(x)$ we will use both $\nabla$ and $D$ to denote its gradient.}

For $j=1,\cdots,m$, let
\begin{equation}\label{phij}
  \phi_{n,j}(y)=\frac{\rho_n}{m}(R(y,x_{n,j}^{(1)})-R(x_{n,j}^{(1)},x_{n,j}^{(1)}))+\frac{\rho_n}{m}\sum_{l\neq j}(G(y,x_{n,l}^{(1)})-G(x_{n,j}^{(1)},x_{n,l}^{(1)})),
\end{equation}
\begin{equation}\label{vnji}v_{n,j}^{(i)}(y)=\tilde{u}_n^{(i)}(y)-\phi_{n,j}(y),\ \ \ i=1,2.
\end{equation}
Recall the definition of $\zeta_n$ given before \eqref{fnx}. Our aim is to show that all $b_{j,i}=0$, see Lemma~\ref{lem_est_zetanj}.
We will start by showing that $b_{j,0}=0$. This is done by exploiting the following Pohozaev identity to derive a subtle estimate for $\zeta_n$.
\begin{lemma}[\cite{ly2}]
  \label{lem_po1} For any fixed {$r\in(0,\delta)$}, it holds,
\begin{equation}\label{poho1}\begin{aligned}
 &\frac{1}{2}\int_{\partial B_r(\un{x}_{n,j}^{(1)})}r<Dv_{n,j}^{(1)}+Dv_{n,j}^{(2)},D\zeta_n>\mathrm{d}\sigma -\int_{\partial B_r(\un{x}_{n,j}^{(1)})}r
 <\nu,D(v_{n,j}^{(1)}+v_{n,j}^{(2)})><\nu,D\zeta_n>\mathrm{d}\sigma
 \\&=\int_{\partial B_r(\un{x}_{n,j}^{(1)})}
 \frac{r\rho_nh_j(\un{x}) }{\|v_{n,j}^{(1)}-v_{n,j}^{(2)}\|_{L^\infty(M)}}(e^{v_{n,j}^{(1)}+
 \phi_{n,j}}-e^{v_{n,j}^{(2)}+\phi_{n,j}})\mathrm{d}\sigma
\\&  -\int_{B_r(\un{x}_{n,j}^{(1)})}\frac{\rho_nh_j(\un{x}) (e^{v_{n,j}^{(1)}+\phi_{n,j}}-e^{v_{n,j}^{(2)}+\phi_{n,j}})}
{\|v_{n,j}^{(1)}-v_{n,j}^{(2)}\|_{L^\infty(M)}}(2+<D(\log h_j+\phi_{n,j}),\un{x}-\un{x}_{n,j}^{(1)}>)\mathrm{d} \un{x}.
\end{aligned}
\end{equation}
\end{lemma}
\begin{proof} The identity
\eqref{poho1} has been first obtained in \cite{ly2}. We prove it  for reader's convenience.
First of all,  we observe that in local coordinates it holds,
\begin{equation}\label{poho1_st1}\begin{aligned}
  &\{\Delta (v_{n,j}^{(1)}-v_{n,j}^{(2)})\}\{\nabla(v_{n,j}^{(1)}+v_{n,j}^{(2)})\cdot(\un{x}-\un{x}_{n,j}^{(1)})\} +\{\Delta (v_{n,j}^{(1)}+v_{n,j}^{(2)})\}
  \{\nabla(v_{n,j}^{(1)}-v_{n,j}^{(2)})\cdot(\un{x}-\un{x}_{n,j}^{(1)})\}
\\=~&\textrm{div}\left\{\nabla(v_{n,j}^{(1)}-v_{n,j}^{(2)})[\nabla(v_{n,j}^{(1)}+v_{n,j}^{(2)})\cdot(\un{x}-\un{x}_{n,j}^{(1)})]\right\} +\textrm{div}\left\{\nabla(v_{n,j}^{(1)}+v_{n,j}^{(2)})[\nabla(v_{n,j}^{(1)}-v_{n,j}^{(2)})\cdot(\un{x}-\un{x}_{n,j}^{(1)})]\right\}
\\&-\textrm{div}\left\{[\nabla(v_{n,j}^{(1)}-v_{n,j}^{(2)})\cdot\nabla(v_{n,j}^{(1)}+v_{n,j}^{(2)})](\un{x}-\un{x}_{n,j}^{(1)})\right\}.
\end{aligned}\end{equation}
By the definition of $v_{n,j}^{(i)}$,  we also see that, {for $\un{x}\in B_r(\un{x}_{n,j}^{(1)})$,}
  \begin{equation*}
{\Delta} (v_{n,j}^{(1)}-v_{n,j}^{(2)})+\rho_nh_j({\un{x}})(e^{\tilde{u}_n^{(1)}}-e^{\tilde{u}_n^{(2)}})=0,
\ \ \textrm{and}\ \  {\Delta} (v_{n,j}^{(1)}+v_{n,j}^{(2)})+\rho_nh_j({\un{x}})(e^{\tilde{u}_n^{(1)}}+e^{\tilde{u}_n^{(2)}})=0.\end{equation*}
 and then we find that,
\begin{equation}\label{poho1_st2}\begin{aligned}
&\{\Delta (v_{n,j}^{(1)}-v_{n,j}^{(2)})\}\{\nabla(v_{n,j}^{(1)}+v_{n,j}^{(2)})\cdot(\un{x}-\un{x}_{n,j}^{(1)})\} +\{\Delta (v_{n,j}^{(1)}+v_{n,j}^{(2)})\}\{\nabla(v_{n,j}^{(1)}-v_{n,j}^{(2)})\cdot(\un{x}-\un{x}_{n,j}^{(1)})\}\\
&=-\rho_n (e^{v_{n,j}^{(1)}+\phi_{n,j}+\log h_j}-e^{v_{n,j}^{(2)}+\phi_{n,j}+\log h_j})\{\nabla(v_{n,j}^{(1)}+v_{n,j}^{(2)})
\cdot(\un{x}-\un{x}_{n,j}^{(1)})\}\\
&\quad-\rho_n (e^{v_{n,j}^{(1)}+\phi_{n,j}+\log h_j}+e^{v_{n,j}^{(2)}+\phi_{n,j}+\log h_j})
\{\nabla(v_{n,j}^{(1)}-v_{n,j}^{(2)})\cdot(\un{x}-\un{x}_{n,j}^{(1)})\}\\
&= -2\rho_n e^{v_{n,j}^{(1)}+\phi_{n,j}+\log h_j}\{\nabla v_{n,j}^{(1)}\cdot(\un{x}-\un{x}_{n,j}^{(1)})\}
+2\rho_n e^{v_{n,j}^{(2)}+\phi_{n,j}+\log h_j}\{\nabla v_{n,j}^{(2)}\cdot(\un{x}-\un{x}_{n,j}^{(1)})\}
\\&=-\textrm{div}\left(2\rho_n (e^{v_{n,j}^{(1)}+\phi_{n,j}+\log h_j}-e^{v_{n,j}^{(2)}+\phi_{n,j}+\log h_j})  (\un{x}-\un{x}_{n,j}^{(1)})\right)\\
&\quad+4\rho_n (e^{v_{n,j}^{(1)}+\phi_{n,j}+\log h_j}-e^{v_{n,j}^{(2)}+\phi_{n,j}+\log h_j}) +2\rho_n (e^{v_{n,j}^{(1)}+\phi_{n,j}+\log h_j}-e^{v_{n,j}^{(2)}+\phi_{n,j}+\log h_j})
\{\nabla(\phi_{n,j}+\log h_j)\cdot(\un{x}-\un{x}_{n,j}^{(1)})\}.
\end{aligned}\end{equation}
Clearly,  since $\zeta_n=\frac{v_{n,j}^{(1)}-v_{n,j}^{(2)}}{\|v_{n,j}^{(1)}-v_{n,j}^{(2)}\|_{L^{\infty}(M)}}$, then
\eqref{poho1_st1}, \eqref{poho1_st2} yield \eqref{poho1}, as claimed.
\end{proof}
Next we estimate both sides of \eqref{poho1}. Recall the definition of $A_{n,j}$ given in Lemma \ref{lem_out}.
 \begin{lemma}\label{poho1_lhs}
\begin{equation*}\label{lhs}\begin{aligned}
\mathrm{LHS~of}\ \eqref{poho1}
=&-4A_{n,j}{-\frac{256 b_0 e^{-\lambda_{n,1}^{(1)}}h(q_j)e^{G_j^*(q_j)}}{\rho_n (h(q_1))^2e^{G_1^*(q_1)}}
\int_{M_j\setminus  U^{\sscp M}_{r}(q_j)} e^{ \Phi_j(x,\mathbf{q})}  \mathrm{d}\mu(x)}
 {+o(e^{-\frac{\lambda_{n,j}^{(1)}}{2}}\sum_{l=1}^m|A_{n,l}|) + o(e^{-\lambda_{n,j}^{(1)}})}.                       \end{aligned}
\end{equation*}
for fixed $r\in (0,r_0)$ with $r_0$ as defined in \eqref{190117}.
\end{lemma}
\begin{proof}
Next, let us denote by,
\begin{equation}\label{tildeg}
  \widetilde{G}({\un{x}})=\frac{\rho_n}{m}\sum_{l=1}^m G({\un{x},\un{x}_{n,l}^{(1)}}),
\end{equation}
so that, for
$\un{x}\in B_{2r_0}(\un{x}_{n,j}^{(1)})\setminus\{\un{x}_{n,j}^{(1)}\}$, we have,
\begingroup
\abovedisplayskip=6pt
\belowdisplayskip=6pt
\begin{equation}\label{tildegwithphi}
               \nabla(\widetilde{G}(\un{x})-\phi_{n,j}(\un{x}))=-\frac{\rho_n}{2\pi m}\frac{\un{x}-\un{x}_{n,j}^{(1)}}{|\un{x}-\un{x}_{n,j}^{(1)}|^2}.
             \end{equation}
\endgroup
In view of \eqref{est_wn}, \eqref{rhon8pi}, and \eqref{rhon}, we conclude that,
\begingroup
\abovedisplayskip=6pt
\belowdisplayskip=6pt
\begin{equation*}
  \begin{aligned}
  o(e^{-\frac{\lambda_{n,j}^{(1)}}{2}})&=\nabla_{\sscp M} w_n =\nabla_{\sscp M}\left(v_{n,j}^{(i)}+\phi_{n,j}-\sum_{l=1}^m\frac{\rho_n}{m}G(x,x_{n,l}^{(i)})\right)  +
  O(\lambda_{n,j}^{(1)} e^{-\lambda_{n,j}^{(1)}})
  \ \ \textrm{in}\ \ M\setminus \cup_{l=1}^m U^{\sscp M}_{\delta}(x_{n,l}^{(i)}),
  \end{aligned}
\end{equation*}
\endgroup
where $\delta<\frac{r}{4}$. Therefore we find that,
\begingroup
\abovedisplayskip=6pt
\belowdisplayskip=6pt
\begin{equation}\label{grad_til}
  \begin{aligned}
  \nabla v_{n,j}^{(i)}= \nabla(\widetilde{G}-\phi_{n,j})+o(e^{-\frac{\lambda_{n,j}^{(1)}}{2}})\ \ \textrm{in}\ \  \cup_{l=1}^m
  B_{2r_0}(\un{x}_{n,l}^{(i)})\setminus B_{r/2}(\un{x}_{n,l}^{(i)}).
  \end{aligned}
\end{equation}
\endgroup
As a consequence, letting $\nu$ be the exterior unit normal, then \eqref{grad_til} together with \eqref{tildegwithphi} and \eqref{rhon}, imply that,
\begingroup
\abovedisplayskip=6pt
\belowdisplayskip=6pt
\begin{equation}\label{lhs1}\begin{aligned}
 \textrm{LHS of}\ \eqref{poho1}
 &=\int_{\partial B_{{r}}(\un{x}_{n,j}^{(1)})}r<D(\widetilde{G}-\phi_{n,j}),D\zeta_n>\mathrm{d}\sigma -
 2\int_{\partial B_{{r}}(\un{x}_{n,j}^{(1)})}r<\nu,D(\widetilde{G}-\phi_{n,j}))><\nu,D\zeta_n>\mathrm{d}\sigma
 \\& +o(e^{-\frac{\lambda_{n,j}^{(1)}}{2}}\|D\zeta_n\|_{L^\infty( {\partial B_{{r}}(\un{x}_{n,j}^{(1)})} )})
 \\&=\int_{\partial B_{{r}}(\un{x}_{n,j}^{(1)})}\frac{\rho_n}{2\pi m} <\nu,D\zeta_n>\mathrm{d}\sigma
 +o(e^{-\frac{\lambda_{n,j}^{(1)}}{2}}\|D\zeta_n\|_{L^\infty({\partial B_{{r}}(\un{x}_{n,j}^{(1)})} )})
                   \\&=\int_{\partial B_{{r}}(\un{x}_{n,j}^{(1)})} 4 <\nu,D\zeta_n>\mathrm{d}\sigma+
                   o(e^{-\frac{\lambda_{n,j}^{(1)}}{2}}\|D\zeta_n\|_{L^\infty({\partial B_{{r}}(\un{x}_{n,j}^{(1)})} )}).  \end{aligned}
\end{equation}
\endgroup

\begingroup
\abovedisplayskip=6pt
\belowdisplayskip=6pt
By \eqref{lhs1} and Lemma \ref{lem_out}, we also see that,
\begin{equation}\label{lhs2}\begin{aligned}
&\textrm{LHS of}\ \eqref{poho1}
                     =\int_{\partial B_{{r}}(\un{x}_{n,j}^{(1)})} 4 <\nu,D\zeta_n>\mathrm{d}\sigma+o(e^{- \lambda_{n,j}^{(1)} })+o(e^{-\frac{\lambda_{n,j}^{(1)}}{2}} \sum_{l=1}^m|A_{n,l}|).  \end{aligned}
\end{equation}
\endgroup
To estimate the right hand side of \eqref{lhs2}, we need a refined estimate about $\zeta_n$ on $\partial B_{r} (\un{x}_{n,j}^{(1)})$.
So, by the Green representation formula with $x\in {\partial U^{\sscp M}_{r}({x}_{n,j}^{(1)})}$, we find that (see \eqref{green_zeta}),
{\allowdisplaybreaks\begin{align*}
&\zeta_n(x)-\int_M \zeta_n\mathrm{d}\mu= \int_MG(y,x)f_n^*(y)\mathrm{d}\mu(y)\\
&=\sum_{ {l} =1}^m A_{n, {l} }G(x_{n, {l} }^{(1)},x)
+\sum_{ {l} =1}^m\sum_{h=1}^2 {B_{n, {l} ,h}}T_{*}^{-1}\left(\partial_{\un{y}_h} G(\un{y},\un{x})\Big|_{\un{y}=\un{x}_{n, {l} }^{(1)}}\right)
 +\frac{1}{2}\sum_{ {l} =1}^m\sum_{h, { {k}}=1}^2C_{n, {l} ,h,{ {k}}}T_{ {*}}^{-1}\left(\partial^2_{\un{y}_h \un{y}_{ {k}}}G(\un{y},\un{x})\Big|_{\un{y}=\un{x}_{n,{l} }^{(1)}}\right)
\\&+\sum_{ {l} =1}^m\int_{M_{ {l}} }\Psi_{n, {l} }(y,x)f_n^*(y)\mathrm{d}\mu(y), \ \ \textrm{where}
\end{align*}}
\begin{align}
& A_{n, {l} }=\int_{M_{ {l}} }f_n^*(y)\mathrm{d}\mu(y),\ \ \ B_{n, {l} ,h}=\int_{B_{r_0}(x_{n,l}^{(1)})}(\un{y}-\un{x}_{n, {l} }^{(1)})_h f_n^*(\un{y})e^{2\varphi_{{l}} }\mathrm{d}\un{y},\notag
\end{align}
\begin{align}
C_{n, {l} ,h,{ {k}}}  =\int_{B_{r_0}(x_{n,l}^{(1)})}(\un{y}-\un{x}_{n, {l} }^{(1)})_h(\un{y}-\un{x}_{n, {l} }^{(1)})_{ {k}} f_n^*(\un{y})e^{2\varphi_{ {l}} }\mathrm{d}\un{y},\notag
\end{align}
\begin{align}
\label{dnj}
{\Psi_{n, {l} }(y,x)}=~&G(y,x)-G(x_{n, {l} }^{(1)},x) -
T_{ {*}} ^{-1}\left(<\partial_{\un{y}} G(\un{y},\un{x})\Big|_{\un{y}=\un{x}_{n, {l} }^{(1)}},\un{y}-\un{x}_{n, {l} }^{(1)}>
{1_{B_{r_0}(x_{n,l}^{(1)})}(\un{y})}\right)
\nonumber\\&- T_{ {*}} ^{-1}\left(\frac{1}{2}<\partial^2_{\un{y}}  G(\un{y},\un{x})\Big|_{\un{y}=\un{x}_{n, {l} }^{(1)}}
(\un{y}-\un{x}_{n, {l} }^{(1)}),\un{y}-\un{x}_{n, {l} }^{(1)}>
{1_{B_{r_0}(x_{n,l}^{(1)})}(\un{y})}\right).
\end{align}
At this point,  let us fix {$0<\overline{\theta}<\frac{r}{2}$}. By using Lemma \ref{lem_diff_tilu} and Lemma \ref{lem_equalb0}, we find that,
\begin{equation}\begin{aligned}\label{est_fn1}f_n^*{(y)}&=\rho_nh e^{\tilde{u}_n^{(1)}}(\zeta_n{(y)}+O(\|\tilde{u}_n^{(1)}-\tilde{u}_n^{(2)}\|_{L^{\infty}(M)}))
=\rho_nh e^{\tilde{u}_n^{(1)}{(y)}}(-b_0+o(1)),
\end{aligned}\end{equation}
for any {$y\in   M_j\setminus  U^{\sscp M}_{\overline{\theta}}(x_{n,j}^{(1)})$,} and, in view of  \eqref{wnx} and \eqref{est_wn},
\begin{equation}\begin{aligned}\label{est_fn1_1}\tilde{u}_n^{(1)}{(y)}-\sum_{l=1}^m\rho_{n,l}^{(1)} G({y},x_{n,l}^{(1)})-\int_M
\tilde{u}_n^{(1)} \mathrm{d}\mu=o(e^{-\frac{\lambda_{n,j}^{(1)}}{2}}) \ \ \textrm{ for}\ \
{y}\in {M_j}\setminus  U^{\sscp M}_{\overline{\theta}}(x_{n,j}^{(1)}).\end{aligned}\end{equation}
By \eqref{est_fn1}-\eqref{est_fn1_1}, \eqref{lambda_exp}, and {\eqref{rhon8pi}}, we conclude that,
\begin{equation}\begin{aligned}\label{est_fn2}f_n^*{(y)}&=\rho_nh e^{\sum_{l=1}^m\rho_{n,l}^{(1)} G({y},x_{n,l}^{(1)})+\int_M\tilde{u}_n^{(1)} \mathrm{d}\mu}(-b_0+o(1))
\\&=\rho_nh e^{-\lambda_{n,j}^{(1)}-2\log(\frac{\rho_n h(x_{n,j}^{(1)})}{8})-G_j^*(x_{n,j}^{(1)})+\sum_{l=1}^m\rho_{n,l}^{(1)} G({y},x_{n,l}^{(1)}) }(-b_0+o(1))
\\&= \frac{64 e^{-\lambda_{n,j}^{(1)}}}{\rho_nh(x_{n,j}^{(1)})}
{e^{ \Phi_j(y,\mathbf{x}_n)}}(-b_0+o(1)) \ \ \textrm{for} \ \
{y\in   M_j}\setminus  U^{\sscp M}_{\overline{\theta}}(x_{n,j}^{(1)}),
\end{aligned}\end{equation}
where $\mathbf{x}_n=(x_{n,1},x_{n,2},\cdots,x_{n,m})$ and,
 $\Phi_j(y,\mathbf{x}_n)=\sum_{l=1}^m 8\pi  G(y,x_{n,l}^{(1)})-G_j^*(x_{n,j}^{(1)})+\log h(y)-\log h(x_{n,j}^{(1)}).$ 
 
On the other hand, by \eqref{eta_ets}, we have, for ${y}\in U^{\sscp M}_{\overline{\theta}}(x_{n,j}^{(1)})$,
\begin{equation}\begin{aligned}\label{est_fn2_in}f_n^*{(y)}&=\rho_nh e^{\tilde{u}_n^{(1)}}(\zeta_n{(y)}+O(\|\tilde{u}_n^{(1)}-\tilde{u}_n^{(2)}\|_{L^{\infty}(M)}))
 =O\left(\frac{e^{\lambda_{n,j}^{(1)}}}{(1+e^{\lambda_{n,j}^{(1)}}|\un{y}-\un{x}_{n,j}^{(1)}|^2)^2}\right)\le
O\left(\frac{e^{-\lambda_{n,j}^{(1)}}}{|\un{y}-\un{x}_{n,j}^{(1)}|^4}\right).
\end{aligned}\end{equation}
Next, by   \eqref{dnj},
we have for $y\in U^{\sscp M}_{\overline{\theta}}(x_{n,j}^{(1)})$ and $x\in\partial U^{\sscp M}_r(x_{n,j}^{(1)})$,
\begin{equation}\begin{aligned}\label{est_fn2_1}
\Psi_{n,j}({y,x})=O\left(\frac{|\un{y}-\un{x}_{n,j}^{(1)}|^3}{|\un{x}-\un{x}_{n,j}^{(1)}|^3}\right),\ \ {\textrm{and}\ \
(\nabla_{\sscp M})_x\Psi_{n,j}({y,x})=O\left(\frac{|\un{y}-\un{x}_{n,j}^{(1)}|^3}{|\un{x}-\un{x}_{n,j}^{(1)}|^4}\right).}\end{aligned}\end{equation}
Let us define,
\begingroup
\abovedisplayskip=6pt
\belowdisplayskip=6pt
\begin{equation}\begin{aligned}\label{overlineg}
\overline{G}_n(x)&=\int_M\zeta_n\mathrm{d}\mu+\sum_{{ {l}}=1}^m A_{n,{ {l}}}G(x_{n,{ {l}}}^{(1)},x)
+\sum_{{ {l}}=1}^m\sum_{h=1}^2{ B_{n,{ {l}},h}}T_{ {l}}^{-1}\left(\partial_{\un{y}_h} G(\un{y},\un{x})\Big|_{\un{y}=\un{x}_{n,{ {l}}}^{(1)}}\right)\\&+
\frac{1}{2}\sum_{{ {l}}=1}^m\sum_{h, {{k}}=1}^2C_{n,{ {l}},h,{{k}}}
T_{ {l}}^{-1}\left(\partial^2_{\un{y}_h \un{y}_{{k}}}G(\un{y},\un{x})\Big|_{\un{y}=\un{x}_{n,{ {l}}}^{(1)}}\right),\end{aligned}\end{equation}
\endgroup

so that, by \eqref{est_fn2}-\eqref{est_fn2_1}, we conclude that,  for $x\in  \partial U^{\sscp M}_r(x_{n,j}^{(1)})$, it holds,
\begingroup
\abovedisplayskip=6pt
\belowdisplayskip=6pt
\begin{equation}\label{zeta_and_g_n}
\begin{aligned}   &\zeta_n(x)-\overline{G}_n(x)
 =\sum_{l=1}^m{\left(\int_{{M_l\setminus U^{\sscp M}_{\overline{\theta}}(x_{n,l}^{(1)})}}
\Psi_{n,l}(y,x)f_n^*(y)\mathrm{d}\mu(y) + \int_{U_{\overline{\theta}}(x_{n,l}^{(1)})}\Psi_{n,l}(y,x)f_n^*(y)\mathrm{d}\mu(y)\right)}
\\&{=-b_0\sum_{l=1}^m \int_{{M_l\setminus U^{\sscp M}_{\overline{\theta}}(x_{n,l}^{(1)})}}
\frac{64 e^{-\lambda_{n,l}^{(1)}}\Psi_{n,l}(y,x)}
{\rho_nh(x_{n,l}^{(1)})}{e^{\Phi_l({y},\mathbf{x}_n)}}\mathrm{d}\mu(y)}
 {+O\left( \sum_{l=1}^m \int_{B_{{\overline{\theta}}}(\un{x}_{n,l}^{(1)})}\frac{|\un{y}-\un{x}_{n,l}^{(1)}|^3}{|\un{x}-\un{x}_{n,j}^{(1)}|^3}
\frac{e^{-\lambda_{n,l}^{(1)}}e^{2\varphi_j}}{|\un{y}-\un{x}_{n,l}^{(1)}|^4}\mathrm{d}\un{y}\right)}{+o(e^{-\lambda_{n,j}^{(1)}})}
 \\&=-b_0\sum_{l=1}^m \int_{{M_l\setminus U^{\sscp M}_{\overline{\theta}}(x_{n,l}^{(1)})}}
\frac{64 e^{-\lambda_{n,l}^{(1)}}\Psi_{n,l}(y,x)}{\rho_nh(x_{n,l}^{(1)})}
{e^{ \Phi_l({y},\mathbf{x}_n)}}\mathrm{d}\mu(y)
 +O\left( \frac{{\overline{\theta}}e^{-\lambda_{n,j}^{(1)}}}{|\un{x}-\un{x}_{n,j}^{(1)}|^3}\right){+o(e^{-\lambda_{n,j}^{(1)}})}{\ \ \textrm{in}}\ \
C^1(\partial U^{\sscp M}_r(x_{n,j}^{(1)})),
\end{aligned}\end{equation}
\endgroup

At this point, let us set,
\begin{equation}
\label{zeta_s}
\zeta_n^*(x)=-b_0\sum_{l=1}^m \int_{{M_l\setminus U^{\sscp M}_{\overline{\theta}}(x_{n,l}^{(1)})}}
\frac{64 e^{-\lambda_{n,l}^{(1)}}\Psi_{n,l}(y,x)}{\rho_nh(x_{n,l}^{(1)})} {e^{ \Phi_l({y},\mathbf{x}_n)}}\mathrm{d}\mu(y),
\end{equation}
and then substitute \eqref{zeta_and_g_n} into \eqref{lhs2}, to obtain,
\begin{equation}
\label{lhs3}
\begin{aligned}
 \textrm{LHS of}\ \eqref{poho1}
 &=\int_{\partial B_r(\un{x}_{n,j}^{(1)})} 4 <\nu,D(\overline{G}_n+\zeta_n^*){(\un{x})>\mathrm{d}\sigma(\un{x})}   +o(e^{-\frac{\lambda_{n,j}^{(1)}}{2}} \sum_{l=1}^m|A_{n,l}|)+O\left( \frac{\overline{\theta}e^{-\lambda_{n,j}^{(1)}}}{r^3}\right)+o(e^{- \lambda_{n,j}^{(1)} }),
\end{aligned}
\end{equation} {for any} $\overline{\theta}\in(0,\frac{r}{2}).$
To estimate the right hand side of \eqref{lhs3}, we note that for any pair of (smooth enough) functions $u$ and $v$, it holds,
\begin{equation}\label{eq_uv}\begin{aligned}
&  \Delta u \{\nabla  v \cdot (\un{x}-\un{x}_{n,j}^{(1)})\}+\Delta v\{\nabla u\cdot (\un{x}-\un{x}_{n,j}^{(1)})\}
  \\&=\textrm{div}\left(\nabla u (\nabla v \cdot (\un{x}-\un{x}_{n,j}^{(1)}))+\nabla v  (\nabla u\cdot (\un{x}-\un{x}_{n,j}^{(1)}))-
  \nabla u\cdot\nabla v (\un{x}-\un{x}_{n,j}^{(1)})\right).\end{aligned}
\end{equation}
In view of \eqref{overlineg} and \eqref{int_tilu}, we also see that, for any fixed $\un{\theta}\in(0,r)$,
\begin{equation}\label{del1}\begin{aligned}
\Delta\overline{G}_n{(\un{x})}&=\sum_{l=1}^m A_{n,l}=\int_M f_n^*\mathrm{d}\mu  =
\int_M\frac{\rho_nh(e^{\tilde{u}_n^{(1)}}-e^{\tilde{u}_n^{(2)}})}{\|\tilde{u}_n^{(1)}-\tilde{u}_n^{(2)}\|_{L^{\infty}(M)}}
\mathrm{d}\mu=0\ \ {\textrm{for}\ \ \un{x}\in} U^{\sscp M}_r({x}_{n,j}^{(1)})\setminus U^{\sscp M}_{\un{\theta}}({x}_{n,j}^{(1)}),\end{aligned}\end{equation}
and, moreover, by using \eqref{tildeg} and \eqref{phij}, we have,
\begin{equation}\label{del2}
\Delta (\widetilde{G}-\phi_{n,j}){(\un{x})}=0\ \ {\textrm{for}\ \ \un{x}\in} B_r(\un{x}_{n,j}^{(1)})\setminus B_{\un{\theta}}(\un{x}_{n,j}^{(1)}).\end{equation}
By using {\eqref{eq_uv}-\eqref{del2} and \eqref{tildegwithphi}}, we conclude that,
\begin{equation*}\label{after_inserting0}\begin{aligned}
0=&\int_{B_r(\un{x}_{n,j}^{(1)})\setminus B_{\un{\theta}}(\un{x}_{n,j}^{(1)})}
\Bigg[\Delta  \overline{G}_n \{\nabla (\widetilde{G}-\phi_{n,j})\cdot (\un{x}-\un{x}_{n,j}^{(1)})\} +
\Delta(\widetilde{G}-\phi_{n,j}) \{\nabla \overline{G}_n\cdot (\un{x}-\un{x}_{n,j}^{(1)})\}\Bigg]\mathrm{d}\un{x}
  \\=&\int_{\partial [B_r(\un{x}_{n,j}^{(1)})\setminus B_{\un{\theta}}(\un{x}_{n,j}^{(1)})]}
  \Bigg[\frac{\partial \overline{G}_n }{\partial \nu}(\nabla (\widetilde{G}-\phi_{n,j})\cdot (\un{x}-\un{x}_{n,j}^{(1)}))\\&+
  \frac{\partial (\widetilde{G}-\phi_{n,j})}{\partial\nu} (\nabla \overline{G}_n\cdot (\un{x}-\un{x}_{n,j}^{(1)})) -\nabla
  \overline{G}_n\cdot\nabla(\widetilde{G}-\phi_{n,j})<\un{x}-\un{x}_{n,j}^{(1)},\nu>\Bigg]\mathrm{d} \sigma
  \\=&-\frac{\rho_n}{2\pi m}\int_{\partial [B_r(\un{x}_{n,j}^{(1)})\setminus B_{\un{\theta}}(\un{x}_{n,j}^{(1)})]}
  \frac{\partial \overline{G}_n }{\partial \nu}\mathrm{d} \sigma,\end{aligned}
\end{equation*}
and thus,
\begin{equation}\label{after_inserting}\begin{aligned}
  \int_{\partial  B_r(\un{x}_{n,j}^{(1)}) }\frac{\partial \overline{G}_n }{\partial \nu}{(\un{x})\mathrm{d} \sigma(\un{x})}=
  \int_{\partial  B_{\un{\theta}}(\un{x}_{n,j}^{(1)}) }\frac{\partial \overline{G}_n }{\partial \nu}{(\un{x})\mathrm{d} \sigma(\un{x})}.
  \end{aligned}
\end{equation}
At this point, let us denote by $o_{\sscp \un{\theta}}(1)$ any quantity which converges to $0$ as $ \un{\theta}\to 0^+$, and then observe that,
\begin{equation}
\label{aterm}
\begin{aligned}
&4\int_{\partial B_{\un{\theta}}(\un{x}_{n,j}^{(1)})}<\nu, \sum_{l=1}^m A_{n,l} D_{\un{x}}G(\un{x}_{n,l}^{(1)},\un{x})>\mathrm{d}\sigma{(\un{x})}\\
&=4A_{n,j}\int_{\partial B_{\un{\theta}}(\un{x}_{n,j}^{(1)})}<\nu, D_{\un{x}}G(\un{x}_{n,j}^{(1)},\un{x})>\mathrm{d}\sigma{(\un{x})}+o_{\sscp \un{\theta}}(1) =-4A_{n,j}+o_{\sscp \un{\theta}}(1){.}
\end{aligned}
\end{equation}
 {By setting $z=\un{x}-\un{x}_{n,j}^{(1)}$ and since,}
$
 { D_iD_h(\log |z|)=\frac{\delta_{ih}|z|^2-2z_iz_h}{|z|^4},}
$
{then we find that,}
\begin{equation}
\label{int_theta0}
\begin{aligned}
&\int_{\partial B_{\un{\theta}}(\un{x}_{n,j}^{(1)})}<\nu,D_{\un{x}}\partial_{\un{y}_h}(\log|\un{y}-\un{x}|)\Big|_{\un{y}=\un{x}_{n,j}^{(1)}}>
\mathrm{d}\sigma{(\un{x})}=-\int_{\partial B_{\un{\theta}}(0)}\sum_{i=1}^2\frac{z_i}{|z|}\left(\frac{\delta_{ih}|z|^2-{2z_iz_h}}{|z|^4}\right)\mathrm{d}\sigma{(z)}=0.
\end{aligned}
\end{equation} 
We observe that, if $h=k$ then,
$
  D_{i}(\log |z|)=\frac{z_i}{|z|^2},\ \ D_iD_{hh}^2(\log  |z|)=-\frac{2z_i}{|z|^4}-\frac{4z_h\delta_{ih}}{|z|^4}+\frac{8z_h^2z_i}{|z|^6},
$
and thus,
\begin{equation}\label{small1}\begin{aligned}
 & \int_{\partial B_{\un{\theta}}(\un{x}_{n,j}^{(1)})}<\nu,D_{\un{x}}\frac{\partial^2}{\partial \un{y}_h^2}\log\frac{1}{|\un{y}-\un{x}|}
  \Big|_{\un{y}=\un{x}_{n,j}^{(1)}}>\mathrm{d}\sigma{(\un{x})} =\int_{\partial B_{\un{\theta}}(0)}\left({\frac{2}{|z|^3}-\frac{4z_h^2}{|z|^5}}\right)
  \mathrm{d}\sigma{(z)}=0.\end{aligned}
\end{equation}
If $h\neq k$, then,
$
  D_iD_{hk}^2(\log  |z|)=-\frac{2(z_h\delta_{ki}+z_k\delta_{hi})}{|z|^4}+\frac{8z_kz_iz_h}{|z|^6},
$
which implies that,
\begin{equation}\label{int_theta2}\begin{aligned}
 & \int_{\partial B_{\un{\theta}}(\un{x}_{n,j}^{(1)})}<\nu,D_{\un{x}}\frac{\partial^2}{\partial \un{y}_k\partial \un{y}_h}\log\frac{1}{|\un{y}-\un{x}|}
  \Big|_{\un{y}=\un{x}_{n,j}^{(1)}}>\mathrm{d}\sigma{(\un{x})} =
  \int_{\partial B_{\un{\theta}}(0)}\left(\frac{4z_hz_k}{|z|^5}-\frac{8z_hz_k}{|z|^5}\right)\mathrm{d}\sigma{(z)}=0.\end{aligned}
\end{equation}
From \eqref{after_inserting}-{\eqref{int_theta2}}, we conclude that,
\begin{equation}\label{small2}
  4\int_{\partial B_r(\un{x}_{n,j}^{(1)})}<\nu, D_{\un{x}}\overline{G}_n(\un{x})>\mathrm{d}\sigma{(\un{x})}=-4A_{n,j}+o_{\sscp \un{\theta}}(1).
\end{equation}
Next we estimate the other term in \eqref{lhs3}, that is $4\int_{\partial B_r(\un{x}_{n,j}^{(1)})}<\nu, D_{\un{x}}\zeta_n^*{(\un{x})}>\mathrm{d}\sigma{(\un{x})}$, where $\zeta_n^*$
is defined in \eqref{zeta_s}. Clearly we have,
\begin{equation*}\label{Psinj}\begin{aligned}
D_{\un{x}}\Psi_{n,j}(\un{y},\un{x})&=D_{\un{x}}\Bigg\{G(\un{y},\un{x})-G(\un{x}_{n,j}^{(1)},\un{x})-
<\partial_{\un{y}} G(\un{y},\un{x})\Big|_{\un{y}=\un{x}_{n,j}^{(1)}}, \un{y}-\un{x}_{n,j}^{(1)}>{1_{B_{r_0}(x_{n,j}^{(1)})}(\un{y})}
\\&-\frac{1}{2}<\partial^2_{\un{y}}G(\un{y},\un{x})\Big|_{\un{y}=\un{x}_{n,j}^{(1)}}(\un{y}-\un{x}_{n,j}^{(1)}),\un{y}-\un{x}_{n,j}^{(1)}>
{1_{B_{r_0}(x_{n,l}^{(1)})}(\un{y})}\Bigg\}.
\end{aligned}
\end{equation*}
If ${{y}\in M_j\setminus U^{\sscp M}_{\overline{\theta}}({x}_{n,j}^{(1)}),\;\un{y}=T_a(y)}$  and $\un{x}\in\partial B_{{\theta}}(\un{x}_{n,j}^{(1)})$ with
${\theta}\ll\overline{\theta}^2$, then we find that,
\begin{equation}\label{bdd1}
{|D_{\un{x}}G(\un{y},\un{x})|\le \frac{C}{\sqrt{\theta}}\ \ \textrm{ for some constant}\ \  C>0,}
\end{equation}
which implies $\int_{\partial B_{{\theta}}(\un{x}_{n,j}^{(1)})}<\nu, D_{\un{x}}G(\un{y},\un{x})>\mathrm{d} \un{x}=o_{\sscp {\theta}}(1).$ Thus \eqref{int_theta0}, \eqref{small1}, and \eqref{int_theta2}, \eqref{bdd1} imply that,
\begin{equation}\label{boundary_int1}
\begin{aligned}
&4\int_{\partial B_{{\theta}}(\un{x}_{n,j}^{(1)})}<\nu, D_{\un{x}}\Psi_{n,j}(\un{y},\un{x})>\mathrm{d} \un{x}
=- 4\int_{\partial B_{{\theta}}(\un{x}_{n,j}^{(1)})}<\nu, D_{\un{x}}G(\un{x}_{n,j}^{(1)},\un{x})>\mathrm{d} \un{x}
\\&\quad-4\int_{\partial B_{{\theta}}(\un{x}_{n,j}^{(1)})}<\nu,  D_{\un{x}}<\partial_{\un{y}} G(\un{y},\un{x})
\Big|_{\un{y}=\un{x}_{n,j}^{(1)}},\un{y}-\un{x}_{n,j}^{(1)}>
{1_{B_{r_0}(x_{n,j}^{(1)})}(\un{y})}>\mathrm{d} \un{x}
\\&\quad-2\int_{\partial B_{{\theta}}(\un{x}_{n,j}^{(1)})} <\nu,  D_{\un{x}} <\partial^2_{\un{y}}  G(\un{y},\un{x})\Big|_{\un{y}=\un{x}_{n,j}^{(1)}}
(\un{y}-\un{x}_{n,j}^{(1)}),\un{y}-\un{x}_{n,j}^{(1)}>1_{B_{r_0}(x_{n,j}^{(1)})(\un{y})}>\mathrm{d}\un{x}+ o_{\sscp {\theta}}(1)
\\&=4+o_{\sscp {\theta}}(1)\ \ \ {\textrm{for}\ \ {{y}\in M_j\setminus U^{\sscp M}_{\overline{\theta}}({x}_{n,j}^{(1)}),\;\un{y}=T_a(y)} \ \ \textrm{and}\ \
\un{x}\in \partial B_{{\theta}}(\un{x}_{n,j}^{(1)}). }
\end{aligned}
\end{equation}
At this point, let us observe that $-\Delta_{\un{x}}\Psi_{n,j}(\un{y},\un{x})=\delta_{\un{y}}$ ${\textrm{for}\ \ \un{x}\in
B_{r}(\un{x}_{n,j}^{(1)})\setminus B_{{\theta}}(\un{x}_{n,j}^{(1)})}$ and
let us choose $u(\un{x})=\Psi_{n,j}(\un{y},\un{x})$ and $v(\un{x})=\widetilde{G}(\un{x})-\phi_{n,j}(\un{x})$ in \eqref{eq_uv}.
Then we consider the following three cases:

{(i) If $\un{y}\in B_r(\un{x}_{n,j}^{(1)})\setminus B_{\overline{\theta}}(\un{x}_{n,j}^{(1)})$, then,} from \eqref{boundary_int1} {and \eqref{eq_uv}},
we obtain that,
\begin{equation}\label{boundary_int2}
          \begin{aligned}
          &4\int_{\partial B_{r}(\un{x}_{n,j}^{(1)})}<\nu, D_{\un{x}}\Psi_{n,j}(\un{y},\un{x})>\mathrm{d} \un{x}
              =4\int_{\partial B_{{\theta}}(\un{x}_{n,j}^{(1)})}<\nu, D_{\un{x}}\Psi_{n,j}(\un{y},\un{x})>\mathrm{d} \un{x}-4=o_{\sscp {\theta}}(1).
                    \end{aligned}
        \end{equation}

      {(ii)   If $y\in {M_j\setminus U^{\sscp M}_r({x}_{n,j}^{(1)}),\;\un{y}=T_a(y)}$, then we see from \eqref{boundary_int1} and \eqref{eq_uv}} that,
       \begin{equation}\label{boundary_int3}
          \begin{aligned}
          &4\int_{\partial B_{r}(\un{x}_{n,j}^{(1)})}<\nu, D_{\un{x}}\Psi_{n,j}(\un{y},\un{x})>\mathrm{d} \un{x}
   =4\int_{\partial B_{{\theta}}(\un{x}_{n,j}^{(1)})}<\nu, D_{\un{x}}\Psi_{n,j}(\un{y},\un{x})>\mathrm{d} \un{x}=4+o_{\sscp {\theta}}(1).
                    \end{aligned}
        \end{equation}

     {(iii) If $y\in {M_l}$ and $\un{x}\in\partial B_{{\theta}}(\un{x}_{n,j}^{(1)})$, $l\neq j$, then} we have $|D_{\un{x}}\Psi_{n,l}(\un{y},\un{x})|\le C$
     for some constant $C>0$. So we conclude
         \begin{equation}\label{boundary_int4}
          \begin{aligned}
          4\int_{\partial B_{r}(\un{x}_{n,j}^{(1)})}<\nu, D_{\un{x}}\Psi_{n,l}(\un{y},\un{x})>\mathrm{d} \un{x}=
          &4\int_{\partial B_{{\theta}}(\un{x}_{n,j}^{(1)})}<\nu, D_{\un{x}}\Psi_{n,l}(\un{y},\un{x})>\mathrm{d} \un{x}
=o_{\sscp {\theta}}(1),
                    \end{aligned}
        \end{equation}
and so, by \eqref{zeta_s} and \eqref{boundary_int2}-\eqref{boundary_int4}, we finally conclude that,
\begin{equation*}\label{boundary_int5}
          \begin{aligned}
          & {4\int_{\partial B_{r}(\un{x}_{n,j}^{(1)})}<\nu, D_{\un{x}}\zeta_n^*{(\un{x})}>\mathrm{d} \un{x}}
       ={
         -\sum_{l=1}^m\frac{256 b_0 e^{-\lambda_{n,l}^{(1)}}}{\rho_n h(x_{n,l}^{(1)})}
         \int_{{{M_l}\setminus U^{\sscp M}_{{\overline{\theta}}}({x}_{n,l}^{(1)})}}T_{*}^{-1}\left(
         \int_{\partial B_r(\un{x}_{n,j}^{(1)})}<\nu, D_{\un{x}}\Psi_{n,l}>\mathrm{d} \un{x}\right)  {e^{ \Phi_{{l}}(y,\mathbf{x}_n)}}
         \mathrm{d}\mu(y)}
         \\&={ -\frac{256 b_0 e^{-\lambda_{n,j}^{(1)}}}{\rho_n h(x_{n,j}^{(1)})}\int_{{{M_j}}\setminus U^{\sscp M}_{r}({x}_{n,j}^{(1)})}
         {e^{ \Phi_j(y,\mathbf{x}_n)}}\mathrm{d}\mu(y)+o(e^{-\lambda_{n,j}^{(1)}}).}
                    \end{aligned}
        \end{equation*}
        Finally, By \eqref{relation_lambda} and \eqref{diffxandq}, we see that,
        \begin{equation}\label{boundary_int6}
          \begin{aligned}
         & {4\int_{\partial B_{r}(\un{x}_{n,j}^{(1)})}<\nu, D_{{\un{x}}}\zeta_n^*>\mathrm{d}{\un{x}}} ={-\frac{256 b_0 e^{-\lambda_{n,1}^{(1)}}h(q_j)e^{G_j^*(q_j)}}{\rho_n (h(q_1))^2e^{G_1^*(q_1)}}
              \int_{M_j\setminus U^{\sscp M}_{r}(q_j)} e^{ \Phi_j(y,\mathbf{q})} \mathrm{d}\mu(y)+o(e^{-\lambda_{n,j}^{(1)}}).}
                    \end{aligned}
        \end{equation}
Obviously \eqref{lhs3}, \eqref{small2}, and \eqref{boundary_int6} conclude the proof of Lemma \ref{poho1_lhs}.
              \end{proof}

To estimate the right hand side of \eqref{poho1} of Lemma \ref{lem_po1}, we note  that,
\begin{equation*}\label{def_gn}\begin{aligned}
f_n^*(x)&=\frac{\rho_nh(x)(e^{v_{n,j}^{(1)}+\phi_{n,j}}-e^{v_{n,j}^{(2)}+\phi_{n,j}})}{\|v_{n,j}^{(1)}-v_{n,j}^{(2)}\|_{L^\infty(M)}}
=\frac{\rho_nh(x)e^{\tilde{u}_{n,j}^{(1)}}(1-e^{\tilde{u}_{n,j}^{(2)}-\tilde{u}_{n,j}^{(1)}})}{\|\tilde{u}_{n,j}^{(1)}-\tilde{u}_{n,j}^{(2)}\|_{L^\infty(M)}}
=\rho_nh(x)e^{\tilde{u}_{n,j}^{(1)}}(\zeta_n+o(1)),\end{aligned}
\end{equation*}
where we used \eqref{fnx} and \eqref{vnji}. Then we will need the following estimate:
\begin{lemma}\label{poho1_rhs}
\begin{equation*}\begin{aligned} & (i) \   \int_{\partial B_r(\un{x}_{n,j}^{(1)})} rf_n^* e^{2\varphi_j}\mathrm{d}\sigma
= {-\frac{ 128e^{-\lambda_{n,1}^{(1)}}b_0h(\un{q}_j)e^{G_j^*(\un{q}_j)}}{\rho_n (h(\un{q}_1))^2e^{G_1^*(\un{q}_1)}}
\frac{\pi}{r^2}}\\
&{-\frac{32\pi e^{-\lambda_{n,1}^{(1)}}b_0  h(\un{q}_j)e^{G_j^*(\un{q}_j)} (\Delta\log h(\un{q}_j)+8\pi m -2K(\un{q}_j))}
{\rho_n (h(\un{q}_1))^2 e^{G_1^*(\un{q}_1)}}}
  {+O(r e^{-\lambda_{n,1}^{(1)}})+\frac{o(e^{-\lambda_{n,1}^{(1)}})}{r^2}}.
\\&  (ii) \ \sum_{j=1}^m\int_{U^{\sscp M}_r(x_{n,j}^{(1)})}   f_n^* \mathrm{d}\mu( x)
 =\frac{64b_0 e^{-\lambda_{n,1}^{(1)}}}{\rho_n (h(q_1))^2e^{G_1^*(q_1)}} \sum_{j=1}^m
\int_{M_j\setminus  U^{\sscp M}_{r}(q_j)} h(q_j)e^{G_j^*(q_j)} e^{ \Phi_j(x,\mathbf{q})}  \mathrm{d}\mu(x)
  +o(e^{-\lambda_{n,1}^{(1)}}).
  \\& (iii)\   \int_{B_r(\un{x}_{n,j}^{(1)})} f_n^*e^{2\varphi_j}<D(\log h+\phi_{n,j}),\un{x}-\un{x}_{n,j}^{(1)}> \mathrm{d}\un{x}
 \\&=\Bigg[{\frac{32\pi \Big(\int_M\zeta_n\mathrm{d}\mu
{-\frac{\|\tilde{u}_n^{(1)}-\tilde{u}_n^{(2)}\|_{L^{\infty}(M)}}{2}(\int_M\zeta_n\mathrm{d}\mu)^2}\Big)(\Delta \log h(\un{q}_j) +8\pi m -2K(\un{q}_j))
 }{\rho_n (h(\un{q}_1))^2e^{G_{1}^*(\un{q}_1)}}}
 \\& {\times h(\un{q}_j)
 e^{G_{j}^*(\un{q}_j)}e^{-\lambda_{n,1}^{(1)}} \Big(\lambda_{n,1}^{(1)}+\log\Big(\frac{\rho_n (h(\un{q}_1))^2e^{G_1(\un{q}_1)}}{8h(\un{q}_j)e^{G_j(\un{q}_j)}}r^2\Big)-2+O(R^{-2})\Big) }\Bigg] \\
 & {+O(1)(\frac{|\log r|e^{-\lambda_{n,j}^{(1)}}}{(\lambda_{n,j}^{(1)})^2})  +O(re^{-\lambda_{n,j}^{(1)}})+
 o(e^{-\lambda_{n,j}^{(1)}})|\log R|} +O(1)\left(\frac{e^{-2\lambda_{n,j}^{(1)}}}{r^2}\right) \\& +O(1)
 \Bigg(\sum_{{l}=1}^m (|A_{n,{l}}|+e^{-\frac{\lambda_{n,j}^{(1)}}{2}})\Big(\frac{e^{-\frac{\lambda_{n,j}^{(1)}}{2}}}{R} +e^{-\lambda_{n,j}^{(1)}}
 (\lambda_{n,j}^{(1)}+|\log r|)\Big)\Bigg) \   {\textrm{for any}\   R>1,}      \end{aligned}
\end{equation*}
where $O(1)$ here is used to denote any quantity uniformly bounded with respect to {$r, R$} and $n$. \end{lemma}
\begin{proof} $(i)$  We first observe that \eqref{f_q_j} and \eqref{est_fn2} imply that,
\begin{equation}\label{r_poho0}\begin{aligned}
&\int_{\partial B_r(\un{x}_{n,j}^{(1)})} rf_n^*e^{2\varphi_j}\mathrm{d}\sigma
=\int_{\partial B_r(\un{x}_{n,j}^{(1)})}\frac{64 e^{-\lambda_{n,j}^{(1)}}(-b_0+o(1))e^{f_{\mathbf{q},j}+2\varphi_{j}}}
{\rho_n h(x_{n,j}^{(1)})|\un{x}-\un{x}_{n,j}^{(1)}|^3}\mathrm{d}\sigma.\end{aligned}\end{equation}
Since $f_{\mathbf{q},j}({\un{q}}_j)=0$, $\nabla f_{\mathbf{q},j}({\un{q}}_j)=0$, $\varphi_{j}(\un{x}_{n,j}^{(1)})=0$, $\nabla\varphi_{j}(\un{x}_{n,j}^{(1)})=0$,
and in view of \eqref{diffxandq}, we find that,
\begin{equation}
\label{expand}
\begin{aligned}
f_{\mathbf{q},j}(\un{x})+2\varphi_{j}(\un{x})&=
 \frac{1}{2}<D^2_{\un{x}}(f_{\mathbf{q},j}+2\varphi_{j})\Big|_{\un{x}=\un{x}_{n,j}^{(1)}}
(\un{x}-\un{x}_{n,j}^{(1)}),\un{x}-\un{x}_{n,j}^{(1)}>
+o(1)+O(|\un{x}-\un{x}_{n,j}^{(1)}|^3). \end{aligned}
\end{equation}
By \eqref{r_poho0} and \eqref{expand}, we obtain,
\begin{equation*}\label{r_poho}\begin{aligned}
&\int_{\partial B_r(\un{x}_{n,j}^{(1)})} rf_n^*e^{2\varphi_j}\mathrm{d}\sigma
 =\int_{\partial B_r(\un{x}_{n,j}^{(1)})}
-\frac{64 e^{-\lambda_{n,j}^{(1)}} }{\rho_n h(\un{x}_{n,j}^{(1)})
|\un{x}-\un{x}_{n,j}^{(1)}|^3}\\&\times \Big\{b_0(1+\frac{1}{2}<D^2_{\un{x}}(f_{\mathbf{q},j}+2\varphi_{j})\Big|_{\un{x}=\un{x}_{n,j}^{(1)}}
(\un{x}-\un{x}_{n,j}^{(1)}),\un{x}-\un{x}_{n,j}^{(1)}>) +O(|\un{x}-\un{x}_{n,j}^{(1)}|^3)+o(1)\Big\}\mathrm{d}\sigma
\\&{=\int_{\partial B_r(\un{x}_{n,j}^{(1)})}
- \frac{64 e^{-\lambda_{n,j}^{(1)}}b_0(1+\frac{\Delta(f_{\mathbf{q},j}+2\varphi_{j})(\un{x}_{n,j}^{(1)})}{4}|\un{x}-\un{x}_{n,j}^{(1)}|^2 )  }
{\rho_n h(\un{x}_{n,j}^{(1)})|\un{x}-\un{x}_{n,j}^{(1)}|^3}\mathrm{d}\sigma +O(r e^{-\lambda_{n,j}^{(1)}})+\frac{o(e^{-\lambda_{n,j}^{(1)}})}{r^2}}
\\& {=- \frac{128\pi e^{-\lambda_{n,j}^{(1)}}b_0(1+\frac{{\Delta(f_{\mathbf{q},j}+2\varphi_{j})(\un{x}_{n,j}^{(1)})}}{4}r^2)   }
{\rho_n h(\un{x}_{n,j}^{(1)})r^2} +O(r e^{-\lambda_{n,j}^{(1)}})+\frac{o(e^{-\lambda_{n,j}^{(1)}})}{r^2}}
\\& {=-\frac{ 128e^{-\lambda_{n,j}^{(1)}}b_0}{\rho_n h(\un{x}_{n,j}^{(1)})}
\frac{\pi}{r^2}
-\frac{32\pi e^{-\lambda_{n,j}^{(1)}}b_0   ({\Delta\log h(\un{q}_j)+8\pi m -2K(\un{q}_j)})    }{\rho_n h(\un{x}_{n,j}^{(1)})}}
 +O(r e^{-\lambda_{n,j}^{(1)}})+\frac{o(e^{-\lambda_{n,j}^{(1)}})}{r^2},
\end{aligned}
\end{equation*}
where, in the last identity, we used \eqref{delr1}, \eqref{delr} and \eqref{diffxandq}. By using \eqref{relation_lambda} and \eqref{diffxandq}, we find that,
\begin{equation*}\label{r_poho2}\begin{aligned}
& {-\frac{ 128e^{-\lambda_{n,j}^{(1)}}b_0}{\rho_n h(\un{x}_{n,j}^{(1)})}\frac{\pi}{r^2}-
\frac{32\pi e^{-\lambda_{n,j}^{(1)}}b_0   (\Delta\log h(\un{q}_j)+8\pi m -2K(\un{q}_j))}
{\rho_n h(\un{x}_{n,j}^{(1)})}}\\&
{=-\frac{32\pi e^{-\lambda_{n,1}^{(1)}}
b_0  h(\un{q}_j)e^{G_j^*(\un{q}_j)} (\Delta\log h(\un{q}_j)+8\pi m -2K(\un{q}_j))}{\rho_n (h(\un{q}_1))^2 e^{G_1^*(\un{q}_1)}}}{-\frac{ 128e^{-\lambda_{n,1}^{(1)}}b_0h(\un{q}_j)e^{G_j^*(\un{q}_j)}}{\rho_n (h(\un{q}_1))^2e^{G_1^*(\un{q}_1)}}
\frac{\pi}{r^2}+o(e^{-\lambda_{n,1}^{(1)}})},
\end{aligned}
\end{equation*}
which proves $(i)$.

$(ii)$ We  note that $\int_M f_n^*\mathrm{d} \mu=0$, {and thus,}
\begin{equation}\label{second_int0}\begin{aligned}&{\sum_{j=1}^m\int_{U^{\sscp M}_r(x_{n,j}^{(1)})}   f_n^* \mathrm{d} x
 =
-\sum_{j=1}^m\int_{M_j\setminus  U^{\sscp M}_{r}(x_{n,j}^{(1)})}   f_n^* \mathrm{d}\mu( x)}.
 \end{aligned}\end{equation}
 { By \eqref{est_fn2},} \eqref{relation_lambda}, \eqref{rhon8pi} and \eqref{diff_x_nj}, we see that 
\begingroup
\abovedisplayskip=3pt
\belowdisplayskip=3pt
\begin{equation}\label{second_int}\begin{aligned}
&{-\sum_{j=1}^m\int_{M_j\setminus  U^{\sscp M}_{r }({x_{n,j}^{(1)}})}   f_n^* \mathrm{d}\mu( x)}
 {= \sum_{j=1}^m\int_{M_j\setminus  U^{\sscp M}_{r}(x_{n,j}^{(1)})} \frac{64b_0 e^{-\lambda_{n,j}^{(1)}}}{\rho_n h(x_{n,j}^{(1)})}{e^{ \Phi_j(x,\mathbf{x}_n)}}  \mathrm{d}\mu(x)
+o(e^{-\lambda_{n,1}^{(1)}})}
\\& {= \sum_{j=1}^m\int_{M_j\setminus  U^{\sscp M}_{r}(x_{n,j}^{(1)})} \frac{64b_0 e^{-\lambda_{n,1}^{(1)}}h(x_{n,j}^{(1)})e^{G_j^*(x_{n,j}^{(1)})}}{\rho_n (h(x_{n,1}^{(1)}))^2e^{G_1^*(x_{n,1}^{(1)})} }{e^{ \Phi_j(x,\mathbf{x}_n)}}  \mathrm{d}\mu(x)
+o(e^{-\lambda_{n,1}^{(1)}})}
\\& {= \frac{64b_0 e^{-\lambda_{n,1}^{(1)}}}{\rho_n (h(q_1))^2e^{G_1^*(q_1)}} \sum_{j=1}^m
\int_{M_j\setminus  U^{\sscp M}_{r}(q_j)} h(q_j)e^{G_j^*(q_j)}
e^{ \Phi_j(x,\mathbf{q})}  \mathrm{d}\mu(x)
+o(e^{-\lambda_{n,1}^{(1)}}). }\end{aligned}\end{equation}
\endgroup
 {Clearly \eqref{second_int0} and \eqref{second_int} prove $(ii)$.}

\indent $(iii)$ Let us recall the definition of  $\phi_{n,j}$  and $G_j^*$ from  \eqref{phij} and \eqref{g*}.\\ By \eqref{gradatq} and \eqref{rhon}, we find that
$D(\log h_j+\phi_{n,j})(\un{x}_{n,j}^{(1)})=O(\lambda_{n,j}^{(1)} e^{-\lambda_{n,j}^{(1)}}),$
which readily implies that,
\begin{align*}
&D(\log h_j+\phi_{n,j})(\un{x})=D(\log h_j+\phi_{n,j})(\un{x}_{n,j}^{(1)})+< D^2(\log h_j+\phi_{n,j})(\un{x}_{n,j}^{(1)}),\un{x}-\un{x}_{n,j}^{(1)}>+O(|\un{x}-\un{x}_{n,j}^{(1)}|^2)\\
&=< D^2(\log h_j+\phi_{n,j})(\un{x}_{n,j}^{(1)}),\un{x}-\un{x}_{n,j}^{(1)}>+O(\lambda_{n,j}^{(1)}e^{-\lambda_{n,j}^{(1)}})+O(|\un{x}-\un{x}_{n,j}^{(1)}|^2).
\end{align*}
By using \eqref{eta}, \eqref{eta_ets} and the scaling $\un{x}=e^{-\frac{\lambda_{n,j}^{(1)}}{2}}z+\un{x}_{n,j}^{(1)}$, recalling the notation of $\overline{f}$ introduced before Lemma \ref{lem_equalb0}, we find that,
{\allowdisplaybreaks
\begin{align}
\label{est_grad3_1}
&\int_{B_r(\un{x}_{n,j}^{(1)})} f_n^*<D(\log h_j+\phi_{n,j}),\un{x}-\un{x}_{n,j}^{(1)}>e^{2\varphi_j}\mathrm{d} \un{x}
\nonumber\\&=\int_{B_r(\un{x}_{n,j}^{(1)})}\frac{\rho_n h_j(\un{x}_{n,j}^{(1)})e^{\lambda_{n,j}^{(1)}+\eta_{n,j}^{(1)} +
G_j^*(\un{x})-G_j^*(\un{x}_{n,j}^{(1)})+\log h_j(\un{x})-\log h_j(\un{x}_{n,j}^{(1)})} }
{(1+\frac{\rho_n h_j(\un{x}_{n,j}^{(1)})e^{\lambda_{n,j}^{(1)}}}{8}|\un{x}-\un{x}_{n,j,*}^{(1)}|^2)^2}
 \Big(\zeta_n {{-\frac{\|v_{n,j}^{(1)}-v_{n,j}^{(2)}\|_{L^{\infty}(M)}}{2}\zeta_n^2}+
O(\|v_{n,j}^{(1)}-v_{n,j}^{(2)}\|_{L^{\infty}(M)}^2)} \Big)\nonumber \\&\times <D(\log h_j+\phi_{n,j}),\un{x}-\un{x}_{n,j}^{(1)}> \mathrm{d} \un{x}
\nonumber\\&=\int_{B_r(\un{x}_{n,j}^{(1)})}\frac{\rho_n h_j(\un{x}_{n,j}^{(1)})e^{\lambda_{n,j}^{(1)}+\eta_{n,j}^{(1)} +G_j^*(\un{x})-G_j^*(\un{x}_{n,j}^{(1)})
+\log h_j(\un{x})-\log h_j(\un{x}_{n,j}^{(1)})}}
{(1+\frac{\rho_n h_j(\un{x}_{n,j}^{(1)})e^{\lambda_{n,j}^{(1)}}}{8}|\un{x}-\un{x}_{n,j,*}^{(1)}|^2)^2}
(\zeta_n {-\frac{\|\tilde{u}_n^{(1)}-\tilde{u}_n^{(2)}\|_{L^{\infty}(M)}}{2}\zeta_n^2} +
O(\|\tilde{u}_n^{(1)}-\tilde{u}_n^{(2)}\|_{L^{\infty}(M)}^2) )
\nonumber\\& \times \Big< < D^2(\log h_j+\phi_{n,j})(\un{x}_{n,j}^{(1)}), (\un{x}-\un{x}_{n,j}^{(1)})>+O(\lambda_{n,j}^{(1)}
e^{-\lambda_{n,j}^{(1)}})+O(|\un{x}-\un{x}_{n,j}^{(1)}|^2), \un{x}-\un{x}_{n,j}^{(1)}\Big> \mathrm{d} \un{x}
\nonumber\\&=\int_{B_{\Lambda_{n,j,r}^{+}}(0)}\frac{\rho_n h_j(\un{x}_{n,j}^{(1)}) }{(1+\frac{\rho_n h_j(\un{x}_{n,j}^{(1)})}{8}|z+
e^{\frac{\lambda_{n,j}^{(1)}}{2}}(\un{x}_{n,j}^{(1)}-\un{x}_{n,j,*}^{(1)})|^2)^2}
\nonumber \\& \times \Big[\overline{\zeta_n}
{{-\frac{\|\tilde{u}_n^{(1)}-\tilde{u}_n^{(2)}\|_{L^{\infty}(M)}}{2}(\overline{\zeta_n})^2}}
+O(|\overline{\eta_{n,j}^{(1)}}|)+O(e^{-\frac{\lambda_{n,j}^{(1)}}{2}}|z|)
 +{O(\|\tilde{u}_n^{(1)}-\tilde{u}_n^{(2)}\|_{L^{\infty}(M)}^2 )}\Big]
\nonumber\\& \times \Big<   <D^2(\log h_j+\phi_{n,j})(\un{x}_{n,j}^{(1)}), e^{-\frac{\lambda_{n,j}^{(1)}}{2}}z>
+O(\lambda_{n,j}^{(1)} e^{-\lambda_{n,j}^{(1)}})+O(e^{-\lambda_{n,j}^{(1)}}|z|^2),  e^{-\frac{\lambda_{n,j}^{(1)}}{2}}z\Big> \mathrm{d} z=:K_{n,j,r}.
\end{align}}
By using \eqref{est_grad3_1} together with \eqref{eta_ets}, {\eqref{diff_x_nj}}, and Lemma \ref{lem_diff_tilu}, we conclude that,
{\allowdisplaybreaks
\begin{align*}
K_{n,j,r}&=\int_{B_{\Lambda_{n,j,r}^{+}}(0)}\frac{\rho_n h_j(\un{x}_{n,j}^{(1)}) \Big(\overline{\zeta_n}
{{-\frac{\|\tilde{u}_n^{(1)}-\tilde{u}_n^{(2)}\|_{L^{\infty}(M)}}{2}(\overline{\zeta_n})^2}}
+{O(\frac{1}{(\lambda_{n,j}^{(1)})^2 })}+O(e^{-\frac{\lambda_{n,j}^{(1)}}{2}}|z|)\Big) }{(1+\frac{\rho_n h_j(\un{x}_{n,j}^{(1)})}{8}|z|^2)^2}\\
&  \times \Big< <  D^2(\log h_j+\phi_{n,j})(\un{x}_{n,j}^{(1)}), e^{-\frac{\lambda_{n,j}^{(1)}}{2}}z>+
O(\lambda_{n,j}^{(1)} e^{-\lambda_{n,j}^{(1)}})+O(e^{-\lambda_{n,j}^{(1)}}|z|^2),  e^{-\frac{\lambda_{n,j}^{(1)}}{2}}z\Big> \mathrm{d} z
\\&=\int_{B_{\Lambda_{n,j,r}^{+}}(0)}\frac{\rho_n h_j(\un{x}_{n,j}^{(1)})  (\overline{\zeta_n}
{{-\frac{\|\tilde{u}_n^{(1)}-\tilde{u}_n^{(2)}\|_{L^{\infty}(M)}}{2}(\overline{\zeta_n})^2}} )
 }{(1+\frac{\rho_n h_j(\un{x}_{n,j}^{(1)})}{8}|z|^2)^2}
   < D^2(\log h_j+\phi_{n,j})(\un{x}_{n,j}^{(1)})z,z>e^{-\lambda_{n,j}^{(1)}}\mathrm{d} z
\\& {+O(1)(\frac{|\log r|e^{-\lambda_{n,j}^{(1)}}}{(\lambda_{n,j}^{(1)})^2})+O(re^{-\lambda_{n,j}^{(1)}})+o(e^{-\lambda_{n,j}^{(1)}})},
\end{align*}}
Since $
\int\frac{(1-r^2)r^3}{(1+r^2)^3}\mathrm{d}r=\frac{1}{2}\left(\frac{-3r^2-2}{(1+r^2)^2}-\log(1+r^2)\right)+C,
$
then, for any fixed and large $R>0$, by Lemma \ref{lem_est_zetanj} and {Lemma \ref{lem_equalb0}}, we see that,
{\allowdisplaybreaks \begin{align*}
& \int_{B_{R}(0)}\frac{\rho_n h_j(\un{x}_{n,j}^{(1)})  (\overline{\zeta_n}
{{-\frac{\|\tilde{u}_n^{(1)}-\tilde{u}_n^{(2)}\|_{L^{\infty}(M)}}{2}(\overline{\zeta_n})^2}} )
 }{(1+\frac{\rho_n h_j(\un{x}_{n,j}^{(1)})}{8}|z|^2)^2}
< D^2(\log h_j+\phi_{n,j})(\un{x}_{n,j}^{(1)})z,z>e^{-\lambda_{n,j}^{(1)}}\mathrm{d} z\\
&=\int_{B_{R}(0)}\frac{\rho_n h_j(\un{x}_{n,j}^{(1)}) \Big(\sum_{k=0}^2b_{j,k}\psi_{j,k}+o(1)\Big)< D^2(\log h_j+\phi_{n,j})(\un{x}_{n,j}^{(1)})z,z>
e^{-\lambda_{n,j}^{(1)}} }{(1+\frac{\rho_n h_j(\un{x}_{n,j}^{(1)})}{8}|z|^2)^2}\mathrm{d}z\\
&=\int_{B_{R}(0)}\frac{\rho_n h_j(\un{x}_{n,j}^{(1)})b_{j,0}
\left(\frac{1-\frac{\rho_n h_j(\un{x}_{n,j}^{(1)})}{8}|z|^2}
{1+\frac{\rho_n h_j(\un{x}_{n,j}^{(1)})}{8}|z|^2}\right)
\Delta(\log h_j+\phi_{n,j})(\un{x}_{n,j}^{(1)})|z|^2e^{-\lambda_{n,j}^{(1)}} }{2(1+\frac{\rho_n h_j(\un{x}_{n,j}^{(1)})}{8}|z|^2)^2}
\mathrm{d} z+{o(e^{-\lambda_{n,j}^{(1)}})|\log R|}\\
&=\frac{64\pi b_0\Delta(\log h_j+\phi_{n,j})(\un{x}_{n,j}^{(1)})}{\rho_n h_j(\un{x}_{n,j}^{(1)})}e^{-\lambda_{n,j}^{(1)}}
\int_0^{\sqrt{\frac{\rho_n h_j(\un{x}_{n,j}^{(1)})}{8}}R}\frac{(1-s^2)s^3}{(1+s^2)^3}\mathrm{d} s+{o(e^{-\lambda_{n,j}^{(1)}})|\log R|}\\
&=\frac{32\pi b_0\Delta(\log h_j+\phi_{n,j})(\un{x}_{n,j}^{(1)})}{\rho_n h_j(\un{x}_{n,j}^{(1)})}e^{-\lambda_{n,j}^{(1)}}
 \Big(\frac{\frac{\rho_n h_j(\un{x}_{n,j}^{(1)})}{8}R^2(1+\frac{\rho_n h_j(\un{x}_{n,j}^{(1)})}{4}R^2)}
{(1+\frac{\rho_n h_j(\un{x}_{n,j}^{(1)})}{8}R^2)^2}-\log (1+\frac{\rho_n h_j(\un{x}_{n,j}^{(1)})}{8}R^2)\Big)\\
& +{o(e^{-\lambda_{n,j}^{(1)}})|\log R|}.
\end{align*}}
Next, let us observe that,
\begin{equation}
\label{int_ent2}
\int\frac{ r^3}{(1+r^2)^{{2}}}\mathrm{d}r=\frac{1}{2}\Big(\frac{1}{ 1+r^2 }+\log(1+r^2)\Big)+C.
\end{equation}
{In view of \eqref{bdd_of_Zeta}, we also see that if $|z|\ge R$, then it holds,}
\begin{equation}\label{second_part}
{\overline{\zeta_n}(z)=\int_M\zeta_n\mathrm{d}\mu+O(1)\Big(\sum_{{l}=1}^m (|A_{n,{l}}|+
e^{-\frac{\lambda_{n,j}^{(1)}}{2}})(\frac{e^{\frac{\lambda_{n,j}^{(1)}}{2}}}{|z|} +1)\Big),}\ \ \textrm{and thus}
\end{equation}
\begin{equation*}\label{second_part1}
({\overline{\zeta_n}(z))^2=\Big(\int_M\zeta_n\mathrm{d}\mu\Big)^2+O(1)
\Big(\sum_{{l}=1}^m (|A_{n,{l}}|+e^{-\frac{\lambda_{n,j}^{(1)}}{2}})(\frac{e^{\frac{\lambda_{n,j}^{(1)}}{2}}}{|z|} +1)\Big).}
\end{equation*}
{In view of {\eqref{second_part}}, we also find that,}
{\allowdisplaybreaks
\begin{align*}
& \int_{B_{\Lambda_{n,j,r}^{+}}(0)\setminus B_{R}(0)}\frac{\rho_n h_j(\un{x}_{n,j}^{(1)}) (\overline{\zeta_n}
{{-\frac{\|\tilde{u}_n^{(1)}-\tilde{u}_n^{(2)}\|_{L^{\infty}(M)}}{2}(\overline{\zeta_n})^2}} )
}{(1+\frac{\rho_n h_j(\un{x}_{n,j}^{(1)})}{8}|z|^2)^2}
  < D^2(\log h_j+\phi_{n,j})(\un{x}_{n,j}^{(1)})z,z>e^{-\lambda_{n,j}^{(1)}}\mathrm{d} z
\\&=\int_{B_{\Lambda_{n,j,r}^{+}}(0)\setminus B_{R}(0)}\frac{\rho_n h_j(\un{x}_{n,j}^{(1)})\left(\int_M\zeta_n\mathrm{d}\mu
{-\frac{\|\tilde{u}_n^{(1)}-\tilde{u}_n^{(2)}\|_{L^{\infty}(M)}}{2}(\int_M\zeta_n\mathrm{d}\mu)^2}\right)
}{(1+\frac{\rho_n h_j(\un{x}_{n,j}^{(1)})}{8}|z|^2)^2}
 < D^2(\log h_j+\phi_{n,j})(\un{x}_{n,j}^{(1)})z,z>e^{-\lambda_{n,j}^{(1)}}
\mathrm{d} z
\\&\quad{ +\int_{B_{\Lambda_{n,j,r}^{+}}(0)\setminus B_{R}(0)}
\frac{O(1)(\sum_{{l}=1}^m (|A_{n,{l}}|+e^{-\frac{\lambda_{n,j}^{(1)}}{2}})
(\frac{e^{\frac{\lambda_{n,j}^{(1)}}{2}}}{|z|} +1)) e^{-\lambda_{n,j}^{(1)}} |z|^2}{(1+\frac{\rho_n h_j(\un{x}_{n,j}^{(1)})}{8}|z|^2)^2}
\mathrm{d} z}
 =:{(I)+(II)}.
\end{align*}}
It is easy to see that,
\begin{equation*}
\begin{aligned}
\label{est_grad6_2}
{ (II)=O(1)(\sum_{{l}=1}^m (|A_{n,{l}}|+e^{-\frac{\lambda_{n,j}^{(1)}}{2}})
\Big(\frac{e^{-\frac{\lambda_{n,j}^{(1)}}{2}}}{R} +e^{-\lambda_{n,j}^{(1)}}(\lambda_{n,j}^{(1)}+|\log r|)\Big)}.
\end{aligned}
\end{equation*}
Since $\mathfrak{Z}_n:=\int_M\zeta_n\mathrm{d}\mu{-\frac{\|\tilde{u}_n^{(1)}-\tilde{u}_n^{(2)}\|_{L^{\infty}(M)}}{2}(\int_M\zeta_n\mathrm{d}\mu)^2}$ is constant, we also conclude {from \eqref{int_ent2}} that,
\begin{equation}\begin{aligned}\label{est_grad6}
(I)&= \frac{64\pi {\mathfrak{Z}_n}\Delta(\log h_j+\phi_{n,j})(\un{x}_{n,j}^{(1)})}{\rho_n h_j(\un{x}_{n,j}^{(1)})}e^{-\lambda_{n,j}^{(1)}}
\int_{\sqrt{\frac{\rho_n h_j(\un{x}_{n,j}^{(1)})}{8}}R}^{\sqrt{\frac{\rho_n h_j(\un{x}_{n,j}^{(1)})}{8}}
\Lambda_{n,j,r}^{+}}\frac{ s^3}{(1+s^2)^2}\mathrm{d} s
\\&= -\frac{32\pi {\mathfrak{Z}_n}\Delta(\log h_j+\phi_{n,j})(\un{x}_{n,j}^{(1)})}{\rho_n h_j(\un{x}_{n,j}^{(1)})}e^{-\lambda_{n,j}^{(1)}}
\\& \times \Big[\frac{1}{1+\frac{\rho_n h_j(\un{x}_{n,j}^{(1)})}{8}R^2}+\log(1+\frac{\rho_n h_j(\un{x}_{n,j}^{(1)})}{8}R^2)-
\log(1+\frac{\rho_n h_j(\un{x}_{n,j}^{(1)})}{8}r^2e^{\lambda_{n,j}^{(1)}})\Big] {+O(\frac{e^{-2\lambda_{n,j}^{(1)}}}{r^2})}
\\&=-\frac{32\pi {\mathfrak{Z}_n}\Delta(\log h_j+\phi_{n,j})(\un{x}_{n,j}^{(1)})}{\rho_n h_j(\un{x}_{n,j}^{(1)})}e^{-\lambda_{n,j}^{(1)}}
\\& \times \Big[\frac{1}{1+\frac{\rho_n h_j(\un{x}_{n,j}^{(1)})}{8}R^2}+\log(1+\frac{\rho_n h_j(\un{x}_{n,j}^{(1)})}{8}R^2)-
\lambda_{n,j}^{(1)}{-\log(\frac{\rho_n h_j(\un{x}_{n,j}^{(1)})}{8}r^2)}\Big] +{O(\frac{e^{-2\lambda_{n,j}^{(1)}}}{r^2}).}
\end{aligned}
\end{equation}
By Lemma \ref{lem_equalb0}, it is easy to check that,
\begin{equation}\label{intequalb0}
{\int_M\zeta_n\mathrm{d}\mu=-b_0+o(1).}
\end{equation}
From \eqref{est_grad3_1}-\eqref{est_grad6} {and \eqref{intequalb0}}, we obtain
{\allowdisplaybreaks
\begin{align*}
\label{est_grad3}
&\int_{B_r(x_{n,j}^{(1)})} f_n^*<D(\log h_j+\phi_{n,j}),\un{x}-\un{x}_{n,j}^{(1)}>e^{2\varphi_j}\mathrm{d} \un{x}
\\=~&{\frac{32\pi \mathfrak{Z}_n\Delta(\log h_j+\phi_{n,j})(\un{x}_{n,j}^{(1)})}{\rho_n h(\un{x}_{n,j}^{(1)})}
e^{-\lambda_{n,j}^{(1)}} \Big(\lambda_{n,j}^{(1)}+\log(\frac{\rho_n h_j(\un{x}_{n,j}^{(1)})}{8}r^2)-2+O(R^{-2})\Big) }
\\&{+O(1)(\frac{|\log r|e^{-\lambda_{n,j}^{(1)}}}{(\lambda_{n,j}^{(1)})^2})
+O(re^{-\lambda_{n,j}^{(1)}})+o(e^{-\lambda_{n,j}^{(1)}})|\log R|}
+O(1)\left(\frac{e^{-2\lambda_{n,j}^{(1)}}}{r^2}\right)\\&+O(1)
\Bigg(\sum_{{l}=1}^m (|A_{n,{l}}|+e^{-\frac{\lambda_{n,j}^{(1)}}{2}})\Big(\frac{e^{-\frac{\lambda_{n,j}^{(1)}}{2}}}{R} +
e^{-\lambda_{n,j}^{(1)}}(\lambda_{n,j}^{(1)}+|\log r|)\Big)\Bigg).
\end{align*}}
Finally, since \rife{delr1},\rife{delr}, \eqref{diffxandq} and \eqref{rhon} imply that,
\begin{equation*}
\label{est_grad3_22}
{\Delta(\log h_j+\phi_{n,j})(\un{x}_{n,j}^{(1)})=\Delta \log h(\un{q}_j) +8\pi m -2K(\un{q}_j)+{O( \lambda_{n,1}^{(1)}e^{-\lambda_{n,1}^{(1)}})},}
\end{equation*}
then \eqref{relation_lambda} and \eqref{diffxandq}, show that,
\begin{equation}
\begin{aligned}
\label{est_grad3_2}
&\int_{B_r(x_{n,j}^{(1)})} f_n^*<D(\log h_j+\phi_{n,j}),\un{x}-\un{x}_{n,j}^{(1)}>e^{2\varphi_j}\mathrm{d}\un{x}\\
&=\frac{32\pi \mathfrak{Z}_n(\Delta \log h(\un{q}_j) +8\pi m -2K(\un{q}_j))h(\un{q}_j)e^{G_{j}^*(\un{q}_j)}}{\rho_n (h(\un{q}_1))^2e^{G_{1}^*(\un{q}_1)}} e^{-\lambda_{n,1}^{(1)}}
 \Big(\lambda_{n,1}^{(1)}+\log\Big(\frac{\rho_n (h(\un{q}_1))^2e^{G_1(\un{q}_1)}}{8h(q_j)e^{G_j(\un{q}_j)}}r^2\Big)-2+O(R^{-2})\Big)\\
& {+O(1)(\frac{|\log r|e^{-\lambda_{n,j}^{(1)}}}{(\lambda_{n,j}^{(1)})^2})
+O(re^{-\lambda_{n,j}^{(1)}})+o(e^{-\lambda_{n,j}^{(1)}})|\log R|}   +O(1)\left(\frac{e^{-2\lambda_{n,j}^{(1)}}}{r^2}\right)
\\&+O(1)\Bigg(\sum_{{l}=1}^m (|A_{n,{l}}|+e^{-\frac{\lambda_{n,j}^{(1)}}{2}})\Big(\frac{e^{-\frac{\lambda_{n,j}^{(1)}}{2}}}{R}
+e^{-\lambda_{n,j}^{(1)}}(\lambda_{n,j}^{(1)}+|\log r|)\Big)\Bigg).    \end{aligned}
\end{equation}
The estimate \eqref{est_grad3_2} readily yields $(iii)$, as claimed. This fact concludes the proof of Lemma~\ref{poho1_rhs}.
\end{proof}

We can now show that $b_{j,0}=0$ for all $j=1,\cdots,m$.
\begin{lemma}\label{b0=0}$ $\\
$(i)$ $ A_{n,j}=\int_{M_j}f_n^*(y)\mathrm{d}\mu(y)= o(e^{-\frac{\lambda_{n,j}^{(1)}}{2}})$.\\
$(ii)$ $b_0=0$ and in particular $b_{j,0}=0$, $j=1,\cdots,m$.
\end{lemma}
\begin{proof}
$(i)$ By \eqref{poho1}, Lemmas \ref{poho1_lhs}-\ref{poho1_rhs} and \eqref{info_lambda}, we see that,
\begin{equation*}\label{alhs}\begin{aligned}
&-4A_{n,j}{+O(\lambda_{n,j}^{(1)}e^{-\lambda_{n,j}^{(1)}})}
  +o(e^{-\frac{\lambda_{n,j}^{(1)}}{2}}\sum_{l=1}^m|A_{n,l}|) {+o(e^{-\lambda_{n,j}^{(1)}})}
 ={ -2A_{n,j}+  O(\lambda_{n,j}^{(1)}e^{-\lambda_{n,j}^{(1)}})+ O(\frac{e^{-\lambda_{n,j}^{(1)}}}{r^2})+ o(e^{-\lambda_{n,j}^{(1)}}\log R)},
\end{aligned}\end{equation*}
which proves $(i)$.

\noindent $(ii)$  For any $r>0$, let \begin{equation}\label{rj1}
  r_j=r\sqrt{8h(q_j)G_j(q_j)}\ \ \textrm{for}\ \ j=1,\cdots,m.
\end{equation}By \eqref{poho1},
Lemmas \ref{poho1_lhs}-\ref{poho1_rhs} and $(i)$, we have for any $r\in (0,1)$ and $R>1$, 
{\allowdisplaybreaks
\begin{align*}
&\sum_{j=1}^m\Big[-4A_{n,j}{-\frac{256 b_0 e^{-\lambda_{n,1}^{(1)}}h(q_j)e^{G_j^*(q_j)}}{\rho_n (h(q_1))^2e^{G_1^*(q_1)}}
\int_{M_j\setminus U^{\sscp M}_{{{r_j}}}(q_j)} e^{ \Phi_j(y,\mathbf{q})} \mathrm{d}\mu(y)}
{+o(e^{-\frac{\lambda_{n,j}^{(1)}}{2}}\sum_{l=1}^m|A_{n,l}|) + o(e^{-\lambda_{n,j}^{(1)}})}\Big]
\\&= \sum_{j=1}^m\Big[{-\frac{ 128e^{-\lambda_{n,1}^{(1)}}b_0h(q_j)e^{G_j^*(q_j)}}{\rho_n (h(q_1))^2e^{G_1^*(q_1)}}
\frac{\pi}{r_j^2}}{-\frac{32\pi e^{-\lambda_{n,1}^{(1)}}b_0  h(q_j)e^{G_j^*(q_j)} (\Delta_{\sscp M}\log h(q_j)+8\pi m -2K(q_j))    }{\rho_n (h(q_1))^2 e^{G_1^*(q_1)}}}\\
& {-\frac{128b_0 e^{-\lambda_{n,1}^{(1)}}}{\rho_n (h(q_1))^2e^{G_1^*(q_1)}}
\int_{M_j\setminus  U^{\sscp M}_{{{{r_j}}}}(q_j)} h(q_j)e^{G_j^*(q_j)} e^{\Phi_j(x,\mathbf{q})}  \mathrm{d}\mu(x)}\\
& -\frac{32\pi \Big(\int_M\zeta_n\mathrm{d}\mu{-\frac{\|\tilde{u}_n^{(1)}-\tilde{u}_n^{(2)}\|_{L^{\infty}(M)}}{2}(\int_M\zeta_n\mathrm{d}\mu)^2}\Big)
}{\rho_n (h(q_1))^2e^{G_{1}^*(q_1)}}  (\Delta_{\sscp M} \log h(q_j) +8\pi m -2K(q_j))h(q_j)e^{G_{j}^*(q_j)}e^{-\lambda_{n,1}^{(1)}} \\
& {\times \Big(\lambda_{n,1}^{(1)}+\log\Big(\frac{\rho_n (h(q_1))^2e^{G_1(q_1)}}{8h(q_j)e^{G_j(q_j)}}{r^2_j}\Big)-2\Big) } +O(e^{-\lambda_{n,1}^{(1)}})(r+R^{-1})+o(e^{-\lambda_{n,1}^{(1)}})(\log r+\log R)     \Big],
\end{align*}}
\noindent where we used $O(1)$ to denote any quantity uniformly bounded with respect to {$r, R$} and $n$.
Then, in view of  $(i)$,  $\sum_{j=1}^m A_{n,j}=0$ and by the definition of $\ell(\mathbf{q})$ we see that,
\begingroup
\abovedisplayskip=3pt
\belowdisplayskip=3pt
\begin{equation}\label{after_poho2}\begin{aligned}
&{-\frac{256 b_0 e^{-\lambda_{n,1}^{(1)}}}{\rho_n (h(q_1))^2e^{G_1^*(q_1)}}
\sum_{j=1}^m h(q_j)e^{G_j^*(q_j)}\int_{M_j\setminus U^{\sscp M}_{{{r_j}}}(q_j)} {e^{ \Phi_j(y,\mathbf{q})}}
\mathrm{d}\mu(y)}
 \\=~&{-\frac{ 128e^{-\lambda_{n,1}^{(1)}}b_0}{\rho_n (h(q_1))^2e^{G_1^*(q_1)}}\sum_{j=1}^m h(q_j)e^{G_j^*(q_j)}
 {\frac{\pi}{r_j^2}}}{-\frac{32\pi e^{-\lambda_{n,1}^{(1)}}b_0  \ell(\mathbf{q})    }{\rho_n (h(q_1))^2 e^{G_1^*(q_1)}}}
  {-\frac{128b_0 e^{-\lambda_{n,1}^{(1)}}}{\rho_n (h(q_1))^2e^{G_1^*(q_1)}}
\sum_{j=1}^m h(q_j)e^{G_j^*(q_j)}\int_{M_j\setminus  U^{\sscp M}_{{{r_j}}}(q_j)} {e^{ \Phi_j(y,\mathbf{q})}}  \mathrm{d}\mu({y})
}
 \\&{-\frac{32\pi \Big(\int_M\zeta_n\mathrm{d}\mu
 {-\frac{\|\tilde{u}_n^{(1)}-\tilde{u}_n^{(2)}\|_{L^{\infty}(M)}}{2}(\int_M\zeta_n\mathrm{d}\mu)^2}\Big)
 \ell(\mathbf{q})e^{-\lambda_{n,1}^{(1)}}}{\rho_n (h(q_1))^2e^{G_{1}^*(q_1)}} }
{\Big(\lambda_{n,1}^{(1)}+\log\Big(\rho_n (h(q_1))^2e^{{G_1^*}(q_1)}r^2\Big)-2\Big) } \\&{+O(e^{-\lambda_{n,1}^{(1)}})(r+R^{-1})+o(e^{-\lambda_{n,1}^{(1)}})(\log r+\log R)}.    \end{aligned}
\end{equation}
\endgroup

At this point we consider two cases:

\noindent Case 1. $\ell(\mathbf{q})\neq0$.\\
{By \eqref{intequalb0}, \eqref{after_poho2} and in view of Lemma \ref{lem_diff_tilu} we see that,}
$
\ell(\mathbf{q}){(b_0+o(1))}=o(1).$
Therefore, since $\ell(\mathbf{q})\neq0$, then  $b_0=0$.

\noindent Case 2. $\ell(\mathbf{q})=0$ and $D(\mathbf{q})\neq0$.\\
{Since $\ell(\mathbf{q})=0$, then, by using the definition of $D(\mathbf{q})$ and \eqref{after_poho2}, we conclude that,}
\begin{equation*}
{(D(\mathbf{q})+o_r(1))b_0e^{-\lambda_{n,1}^{(1)}}= O(e^{-\lambda_{n,1}^{(1)}}) (r+R^{-1})+o(e^{-\lambda_{n,1}^{(1)}})(\log r+\log R).}
\end{equation*}
Since $r>0$ and $R>0$ are arbitrary, then $D(\mathbf{q})\neq 0$ implies $b_0=0$.

At this point, by $b_0=0$,   Lemma \ref{lem_equalb0} shows that $b_{j,0}=b_0=0$, $j=1,\cdots, m$. This fact concludes the proof of $(ii)$.
\end{proof}

Next we prove that $b_{j,1}=b_{j,2}=0$ by exploiting the following Pohozaev identity.
\begin{lemma}[\cite{ly2}]\label{lem_po2}  We have for $i=1,2$ and  fixed small $r>0$,
\begin{equation}\label{poho2}\begin{aligned}
&\int_{\partial B_r(\un{x}_{n,j}^{(1)})}<\nu, D\zeta_n>D_iv_{n,j}^{(1)} +
<\nu, Dv_{n,j}^{(2)}>D_i\zeta_n\mathrm{d}\sigma
-\frac{1}{2}\int_{\partial B_r(\un{x}_{n,j}^{(1)})}<D(v_{n,j}^{(1)}+v_{n,j}^{(2)}),D\zeta_n>
{\frac{(\un{x}-\un{x}_{n,j}^{(1)})_i}{|\un{x}-\un{x}_{n,j}^{(1)}|}}\mathrm{d}\sigma
\\=&-\int_{\partial B_r(\un{x}_{n,j}^{(1)})}\rho_nh_j(\un{x})\frac{e^{\tilde{u}_n^{(1)}}-e^{\tilde{u}_n^{(2)}}}
{\|\tilde{u}_n^{(1)}-\tilde{u}_n^{(2)}\|_{L^\infty(M)}}\frac{(\un{x}-\un{x}_{n,j}^{(1)})_i}{|\un{x}-\un{x}_{n,j}^{(1)}|}\mathrm{d}\sigma
+\int_{B_r(\un{x}_{n,j}^{(1)})}\rho_n h_j(\un{x})\frac{e^{\tilde{u}_n^{(1)}}-e^{\tilde{u}_n^{(2)}}}
{\|\tilde{u}_n^{(1)}-\tilde{u}_n^{(2)}\|_{L^\infty(M)}}D_i(\phi_{n,j}+\log h_j)\mathrm{d}\un{x}.
\end{aligned}\end{equation}
\end{lemma}
\begin{proof}
The identity
\eqref{poho2} has been obtained in \cite{ly2}. We prove it here for reader's convenience. We first observe that,
\begingroup
\abovedisplayskip=0pt
\belowdisplayskip=0pt
\begin{equation*}\label{eq_v3}
\Delta v_{n,j}^{(i)}+\rho_nh_j e^{\tilde{u}_n^{(i)}}=0\ \ \textrm{in}\ \ B_r(\un{x}_{n,j}^{(1)}),\ \ \textrm{and}
\end{equation*}
\endgroup
\begingroup
\abovedisplayskip=0pt
\belowdisplayskip=0pt
\begin{equation*}\label{poho2_st1}
  \begin{aligned}
&\textrm{div}\left(\nabla\zeta_nD_lv_{n,j}^{(i)}+\nabla v_{n,j}^{(i)} D_l\zeta_n-\nabla\zeta_n\cdot\nabla v_{n,j}^{(i)} e_l\right)
={\Delta} \zeta_n D_l v_{n,j}^{(i)}+{\Delta} v_{n,j}^{(i)} D_l\zeta_n
  \\&=\textrm{div}\Big(-\rho_nh_j(\un{x})\frac{e^{\tilde{u}_n^{(1)}}-e^{\tilde{u}_n^{(2)}}}{\|\tilde{u}_n^{(1)}-\tilde{u}_n^{(2)}\|_{L^\infty(M)}}e_l\Big)
  +\rho_nh_j(\un{x})\frac{e^{\tilde{u}_n^{(1)}}-e^{\tilde{u}_n^{(2)}}}
  {\|\tilde{u}_n^{(1)}-\tilde{u}_n^{(2)}\|_{L^\infty(M)}}
D_l((-1)^{i}(v_{n,j}^{(1)}-v_{n,j}^{(2)})+\phi_{n,j}+\log h_j),
  \end{aligned}
\end{equation*}
\endgroup
where $e_l=\frac{\un{x}_l}{|\un{x}|}$, $l=1,2$. Therefore we find that,
\begingroup
\abovedisplayskip=3pt
\belowdisplayskip=3pt
\begin{equation*}\label{poho2_st2}
  \begin{aligned}&\textrm{div}(-2\rho_n h_j({\un{x}})\frac{e^{\tilde{u}_n^{(1)}}-e^{\tilde{u}_n^{(2)}}}{\|\tilde{u}_n^{(1)}-\tilde{u}_n^{(2)}\|_{L^\infty(M)}}e_l)
+2\rho_n h_j({\un{x}})\frac{e^{\tilde{u}_n^{(1)}}-e^{\tilde{u}_n^{(2)}}}{\|\tilde{u}_n^{(1)}-\tilde{u}_n^{(2)}\|_{L^\infty(M)}}D_l(\phi_{n,j}+\log h_j)
  \\&=\textrm{div}\left\{2\nabla\zeta_nD_lv_{n,j}^{(1)}+2\nabla v_{n,j}^{(2)} D_l\zeta_n-\nabla\zeta_n\cdot\nabla(v_{n,j}^{(1)}+v_{n,j}^{(2)})e_l\right\},\end{aligned}
\end{equation*}
which proves Lemma \ref{lem_po2}.
\end{proof}
\endgroup

Next, we shall estimate the left and right hand side of the identity \eqref{poho2}.
\begin{lemma}\label{lem_poho2r}
 \begin{equation*}
 \mathrm{(RHS)~of  }\ \eqref{poho2}
  =\tilde{B}_j\Big( \sum_{h=1}^2D^2_{hi}(\phi_{n,j}+\log h_j)(\un{x}_{n,j}^{(1)})
  e^{-\frac{\lambda_{n,j}^{(1)}}{2}}b_{j,h} \Big)+o(e^{-\frac{\lambda_{n,j}^{(1)}}{2}}),\   \textrm{where} \  \tilde{B}_j=
 4\sqrt{\frac{8}{\rho_n  h_j(\un{x}_{n,j}^{(1)})}}\int_{\mathbb{R}^2}\frac{|z|^2}{(1+|z|^2)^3}\mathrm{d} z.
\end{equation*}
 \end{lemma}
 \begin{proof}
First of all, in view of \eqref{info_lambda}, we find that,
\begingroup
\abovedisplayskip=3pt
\belowdisplayskip=3pt
\begin{equation}\label{first_term}\begin{aligned}
                                         &\int_{\partial B_r(\un{x}_{n,j}^{(1)})}\rho_n h_j(\un{x})\frac{e^{\tilde{u}_n^{(1)}}-e^{\tilde{u}_n^{(2)}}}
                                         {\|\tilde{u}_n^{(1)}-\tilde{u}_n^{(2)}\|_{L^\infty(M)}}\frac{(\un{x}-\un{x}_{n,j}^{(1)})_i}{|\un{x}-\un{x}_{n,j}^{(1)}|}
                                         \mathrm{d}\sigma
                                      = \int_{\partial B_r(\un{x}_{n,j}^{(1)})}\rho_n h_j(\un{x})e^{\tilde{u}_n^{(1)}}(\zeta_n+o(1))\frac{(\un{x}-\un{x}_{n,j}^{(1)})_i}
                                         {|\un{x}-\un{x}_{n,j}^{(1)}|}\mathrm{d}\sigma
                                         =O(e^{-\lambda_{n,j}^{(1)}}).
                                       \end{aligned}\end{equation}
\endgroup
while by \eqref{eta}, we also see that,
\begingroup
\abovedisplayskip=3pt
\belowdisplayskip=3pt
\begin{equation}\label{sec_term0}
                     \begin{aligned}&\int_{B_r(\un{x}_{n,j}^{(1)})}\frac{\rho_n  h_j(\un{x})(e^{\tilde{u}_n^{(1)}}-e^{\tilde{u}_n^{(2)}})}
                     {\|\tilde{u}_n^{(1)}-\tilde{u}_n^{(2)}\|_{L^\infty(M)}}D_i(\phi_{n,j}+\log h_j)\mathrm{d} \un{x}
=\int_{B_r(\un{x}_{n,j}^{(1)})}\rho_n  h_j(\un{x})e^{\tilde{u}_n^{(1)}}(\zeta_n+o(1))D_i(\phi_{n,j}+\log h_j)\mathrm{d} \un{x}
\\&=\int_{B_r(\un{x}_{n,j}^{(1)})}\frac{\rho_n  h_j(\un{x})e^{\lambda_{n,j}^{(1)}+\eta_{n,j}^{(1)}+
G_j^*(\un{x})-G_j^*(\un{x}_{n,j}^{(1)})}}{(1+\frac{\rho_n h_j(\un{x}_{n,j}^{(1)})}{8}
e^{\lambda_{n,j}^{(1)}}|\un{x}-\un{x}_{n,j,*}^{(1)}|^2)^2}(\zeta_n+o(1))
\\&\quad\times \Big[D_i(\phi_{n,j}+\log h_j)(\un{x}_{n,j}^{(1)})+
\sum_{h=1}^2 D_{ih}^2(\phi_{n,j}+\log h_j)(\un{x}_{n,j}^{(1)})(\un{x}-\un{x}_{n,j}^{(1)})_h + O(|\un{x}-\un{x}_{n,j}^{(1)}|^2)\Big]\mathrm{d} \un{x}
\\&=\int_{B_{\Lambda_{n,j,r}^{+}}(0)}
\frac{\rho_n h_j(e^{-\frac{\lambda_{n,j}^{(1)}}{2}}z+\un{x}_{n,j}^{(1)})e^{\overline{\eta_{n,j}^{(1)}}+
G_j^*(e^{-\frac{\lambda_{n,j}^{(1)}}{2}}z+\un{x}_{n,j}^{(1)})-G_j^*(\un{x}_{n,j}^{(1)})}}
{(1+\frac{\rho_n h_j(\un{x}_{n,j}^{(1)})}{8} |z+ e^{\frac{\lambda_{n,j}^{(1)}}{2}}(\un{x}_{n,j}^{(1)}-\un{x}_{n,j,*}^{(1)})|^2)^2}
(\overline{\zeta_n}+o(1))
\\&\quad\times\Big[D_i(\phi_{n,j}+\log h_j)(\un{x}_{n,j}^{(1)})+\sum_{h=1}^2 D_{ih}^2(\phi_{n,j}+\log h_j)(\un{x}_{n,j}^{(1)})z_h
e^{-\frac{\lambda_{n,j}^{(1)}}{2}} + O(e^{-\lambda_{n,j}^{(1)}}|z|^2)\Big]{\mathrm{d}  z.}
                                        \end{aligned}
                                    \end{equation}
\endgroup

{Next, since} $\mathbf{q}$ is a critical point of $f_m$, then by using \eqref{rhon},  \eqref{rhon8pi}, and \eqref{diffxandq},  we find that,
\begingroup
\abovedisplayskip=3pt
\belowdisplayskip=3pt
                   \begin{equation}\label{deri_phi_h}
                     \begin{aligned}
                    & D_i(\phi_{n,j}+\log h_j)(\un{x}_{n,j}^{(1)})=D_i(G_j^*+\log h_j)(\un{x}_{n,j}^{(1)}) +O(\lambda_{n,j}^{(1)} e^{-\lambda_{n,j}^{(1)}})
                   =O(\lambda_{n,j}^{(1)} e^{-\lambda_{n,j}^{(1)}}).
                      \end{aligned}
                   \end{equation}
   \endgroup

By using \eqref{sec_term0},  \eqref{deri_phi_h}, and \eqref{eta_ets},  we have,
                   \begingroup
\abovedisplayskip=3pt
\belowdisplayskip=3pt
\begin{equation*}\label{sec_term1}
                     \begin{aligned}&\int_{B_r(\un{x}_{n,j}^{(1)})}\frac{\rho_n  h_j(\un{x})(e^{\tilde{u}_n^{(1)}}-e^{\tilde{u}_n^{(2)}})}
                     {\|\tilde{u}_n^{(1)}-\tilde{u}_n^{(2)}\|_{L^\infty(M)}}D_i(\phi_{n,j}+\log h_j)\mathrm{d} \un{x}
     \\&=\int_{B_{\Lambda_{n,j,r}^{+}}(0)}\frac{\rho_n  h_j(\un{x}_{n,j}^{(1)}) \overline{\zeta_n} }
     {(1+\frac{\rho_n h_j(\un{x}_{n,j}^{(1)})}{8} |z|^2)^2}
     \Big( \sum_{h=1}^2D^2_{hi}(\phi_{n,j}+\log h_j)(\un{x}_{n,j}^{(1)})e^{-\frac{\lambda_{n,j}^{(1)}}{2}}z_h \Big)\mathrm{d} z +
     o(e^{-\frac{\lambda_{n,j}^{(1)}}{2}}){,}
                                     \end{aligned}
                   \end{equation*}
                   \endgroup
                   and then, in view of Lemma \ref{lem_est_zetanj} and Lemma \ref{b0=0}, we conclude that,
                   \begingroup
\abovedisplayskip=3pt
\belowdisplayskip=3pt
                   \begin{equation}\label{sec_term2}
                     \begin{aligned}&\int_{B_r(\un{x}_{n,j}^{(1)})}\frac{\rho_n  h_j(\un{x})(e^{\tilde{u}_n^{(1)}}-e^{\tilde{u}_n^{(2)}})}
                     {\|\tilde{u}_n^{(1)}-\tilde{u}_n^{(2)}\|_{L^\infty(M)}}D_i(\phi_{n,j}+\log h_j)\mathrm{d} \un{x}
     \\&=\int_{B_{\Lambda_{n,j,r}^{+}}(0)}\frac{\rho_n  h_j(\un{x}_{n,j}^{(1)}) \sqrt{\frac{\rho_n  h_j(\un{x}_{n,j}^{(1)})}{8}}|z|^2}
     {2(1+\frac{\rho_n h_j(\un{x}_{n,j}^{(1)})}{8} |z|^2)^3}\mathrm{d} z\Big( \sum_{h=1}^2D^2_{hi}(\phi_{n,j}+\log h_j)(\un{x}_{n,j}^{(1)})
     e^{-\frac{\lambda_{n,j}^{(1)}}{2}}b_{j,h} \Big) +o(e^{-\frac{\lambda_{n,j}^{(1)}}{2}})
     \\&=4\sqrt{\frac{8}{\rho_n  h_j(\un{x}_{n,j}^{(1)})}}\int_{\mathbb{R}^2}\frac{|z|^2}{(1+|z|^2)^3}
                            \mathrm{d} z\Big( \sum_{h=1}^2D^2_{hi}(\phi_{n,j}+\log h_j)(\un{x}_{n,j}^{(1)})
                            e^{-\frac{\lambda_{n,j}^{(1)}}{2}}b_{j,h} \Big) +o(e^{-\frac{\lambda_{n,j}^{(1)}}{2}})
                            \\&=\tilde{B}_j\Big( \sum_{h=1}^2D^2_{hi}(\phi_{n,j}+\log y h_j)(x_{n,j}^{(1)})
                            e^{-\frac{\lambda_{n,j}^{(1)}}{2}}b_{j,h} \Big)+o(e^{-\frac{\lambda_{n,j}^{(1)}}{2}}).
     \end{aligned}
     \end{equation}
                   \endgroup
Clearly \eqref{first_term} and \eqref{sec_term2} conclude the proof of Lemma \ref{lem_poho2r}.
 \end{proof}

 \begin{lemma}\label{lem_poho2l} \begin{equation*}\mathrm{(LHS)~of }\ \eqref{poho2}
   =-8\pi\Bigg[\sum_{l\neq j} e^{-\frac{\lambda_{n,l}^{(1)}}{2}}D_i G_{n,l}^*(\un{x}_{n,j}^{(1)})+  e^{-\frac{\lambda_{n,j}^{(1)}}{2}}D_i \sum_{h=1}^2\partial_{\un{y}_h}R(\un{y},\un{x})\Big|_{\un{x}=\un{y}=\un{x}_{n,j}^{(1)}}b_{j,h}\bar{B}_j\Bigg] {+o(e^{-\frac{\lambda_{n,j}^{(1)}}{2}})},
\end{equation*}where $ G_{n,l}^*(\un{x})=\sum_{h=1}^2{\partial_{\un{y}_h}} G(\un{y},\un{x})\Big|_{\un{y}=\un{x}_{n,l}^{(1)}}b_{l,h}\tilde{B}_l.$
 \end{lemma}
 \begin{proof}By the definition of $G_{n,i}^*$, we have for any $\theta\in(0,r)$,
$   \Delta G_{n,l}^*=0\ \ \textrm{in} \ \ B_r(\un{x}_{n,j}^{(1)})\setminus B_{\theta}(\un{x}_{n,j}^{(1)}).$\\
 Then for $\un{x}\in B_r(\un{x}_{n,j}^{(1)})\setminus B_{\theta}(\un{x}_{n,j}^{(1)})$, and setting $e_i=\frac{\un{x}_i}{|\un{x}|}$, $i=1,2$, we have,
 \begin{equation*}\label{feq1}\begin{aligned}
   0&= \Delta G_{n,l}^*D_i\log\frac{1}{|\un{x}-\un{x}_{n,j}^{(1)}|}+\Delta \log\frac{1}{|\un{x}-\un{x}_{n,j}^{(1)}|}D_iG_{n,l}^*
  \\&=\textrm{div}\Bigg(\nabla G_{n,l}^*D_i\log\frac{1}{|\un{x}-\un{x}_{n,j}^{(1)}|}+\nabla \log\frac{1}{|\un{x}-\un{x}_{n,j}^{(1)}|}
  D_iG_{n,l}^* -\nabla G_{n,l}^*\cdot\nabla \log\frac{1}{|\un{x}-\un{x}_{n,j}^{(1)}|}e_i\Bigg),\end{aligned}
 \end{equation*}
 which readily implies that,
 \begin{equation}\label{feq2}\begin{aligned}
 \int_{\partial B_r(\un{x}_{n,j}^{(1)})} \frac{\nabla_i G_{n,l}^*}{|\un{x}-\un{x}_{n,j}^{(1)}|} \mathrm{d}\sigma
&=\int_{\partial B_\theta(\un{x}_{n,j}^{(1)})} \frac{\nabla_i G_{n,l}^*}{|\un{x}-\un{x}_{n,j}^{(1)}|} \mathrm{d}\sigma.\end{aligned}
 \end{equation}
In view of  Lemma \ref{lem_out} and Lemma \ref{b0=0}, we also have,
\begin{equation}\label{zeta_est1}
  \zeta_n(x)-\int_M\zeta_n\mathrm{d}\mu=
  \sum_{l=1}^m e^{-\frac{\lambda_{n,l}^{(1)}}{2}}G_{n,l}^*(x) +o(e^{-\frac{\lambda_{n,j}^{(1)}}{2}}) \ \ \textrm{in} \ \
  C^1(U^{\sscp M}_r(x_{n,j}^{(1)})\setminus U^{\sscp M}_{\theta}(x_{n,j}^{(1)})).
\end{equation}

By using \eqref{rhon8pi}-\eqref{rhon}, we find
$\frac{\rho_n}{m}=\rho_{n,j}+O(\lambda_{n,j}^{(1)}e^{-\lambda_{n,j}^{(1)}})=8\pi+O(\lambda_{n,j}^{(1)}e^{-\lambda_{n,j}^{(1)}})$,
which, together with \eqref{wnx}, \eqref{est_wn}, and \eqref{diff_x1_x_2}, implies that,
\begin{equation}\label{dvn}
       \begin{aligned}Dv_{n,j}^{(i)}(\un{x})&=
       D(\tilde{u}_n^{(i)}-\phi_{n,j}) =D\Big(\tilde{u}_n^{(i)}-\frac{\rho_n}{m}R(\un{x},\un{x}_{n,j}^{(1)})-
       \frac{\rho_n}{m}\sum_{l\neq j} G(\un{x},\un{x}_{n,j}^{(1)})\Big)
       \\&=D\Big(\tilde{u}_n^{(i)}-{\frac{\rho_n}{m}\sum_{l=1}^m G(\un{x},\un{x}_{n,l}^{(1)})-\frac{\rho_n}{2\pi m}
       \log|\underline{x}-\underline{x}_{n,j}^{(1)}|}\Big)
   =D w_{n}^{(i)}-4\frac{(\un{x}-\un{x}_{n,j}^{(1)})}{|\un{x}-\un{x}_{n,j}^{(1)}|^2}+o(e^{-\frac{\lambda_{n,j}^{(1)}}{2}})
       \\&=-4\frac{(\un{x}-\un{x}_{n,j}^{(1)})}{|\un{x}-\un{x}_{n,j}^{(1)}|^2}+
       o(e^{-\frac{\lambda_{n,j}^{(1)}}{2}}) \ \ \textrm{in} \ \ C^1(B_r(\un{x}_{n,j}^{(1)})\setminus B_{\theta}(\un{x}_{n,j}^{(1)})).\end{aligned}
     \end{equation}
Next, since $
 { D_iD_h(\log |z|)=\frac{\delta_{ih}|z|^2-2z_iz_h}{|z|^4},}
$ we see that, \begin{equation*}\begin{aligned}\label{feq3}
              & \int_{\partial B_{\theta}(\un{x}_{n,j}^{(1)})}\frac{\nabla_i G_{n,j}^*}{|\un{x}-\un{x}_{n,j}^{(1)}|}\mathrm{d}\sigma =2\pi D_i \sum_{h=1}^2\partial_{\un{y}_h}R(\un{y},\un{x})\Big|_{\un{x}=\un{y}=\un{x}_{n,j}^{(1)}}b_{j,h}\bar{B}_j+o_\theta(1),
             \end{aligned}\end{equation*}
    which, together with \eqref{zeta_est1}, \eqref{dvn}, and \eqref{feq2}, implies that, for any $\theta\in(0,r)$,
     \begin{equation*}\label{lsh_second_poho}
       \begin{aligned}
   &\textrm{(LHS) of }\ \eqref{poho2}  \\&=-4\int_{\partial B_r(\un{x}_{n,j}^{(1)})}\sum_{l=1}^me^{-\frac{\lambda_{n,l}^{(1)}}{2}}\frac{\nabla_i G_{n,l}^*}
     {|\un{x}-\un{x}_{n,j}^{(1)}|}\mathrm{d}\sigma+o(e^{-\frac{\lambda_{n,j}^{(1)}}{2}})
   \\& =-8\pi\Bigg[\sum_{l\neq j} e^{-\frac{\lambda_{n,l}^{(1)}}{2}}D_i G_{n,l}^*(\un{x}_{n,j}^{(1)})+  e^{-\frac{\lambda_{n,j}^{(1)}}{2}}D_i \sum_{h=1}^2\partial_{\un{y}_h}R(\un{y},\un{x})\Big|_{\un{x}=\un{y}=\un{x}_{n,j}^{(1)}}b_{j,h}\bar{B}_j\Bigg] +
     o(e^{-\frac{\lambda_{n,j}^{(1)}}{2}})+o_\theta(1) e^{-\frac{\lambda_{n,j}^{(1)}}{2}},
       \end{aligned}
     \end{equation*}
     which proves Lemma \ref{lem_poho2l}.
 \end{proof}

Finally, we have the following,
\begin{lemma}\label{bj=0}$b_{j,1}=b_{j,2}=0$, $j=1,\cdots, m$.
\end{lemma}
\begin{proof} Obviously Lemma \ref{lem_po2},  Lemma \ref{lem_poho2r}, and Lemma \ref{lem_poho2l} together imply, for $i=1,2$,
\begin{equation}\label{to_show_bjh}\begin{aligned}
 &\tilde{B}_j\sum_{h=1}^2(D_{hi}^2(\phi_{n,j}+\log h_j)(\un{x}_{n,j}^{(1)})b_{j,h})e^{-\frac{\lambda_{n,j}^{(1)}}{2}}
\\ &=
  -8\pi\Bigg[\sum_{l\neq j} e^{-\frac{\lambda_{n,l}^{(1)}}{2}}D_i G_{n,l}^*(\un{x}_{n,j}^{(1)})+  e^{-\frac{\lambda_{n,j}^{(1)}}{2}}D_i \sum_{h=1}^2\partial_{\un{y}_h}R(\un{y},\un{x})\Big|_{\un{x}=\un{y}=\un{x}_{n,j}^{(1)}}b_{j,h}\bar{B}_j\Bigg] +
  o(e^{-\frac{\lambda_{n,j}^{(1)}}{2}})
\\&=-8\pi\sum_{l\neq j} e^{-\frac{\lambda_{n,l}^{(1)}}{2}}
\sum_{h=1}^2D_{\un{x}_i}\partial _{\un{y}_h } G(\un{y},\un{x})\Big|_{(\un{y},\un{x})=(\un{x}_{n,j}^{(1)},\un{x}_{n,l}^{(1)})}b_{lh}
\tilde{B}_l-8\pi e^{-\frac{\lambda_{n,j}^{(1)}}{2}}\sum_{h=1}^2 D_{\un{x}_i} \partial_{\un{y}_h}R(\un{y},\un{x})\Big|_{\un{x}=\un{y}=\un{x}_{n,j}^{(1)}}b_{j,h}\bar{B}_j+o(e^{-\frac{\lambda_{n,j}^{(1)}}{2}}).
\end{aligned}\end{equation}

Set $\vec{b}=({\hat{b}}_{1,1}\tilde{B}_1,{\hat{b}}_{1,2}\tilde{B}_1,\cdots,{\hat{b}}_{m,1}\tilde{B}_m,{\hat{b}}_{m,2}\tilde{B}_m)$
{, where $\hat{b}_{lh}=\lim_{n\to+\infty}(e^{ \frac{\lambda_{n,j}^{(1)}-\lambda_{n,l}^{(1)}}{2}}b_{lh})$. Then,
by using \eqref{diffxandq}} and
passing to the limit as $n\to +\infty$, we conclude from \rife{to_show_bjh} that,
$
D^2f_m(\un{q}_1,\un{q}_2,\cdots,\un{q}_m) \cdot \vec{b}= 0,
$
where $f_m(\un{x}_1, \cdots,\un{x}_m)$ is a suitably defined local expression of $f_m(x_1, \cdots,x_m)$.
By using the fact that the rank of isothermal maps is always maximum,
{together with the non degeneracy assumption $\textrm{det}(D_{\sscp M}^2f_m(\mathbf{q}))\neq 0$, we conclude that}
$b_{j,1}=b_{j,2}=0$, $j=1,\cdots, m$. \end{proof}

\noindent {\em Proof of Theorem \ref{thm_concen}.} Let $x_n^*$ be a maximum point of $\zeta_n$, then we have,
\begin{equation}\label{asymp0}
   |\zeta_n(x_n^*)|=1.
 \end{equation}    Therefore, in view of Lemma \ref{lem_equalb0} and Lemma \ref{b0=0}, we find that,
$
   \lim_{n\to+\infty}x_n^*=q_j,
$
 {for some $j$}. Moreover, by Lemma \ref{b0=0} and Lemma \ref{bj=0}, {it holds}
  \begin{equation}\label{asymp2}\lim_{n\to+\infty}e^{\frac{\lambda_{n,j}^{(1)}}{2}}s_n=+\infty, \ \textrm{where} \ \ s_n=|\un{x}_n^*-\un{x}_{n,j}^{(1)}|.
 \end{equation}
 Setting $\tilde{\zeta}_n(\un{x})=\zeta_n(s_n \un{x}+\un{x}_{n,j}^{(1)})$, then \eqref{eq_zeta} and \eqref{eta_ets} imply that $\tilde{\zeta}_n$ satisfies,
 \begin{equation*}\label{eq_tilzeta}\begin{aligned}
  0&= \Delta \tilde{\zeta}_n+\rho_n s_n^2 h(s_n \un{x}+\un{x}_{n,j}^{(1)})c_n(s_n \un{x}+\un{x}_{n,j}^{(1)})\tilde{\zeta}_n
  =\Delta \tilde{\zeta}_n+ \frac{\rho_n h(\un{x}_{n,j}^{(1)}) s_n^2e^{\lambda_{n,j}^{(1)}}(1+ O(s_n|\un{x}|)+o(1))\tilde{\zeta}_n}
  {(1+ \frac{\rho_nh(\un{x}_{n,j}^{(1)})}{8}e^{\lambda_{n,j}^{(1)}}|s_n\un{x}+\un{x}_{n,j}^{(1)}-\un{x}_{n,j,*}^{(1)}|^2)^2}.
\end{aligned} \end{equation*}
On the other side, by \eqref{asymp0}, we also have,
\begin{equation}\label{tilzeta0}\Big|\tilde{\zeta}_n\Big(\frac{\un{x}_n^*-\un{x}_{n,j}^{(1)}}{s_n}\Big)\Big|=|\zeta_n(\un{x}_n^*)|=1.\end{equation}
In view of \eqref{asymp2} and $|\tilde{\zeta}_n|\le 1$ we see that $\tilde{\zeta}_n\to
\tilde{\zeta}_0$ on any compact subset of $\mathbb{R}^2\setminus\{0\}$, where   $\tilde{\zeta}_0$ satisfies $\Delta \tilde{\zeta}_0=0$
in $\mathbb{R}^2\setminus\{0\}$. Since $|\tilde{\zeta}_0|\le 1$, we have $\Delta \tilde{\zeta}_0=0$ in $\mathbb{R}^2$, whence
$\tilde{\zeta}_0$ is a constant.
At this point, since $ \frac{|\un{x}_n^*-\un{x}_{n,j}^{(1)}|}{s_n}=1$ and in view of \eqref{tilzeta0}, we find that,
$\tilde{\zeta}_0\equiv1$ or $\tilde{\zeta}_0\equiv-1$. As a consequence we conclude that,
$
  |\zeta_n(\un{x})|\ge \frac{1}{2}\ \ \textrm{if}\ \ \ s_n\le |\un{x}-\un{x}_{n,j}^{(1)}|\le 2s_n,
$
which contradicts \eqref{outside_zeta}, \eqref{result3}, and \eqref{result4} since
$e^{-\frac{\lambda_{n,j}^{(1)}}{2}}\ll s_n$, $\lim_{n\to+\infty}s_n=0$, and $b_0=b_{j,0}=0$.
This fact concludes the proof of Theorem \ref{thm_concen}.\hfill$\Box$

\section{The Dirichlet problem}\label{sec_Dir}
Let $\om$ be an open and bounded two dimensional domain, $\om\subset \mathbb{R}^2$.
As in \cite{CCL}, we say that $\om$ is \underline{regular} if its boundary $\partial \om$ is of class $C^2$ but
for a finite number of points $\{Q_1 , . . . , Q_{N_0}\}\subset \partial \om$ such that the following conditions holds at each $Q_j$.\\
(i) The inner angle $\theta_j$ of $\partial \om$ at $Q_j$ satisfies $0 < {\theta_j \neq \pi} < 2\pi$;\\
(ii) At each $Q_j$ there is an univalent conformal map from $B_\delta (Q_j) \cap \overline{\om}$ to
the complex plane $\mathbb{C}$ such that $\partial \om \cap B_\delta (Q_j)$ is mapped to a $C^2$ curve.\\
Obviously any non degenerate polygon is regular according to this definition.

In this section we are concerned with the uniqueness result for the mean field equation with Dirichlet boundary conditions on regular domains,
\begin{equation}
\label{mf}
\Delta u_n+\rho_n\frac{h(x)e^{u_n(x)}}{\int_\om h e^{u_n }\mathrm{d}x }=0\ \ \textrm{in}\ \ \Omega,\quad  u_n=0\ \ \textrm{on}\ \ \partial\Omega,
\end{equation} where $h(x)=h_0(x)e^{-4\pi\sum_{j=1}^N\alpha_j G_{\sscp \om}(x,p_j)}\ge 0$, $p_j$
are   distinct points in $\om$,  $\alpha_j>-1$,  $h_0>0$, $h_0\in C^{2,\sigma}(\overline{\om})$,  and $G_{\sscp \om}$
is the Green function uniquely defined as follows,
$
-\Delta G_{\sscp \om}(x,p)=\delta_p~\mathrm{in}~\om,~\ \ ~G_{\sscp \om}(x,p)=0~\mathrm{on}~\partial\om.
$

\begin{definition}
Let $u_n$ be a sequence of solutions of  \eqref{mf}. We say that $u_n$ blows up at the points $q_j\notin\{p_1,\cdots,p_N\}$, $j=1,\cdots,m$, if,
$
\frac{h(x)e^{u_n(x)}}{\int_\om h e^{u_n }\mathrm{d}x }\rightharpoonup 8\pi\sum\limits_{j=1}^m\delta_{q_j},
$
weakly in the sense of measure in $\om$.
\end{definition}

Let $R_{\sscp \om}(x,y)=\frac{1}{2\pi}\log |x-y|+G_{\sscp \om}(x,y),
$ be the regular part of $G_{\sscp \om}(x,y)$. For $\mathbf{q}=(q_1,\cdots,q_m)\in \om\times\cdots\times \om$, we denote by,
$
G_{j, {\sscp \om}}^*(x)=8\pi R_{\sscp \om}(x,q_j)+8\pi\sum^{1,\cdots,m}_{l\neq j}G_{\sscp \om}(x,q_l),\ \ \textrm{and}
\ \ \ell_{\sscp \om}(\mathbf{q})=\sum_{j=1}^m[\Delta \log h(q_j)]h(q_j)e^{G_{j, {\sscp \om}}^*(q_j)}.
$

If $m\geq 2$, let us fix a constant $r_0\in(0,\frac{1}{2})$, and a family of open sets $\om_j$ satisfying, $\om_l\cap \om_j=\emptyset$
if $l\neq j$, $\bigcup_{j=1}^m\overline{\om}_j=\overline{\om}$, $B_{r_0}(q_j)\subseteq \om_j$, $j=1,\cdots m$. Then let us define,
\begin{equation}\begin{aligned}\label{D_qD} D_{\sscp \om}(\mathbf{q}) &=\lim_{r\to0}
                 \Bigg[\sum_{j=1}^m h(q_j)e^{{G_{j,{\sscp \om}}}^*(q_j)} \Big(\int_{\om_j\setminus  B_{{{r_j}}}(q_j)}
                 e^{ \sum_{l=1}^m 8\pi  {G_{\sscp\om}}(x,q_l)-{G_{j,{\sscp \om}}}^*(q_j)+\log h(x)-\log h(q_j)}
                 \mathrm{d}x
-\frac{\pi}{r^2_j}\Big)\Bigg],\end{aligned}
\end{equation}{where $r_j=r\sqrt{8h(q_j)e^{{G_{j,{\sscp \om}}}^*(q_j)}} $ and $\om_1\equiv \om$ if $m=1$.} For $(x_1,\cdots, x_m)\in \om\times \cdots \om$, we also define,
\begin{align}\label{f_qD}
&f_{m, {\sscp \om}}(x_1,x_2,\cdots,x_m)=\sum_{j=1}^{m}\big[\log(h(x_j))+4\pi R_{\sscp \om}(x_j,x_j)\big]+4\pi\sum_{l\neq j}^{1,\cdots,m}G_{\sscp \om}(x_l,x_j).
\end{align}

Of course, even in this situation we first need to derive the following improvement of Theorem 6.2 in \cite{cl1}.

\begin{theorem}\label{thm_newD}{
Let $u_n$ be a sequence of solutions of  \eqref{mf} which blows up at the points $q_j\notin\{p_1,\cdots,p_N\}$, $j=1,\cdots,m$, $\delta>0$ be a fixed constant and }
$
{\lambda_{n,j}=\max_{B_\delta(q_j)}\left(u_n-\log\left( \;\int\limits_\om h e^{u_n}\right)\right)\
\textrm{for}\  j=1,\cdots, m.}
$ \\ {Then, for any $n$ large enough, the following estimate holds,}
 \begin{equation}
\label{a5D}\begin{aligned}
{\rho_n-8\pi m}=~&{\frac{2 \ell_{\sscp \om}(\mathbf{q})e^{-\lambda_{n,1}}}{m h^2(q_1)e^{G_{1,\sscp \om}^*(q_1)}}
 \Big(\lambda_{n,1}+ \log  \rho_n h^2(q_1)e^{G_{1,\sscp \om}^*(q_1)} \delta^2 -2\Big)}
+ \frac{8e^{-\lambda_{n,1}} }{h^2(q_1)e^{G_{1,\sscp \om}^*(q_1)}\pi m}\Big( D_{\sscp \om}(\mathbf{q}) +O(\delta^\sigma)\Big)
\\&+O(\lambda_{n,1}^2e^{-\frac{3}{2}\lambda_{n,1}})+O(e^{-(1+\frac{\sigma}{2})\lambda_{n,1}}),\ \ \textrm{where}\ \ \sigma>0 \ \ \textrm{is defined by}\ \  h_0\in C^{2,\sigma}(M).
\end{aligned}\end{equation}
\end{theorem}

Then we have,

\begin{theorem}\label{thm_concenD}Let $u_{n}^{(1)}$ and $u_{n}^{(2)}$ be  two sequence of solutions of \eqref{mf},
with $\rho^{(1)}_n=\rho_n=\rho^{(2)}_n$ and blowing up at the points $q_j\notin\{p_1,
\cdots,p_N\}$, $j=1,\cdots,m$, where $\mathbf{q}=(q_1,\cdots,q_m)$ is a critical point of $f_{m, {\sscp \om}}$ and
$\textrm{det}(D^2f_{m, {\sscp \om}}(\mathbf{q}))\neq 0$.
Assume that, either,
\begin{enumerate}
\item $\ell_{\sscp \om}(\mathbf{q})\neq0$, or,
\item  $\ell_{\sscp \om}(\mathbf{q})=0$ and $D_{\sscp \om}(\mathbf{q})\neq0$.
\end{enumerate}
Then there exists $n_0\ge1$ such that $u_{n}^{(1)}=u_{n}^{(2)}$ for all $n\ge n_0$.
\end{theorem}

\noindent {\it Proof of Theorems \ref{thm_newD} and \ref{thm_concenD}}. The proof of Theorems \ref{thm_newD} and \ref{thm_concenD} can be worked out by a step by step adaptation of the one of Theorems \ref{thm_new} and \ref{thm_concen} with minor changes. Actually the arguments are somehow easier in this case, since we don't need to pass to local isothermal coordinates around each blow up point. In particular it is readily seen that the subtle part of the estimates obtained in section \ref{sec_est} and \ref{sec_poho} relies on the local estimates for blow up solutions of \rife{mfd} listed in section \ref{sec_pre}. The corresponding estimates for the Dirichlet problem was already obtained in \cite{cl1} and have the same form just with minor changes, as for example concerning the fact that here we have $\varphi_j\equiv 0$ and $K\equiv 0$. Actually the estimates about the Dirichlet problem in \cite{cl1} are worked out with $\alpha_j=0$, $j=1,\cdots,N$ but since $\{q_1,\cdots,q_m\}\cap \{p_1,\cdots,p_N\}=\emptyset$, then it is straightforward to check that they still hold as they stand possibly with few changes about the regularity of solutions, see also Remark \ref{remreg}. We refer the reader to \cite{cl1} for more details concerning this point. Actually, by our regularity assumption about $\partial \om$, it can be shown by a moving plane argument (see \cite{CCL}) that solutions of \rife{mf} are uniformly bounded in a fixed neighborhood of $\partial \om$. Therefore, since also $\{q_1,\cdots,q_m\}$ are far from $\partial \om$ by assumption, then it is straightforward to check that all the additional terms coming from the boundary do not affect the estimates needed to conclude the proof. We skip the details to avoid repetitions. \hfill$\Box$

\section{A refined estimate of $\rho_n-8\pi m$ in case $\ell(\mathbf{q})=0$.}\label{apppendix}
\noindent {\em Proof of  Theorem \ref{thm_new}.} In this section we prove Theorem \ref{thm_new}, that is \eqref{a5}. In view of \eqref{int_tilu}, we  see that,
\begin{equation}
\label{ap1}\begin{aligned}
\rho_n&=\rho_n\int_M he^{\tilde{u}_n}\mathrm{d}\mu=\rho_n\Big(\int_{\cup_{j=1}^m{U}^{\sscp M}_{\delta }(q_j)}+
\int_{M\setminus \cup_{j=1}^m{U}^{\sscp M}_{\delta }(q_j)} \Big)he^{\tilde{u}_n}\mathrm{d}\mu
 =\sum_{j=1}^m\rho_{n,j}+\int_{M\setminus \cup_{j=1}^m{U}^{\sscp M}_{\delta }(q_j)} he^{\tilde{u}_n}\mathrm{d}\mu.\end{aligned}\end{equation}

\noindent \textit{Step 1.} In step 1 we provide and estimate about $\int_{M\setminus \cup_{j=1}^m{U}^{\sscp M}_{\delta }(q_j)} he^{\tilde{u}_n}\mathrm{d}\mu.$\\
In view of  \eqref{lambda_exp}, \eqref{wnx} and \eqref{est_wn}, we see that,
\begin{equation}
\label{st1}\begin{aligned}
 \rho_n\int_{M\setminus \cup_{j=1}^m {U}^{\sscp M}_{\delta }(q_j)} he^{\tilde{u}_n}\mathrm{d}\mu
 &=\rho_n\sum_{j=1}^m\int_{M_j\setminus {U}^{\sscp M}_{\delta }(q_j)} he^{w_n+\sum_{l=1}^m\rho_{n,l}G(x,x_{n,l})+\int_M\tilde{u}_n\mathrm{d}\mu}\mathrm{d}\mu
\\&=\rho_n\sum_{j=1}^m\int_{M_j\setminus U^{\sscp M}_{\delta }(q_j)} he^{\sum_{l=1}^m\rho_{n,l}G(x,x_{n,l})+\int_M\tilde{u}_n\mathrm{d}\mu}(1+o(e^{-\frac{\lambda_{n,1}}{2}}))
\mathrm{d}\mu
\\&=\rho_n\sum_{j=1}^m\int_{M_j\setminus U^{\sscp M}_{\delta }(q_j)} he^{\sum_{l=1}^m\rho_{n,l}G(x,x_{n,l})-\lambda_{n,j}-2\log(\frac{\rho_n h(x_{n,j})}{8})-G_j^*(x_{n,j})}(1+o(e^{-\frac{\lambda_{n,1}}{2}}))
\mathrm{d}\mu
.\end{aligned}\end{equation}
Therefore, by using \eqref{diffxandq}, \eqref{rhon8pi} and \eqref{st1}, we conclude that,
\begin{equation}
\label{st2}\begin{aligned}
 \rho_n\int_{M\setminus \cup_{j=1}^m{U}^{\sscp M}_{\delta }(q_j)} he^{\tilde{u}_n}\mathrm{d}\mu
 &=\sum_{j=1}^m\frac{64e^{-\lambda_{n,j}}}{h(q_j)\rho_n}\int_{M_j\setminus U^{\sscp M}_{\delta }(q_j)} e^{\sum_{l=1}^m8\pi G(x,q_l)-G_j^*(q_j)+\log h(x)-\log h(q_j)}(1+o(e^{-\frac{\lambda_{n,1}}{2}}))
\mathrm{d}\mu,
\end{aligned}\end{equation}
and then, in view of the definition of $\Phi_j(x,\mathbf{q})$ (see  \eqref{P_q}) and \eqref{relation_lambda}, we find that,
\begin{equation}
\label{st3}\begin{aligned}
 \rho_n\int_{M\setminus \cup_{j=1}^m {U}^{\sscp M}_{\delta }(q_j)} he^{\tilde{u}_n}\mathrm{d}\mu
 &=\sum_{j=1}^m\frac{64e^{-\lambda_{n,j}}}{h(q_j)\rho_n}\int_{M_j\setminus U^{\sscp M}_{\delta }(q_j)} e^{\Phi_j(x,\mathbf{q})}(1+o(e^{-\frac{\lambda_{n,1}}{2}}))
\mathrm{d}\mu
\\&=\sum_{j=1}^m\frac{64e^{-\lambda_{n,1}}h(q_j)e^{G_j^*(q_j)}}{h^2(q_1)e^{G_1^*(q_1)}\rho_n}(1+O(e^{-\frac{\lambda_{n,1}}{2}}))\int_{M_j\setminus U^{\sscp M}_{\delta }(q_j)} e^{\Phi_j(x,\mathbf{q})}(1+o(e^{-\frac{\lambda_{n,1}}{2}}))
\mathrm{d}\mu
\\&=\sum_{j=1}^m\frac{64e^{-\lambda_{n,1}}h(q_j)e^{G_j^*(q_j)}}{h^2(q_1)e^{G_1^*(q_1)}\rho_n}\int_{M_j\setminus U^{\sscp M}_{\delta }(q_j)} e^{\Phi_j(x,\mathbf{q})}
\mathrm{d}\mu+O(e^{-\frac{3\lambda_{n,1}}{2}}).
\end{aligned}\end{equation}

\noindent \textit{Step 2.} In this step we provide an estimate about $\rho_{n,j}$ (see \eqref{rhonj}). \\
First of all let us set,
\begin{equation}\label{hjs}
h_j(\un{x})=h(\un{x})e^{2\varphi_j(\un{x})}.
\end{equation}By using \eqref{eta} and setting $\tau_{n,j}=e^{\sscp \frac{\lambda_{n,j}}{2}}$ and
$$
I_{n,j}^*=\int_{B_{\delta }(\un{q}_j)}\frac{\rho_nh_j(\un{x}_{n,j})e^{\lambda_{n,j}}
(e^{G_j^*(\un{x})-G_j^*(\un{x}_{n,j})+\log h_j(\un{x})-\log h_j(\un{x}_{n,j})+\eta_{n,j}(\un{x})}-1)}
{(1+\frac{\rho_n h(\un{x}_{n,j})}{8}e^{\lambda_{n,j}}|\un{x}-\un{x}_{n,j,*}|^2)^2}\mathrm{d}\un{x},
$$
we find that,
\begin{equation}
\label{ap2}\begin{aligned}
\rho_{n,j}&=\int_{ {U}^{\sscp M}_{\delta }(q_j)} \rho_nhe^{\tilde{u}_n}\mathrm{d}\mu=
\int_{B_{\delta }(\un{q}_j)}\frac{\rho_nh_je^{\lambda_{n,j}+G_j^*(\un{x})-G_j^*(\un{x}_{n,j})+\eta_{n,j}(\un{x})}}
{(1+\frac{\rho_n h_j(\un{x}_{n,j})}{8}e^{\lambda_{n,j}}|\un{x}-\un{x}_{n,j,*}|^2)^2}\mathrm{d}\un{x}
\\&=
\int_{B_{\delta }(\un{q}_j)}\frac{\rho_nh_j(\un{x}_{n,j})e^{\lambda_{n,j}}}{(1+\frac{\rho_n h(\un{x}_{n,j})}
{8}e^{\lambda_{n,j}}|\un{x}-\un{q}_j+\un{q}_j-\un{x}_{n,j,*}|^2)^2}\mathrm{d}\un{x}+I_{n,j}^*
\\&=
\int_{B_{\sqrt{\frac{\rho_n h(\un{x}_{n,j})}{8}}\delta  \tau_{n,j}}(0)}\frac{8}{(1+|z+\sqrt{\frac{\rho_n h(\un{x}_{n,j})}{8}}
\tau_{n,j}(\un{q}_j-\un{x}_{n,j,*})|^2)^2}\mathrm{d}z
+I_{n,j}^*
\\&=8\pi
-\int_{\mathbb{R}^2\setminus B_{\sqrt{\frac{\rho_n h(\un{x}_{n,j})}{8}}\delta  \tau_{n,j}}(0)}\frac{8}{(1+|z+\sqrt{\frac{\rho_n h(\un{x}_{n,j})}{8}}
\tau_{n,j}(\un{q}_j-\un{x}_{n,j,*})|^2)^2}\mathrm{d}z
+I_{n,j}^*.
\end{aligned}\end{equation}
On the other side, by \eqref{diff_x_nj} and \eqref{diffxandq}, we see that,
\begin{equation}
\label{ap3}\begin{aligned}
&\int_{\mathbb{R}^2\setminus B_{\sqrt{\frac{\rho_n h(\un{x}_{n,j})}{8}}\delta  \tau_{n,j}}(0)}
\frac{8}{(1+|z+\sqrt{\frac{\rho_n h(\un{x}_{n,j})}{8}} \tau_{n,j}(\un{q}_j-\un{x}_{n,j,*})|^2)^2}\mathrm{d}z
\\&=\int_{\mathbb{R}^2\setminus B_{\sqrt{\frac{\rho_n h(\un{x}_{n,j})}{8}}
\delta  \tau_{n,j}}(0)}\frac{8}{(1+|z|^2)^2}\mathrm{d}z
+\int_{\mathbb{R}^2\setminus B_{\sqrt{\frac{\rho_n h(\un{x}_{n,j})}{8}}\delta  \tau_{n,j}}(0)}
\Bigg(\frac{8}{(1+|z+O(\lambda_{n,j}e^{-\frac{\lambda_{n,j}}{2}})|^2)^2}-\frac{8}{(1+|z|^2)^2}\Bigg)\mathrm{d}z
\\&=\int^{\infty}_{\sqrt{\frac{\rho_n h(\un{x}_{n,j})}{8}}\delta  \tau_{n,j}}\frac{16\pi r}{(1+r^2)^2}\mathrm{d}r
+\int_{\mathbb{R}^2\setminus B_{\sqrt{\frac{\rho_n h(\un{x}_{n,j})}{8}}\delta  \tau_{n,j}}(0)}
\frac{O(\lambda_{n,j}e^{-\frac{\lambda_{n,j}}{2}})}{(1+|z|)^5}\mathrm{d}z
 =\frac{64\pi e^{-\lambda_{n,j}}}{\rho_n h(\un{x}_{n,j})\delta ^2 }+O(\lambda_{n,j}e^{-2\lambda_{n,j}}),
\end{aligned}\end{equation}
where we used the identity,  $\int_R^\infty \frac{16\pi r}{(1+r^2)^2}\mathrm{d}r=\frac{8\pi}{R^2+1}=\frac{8\pi}{R^2}+O(R^{-4})$ for $R\gg1$. \\
Therefore \eqref{ap3} and \eqref{relation_lambda}, \eqref{diffxandq}, show that,
\begin{equation}
\label{0ap3}\begin{aligned}
&\int_{\mathbb{R}^2\setminus B_{\sqrt{\frac{\rho_n h(\un{x}_{n,j})}{8}}\delta  e^{\frac{\lambda_{n,j}}{2}}}(0)}\frac{8}{(1+|z+\sqrt{\frac{\rho_n h(\un{x}_{n,j})}{8}} e^{\frac{\lambda_{n,j}}{2}}(\un{q}_j-\un{x}_{n,j,*})|^2)^2}\mathrm{d}z
 =\frac{64\pi h(q_j)e^{G_j^*(q_j)} e^{-\lambda_{n,1}}}{\rho_n h^2(q_1)e^{G_1^*(q_1)}\delta ^2 }+O(\lambda_{n,j}e^{-2\lambda_{n,j}}).
\end{aligned}\end{equation}

The estimate of the term $I_{n,j}^*$ is more delicate. Toward this goal we have to work out a refined version of an argument first introduced in \cite{cl1}.\\
First of all, let us recall that (see \eqref{eta}),
\[\eta_{n,j}(\un{x})=\tilde{u}_n(\un{x})-U_{n,j}(\un{x})-(G_{j}^{*}(\un{x})-G_{j}^{*}(\un{x}_{n,j})),\; x\in B_{\delta }(\un{q}_j).\]
Then, {in view of \eqref{delr}}, for $x\in B_{\delta }(\un{q}_j)$ we have,
\begin{equation}\label{ap5}\begin{aligned}
\Delta\eta_{n,j}
&=-\rho_ne^{2\varphi_j}(he^{\tilde{u}_n}-1)+\frac{\rho_nh(\un{x}_{n,j})e^{\lambda_{n,j}}}
{(1+\frac{\rho_nh(\un{x}_{n,j})}{8}e^{\lambda_{n,j}}|\un{x}-\un{x}_{n,j,*}|^2)^2}
-8\pi m e^{2\varphi_j}
\\&=-\frac{\rho_nh_j(\un{x}_{n,j})e^{\lambda_{n,j}}(e^{G_j^*(\un{x})-G_j^*(\un{x}_{n,j})+\log h_j(\un{x})-\log h_j(\un{x}_{n,j})+\eta_{n,j}(\un{x})}-1)}
{(1+\frac{\rho_nh(\un{x}_{n,j})}{8}e^{\lambda_{n,j}}|\un{x}-\un{x}_{n,j,*}|^2)^2} -(8\pi m-\rho_n) e^{2\varphi_j},
\end{aligned}\end{equation}
which immediately implies that,
\begin{equation}\label{ap6}\begin{aligned}
 I^*_{n,j}=-\int_{\partial B_{\delta }(\un{q}_j)}\frac{\partial \eta_{n,j}}{\partial \nu}\mathrm{d}\sigma-\int_{B_{\delta }(\un{q}_j)}(8\pi m-\rho_n)
 e^{2\varphi_j}\mathrm{d}\un{x}.
\end{aligned}\end{equation}

Next we obtain an estimate about $\int_{\partial B_{\delta }(\un{q}_j)}\frac{\partial \eta_{n,j}}{\partial \nu}\mathrm{d}\sigma$.
Let us define,
$$
\mathcal{A}_{n,j}(\un{x})=\frac{\rho_nh_j(\un{x}_{n,j})e^{\lambda_{n,j}}(e^{G_j^*(\un{x})-G_j^*(\un{x}_{n,j})+\log h_j(\un{x})-\log h_j(\un{x}_{n,j})+\eta_{n,j}(\un{x})}-1)}
{(1+\frac{\rho_nh(\un{x}_{n,j})}{8}e^{\lambda_{n,j}}|\un{x}-\un{x}_{n,j,*}|^2)^2},
$$
$$
\mathcal{B}_{n,j}(\un{x})=\frac{\rho_nh_j(\un{x}_{n,j})e^{\lambda_{n,j}}(e^{G_j^*(\un{x})-G_j^*(\un{x}_{n,j})+\log h_j(\un{x})-\log h_j(\un{x}_{n,j})
+\eta_{n,j}(\un{x})}-1-\eta_{n,j}(\un{x}))}
{(1+\frac{\rho_nh(\un{x}_{n,j})}{8}e^{\lambda_{n,j}}|\un{x}-\un{x}_{n,j,*}|^2)^2},\ \ \textrm{and}
$$
 \begin{align}\label{ap7}
\psi_{n,j}(\un{x})=\frac{1-\frac{\rho_nh(\un{x}_{n,j})}{8}e^{\lambda_{n,j}}|\un{x}-\un{x}_{n,j,*}|^2}
{1+\frac{\rho_nh(\un{x}_{n,j})}{8}e^{\lambda_{n,j}}|\un{x}-\un{x}_{n,j,*}|^2},
\end{align}
which satisfies $
\Delta \psi_{n,j}+\frac{\rho_nh(\un{x}_{n,j})e^{\lambda_{n,j}}}{(1+\frac{\rho_nh(\un{x}_{n,j})}{8}e^{\lambda_{n,j}}|\un{x}-\un{x}_{n,j,*}|^2)^2}\psi_{n,j}=0.
$\\
In view of \eqref{ap5}   and integrating by parts, we find that,
\begin{equation}\label{ap9}\begin{aligned}
  &\int_{\partial B_{\delta }(\un{q}_j)}\Bigg(\psi_{n,j}\frac{\partial \eta_{n,j}}{\partial \nu}-\eta_{n,j}\frac{\partial \psi_{n,j}}{\partial \nu}\Bigg)
 \mathrm{d}\sigma
 =\int_{B_{\delta }(\un{q}_j)}\Bigg(\psi_{n,j}\Delta \eta_{n,j}-\eta_{n,j}\Delta \psi_{n,j}\Bigg)\mathrm{d}\un{x}\\
 &=\int_{B_{\delta }(\un{q}_j)}\Big(-\psi_{n,j}\mathcal{A}_{n,j}-\psi_{n,j}(8\pi m-\rho_n) e^{2\varphi_j} +
\frac{\eta_{n,j}\psi_{n,j}\rho_nh(\un{x}_{n,j})e^{\lambda_{n,j}}}{(1+\frac{\rho_nh(\un{x}_{n,j})}{8}e^{\lambda_{n,j}}|\un{x}-\un{x}_{n,j,*}|^2)^2}\Big)
 \mathrm{d}\un{x}\end{aligned}
\end{equation}
In the same time, for  $x\in\partial B_{\delta }(\un{q}_j)$, we have,
\begin{equation}\label{ap10}
  \begin{aligned}
  \psi_{n,j}(\un{x})=-1+\frac{2}{1+\frac{\rho_nh(\un{x}_{n,j})}{8}e^{\lambda_{n,j}}|\un{x}-\un{x}_{n,j,*}|^2}=-1+O(e^{-\lambda_{n,j}}), \  \nabla\psi_{n,j}=O(e^{-\lambda_{n,j}}).
  \end{aligned}
\end{equation}
In view of \eqref{eta_ets} and \cite[Lemma 4.1]{cl1}, we also have for $x\in\partial B_{\delta }(\un{q}_j)$,
\begin{equation}\label{ap11}
  \begin{aligned}
  |\eta_{n,j}|+|\nabla\eta_{n,j}|=O(\lambda_{n,j}^2e^{-\lambda_{n,j}}).  \end{aligned}
\end{equation}
At this point, by using \eqref{ap9}-\eqref{ap11} and \eqref{rhon}, we see that,
\begin{equation}\label{ap12}\begin{aligned}
&-\int_{\partial B_{\delta }(\un{q}_j)}\frac{\partial \eta_{n,j}}{\partial \nu}\mathrm{d}\sigma
 =  \int_{B_{\delta }(\un{q}_j)} (1-\frac{2}{1+\frac{\rho_nh(\un{x}_{n,j})}{8}e^{\lambda_{n,j}}
|\un{x}-\un{x}_{n,j,*}|^2})(8\pi m-\rho_n) e^{2\varphi_j}  - \psi_{n,j}\mathcal{B}_{n,j}(\un{x})\mathrm{d}\un{x}+O(\lambda_{n,j}^2e^{-2\lambda_{n,j}})\\
=&\int_{B_{\delta }(\un{q}_j)}(8\pi m-\rho_n) e^{2\varphi_j}\mathrm{d}\un{x}-\int_{B_{\delta }(\un{q}_j)}\psi_{n,j}\mathcal{B}_{n,j}(\un{x})\mathrm{d}\un{x} +O(|8\pi m-\rho_n|\lambda_{n,j}e^{-\lambda_{n,j}})+O(\lambda_{n,j}^2e^{-2\lambda_{n,j}})\\
=&\int_{B_{\delta }(\un{q}_j)}(8\pi m-\rho_n) e^{2\varphi_j}\mathrm{d}\un{x}-\int_{B_{\delta }(\un{q}_j)}\psi_{n,j}\mathcal{B}_{n,j}(\un{x})\mathrm{d}\un{x}
 +O(\lambda_{n,j}^2e^{-2\lambda_{n,j}}).
 \end{aligned}
\end{equation}
Next, let us observe that, since $h\in C^{2,\sigma}(M)$ and in view of \eqref{diff_x_nj}, \eqref{eta_ets} and \eqref{gradatq}, for $x\in B_{\delta }(\un{q}_j)$ we have,
\begin{equation}\label{ap14}
\begin{aligned}
&e^{G_j^*(\un{x})-G_j^*(\un{x}_{n,j})+\log h_j(\un{x})-\log h_j(\un{x}_{n,j})+\eta_{n,j}(\un{x})}-1-\eta_{n,j}(\un{x})
 \\&=\nabla_x (G_j^*(\un{x})+\log h_j(\un{x}))\Big|_{x=\un{x}_{n,j}}\cdot(\un{x}-\un{x}_{n,j,*})
  +\frac{1}{2}\sum_{1\le k,l\le 2}\nabla^2_{x_kx_l} (G_j^*(\un{x})+\log h_j(\un{x}))\Big|_{x=\un{x}_{n,j}}(\un{x}-\un{x}_{n,j,*})_k(\un{x}-\un{x}_{n,j,*})_l\\& +O(|\un{x}-\un{x}_{n,j,*}|^{2+\sigma})
+O(\lambda_{n,j}^4e^{-2\lambda_{n,j}})+O(\lambda_{n,j}^2e^{-\lambda_{n,j}}|\un{x}-\un{x}_{n,j,*}|).
  \end{aligned}
\end{equation}
Clearly \eqref{diff_x_nj} and \eqref{diffxandq} imply that,
\begin{equation}\label{ap15}
  \begin{aligned}
  &|B_{\delta }(\un{q}_j)\setminus B_{\delta }(\un{x}_{n,j,*})|+|B_{\delta }(\un{x}_{n,j,*})\setminus B_{\delta }(\un{q}_j)|=O(|\un{x}_{n,j,*}-\un{q}_j|)=O(\lambda_{n,j}e^{-\lambda_{n,j}}).\end{aligned}
\end{equation}
At this point we use \eqref{ap12}, \eqref{ap14}, and \eqref{ap15}, to conclude that,
{\allowdisplaybreaks
\begin{align}\label{ap16}
 &\int_{B_{\delta }(\un{q}_j)}\psi_{n,j}\mathcal{B}_{n,j}(\un{x})\mathrm{d}\un{x}\nonumber\nonumber\\
 &=\int_{B_{\delta }(\un{x}_{n,j,*})}\psi_{n,j}\frac{\rho_nh_j(\un{x}_{n,j})e^{\lambda_{n,j}}}
{(1+\frac{\rho_nh(\un{x}_{n,j})}{8}e^{\lambda_{n,j}}|\un{x}-\un{x}_{n,j,*}|^2)^2}
 \Big[(\nabla_x (G_j^*(\un{x})+\log h_j(\un{x}))\Big|_{x=\un{x}_{n,j}}\cdot(\un{x}-\un{x}_{n,j,*})\nonumber\\
&+\frac{1}{2}\sum_{1\le k,l\le 2}\nabla^2_{x_kx_l} (G_j^*(\un{x})+\log h_j(\un{x}))
\Big|_{x=\un{x}_{n,j}}(\un{x}-\un{x}_{n,j,*})_k(\un{x}-\un{x}_{n,j,*})_l\nonumber\\
&+O(|\un{x}-\un{x}_{n,j,*}|^{2+\sigma})+O(\lambda_{n,j}^4e^{-2\lambda_{n,j}})+O(\lambda_{n,j}^2e^{-\lambda_{n,j}}|\un{x}-\un{x}_{n,j,*}|))\Big]\mathrm{d}\un{x}
 +O(\lambda_{n,j}e^{-2\lambda_{n,j}})\nonumber\\
&=\int_{B_{\sqrt{\frac{\rho_nh(\un{x}_{n,j})}{8}}\delta  e^{\frac{\lambda_{n,j}}{2}}}(0)}\frac{8(1-|z|^2)}{(1+|z|^2)^3} \Big[(\nabla_x (G_j^*(\un{x})+\log h_j(\un{x}))\Big|_{x=\un{x}_{n,j}}\cdot z (\sqrt{\frac{8}{\rho_nh(\un{x}_{n,j})}} e^{-\frac{\lambda_{n,j}}{2}})\nonumber\\
&+\frac{1}{2}\sum_{1\le k,l\le 2}\nabla^2_{x_kx_l} (G_j^*(\un{x})+\log h_j(\un{x}))\Big|_{x=\un{x}_{n,j}}z_kz_l(\frac{8}{\rho_nh(\un{x}_{n,j})}e^{-\lambda_{n,j}})+O(e^{-(1+\frac{\sigma}{2})\lambda_{n,j}}|z|^{2+\sigma})+O(\lambda_{n,j}^2e^{-\frac{3}{2}\lambda_{n,j}}|z|))\Big]\mathrm{d}z
\nonumber\\&+O(\lambda_{n,j}^4e^{-2\lambda_{n,j}})\nonumber\\
&=\frac{16(\Delta \log h(\un{x}_{n,j})-2K(\un{x}_{n,j})+8\pi m)e^{-\lambda_{n,j}}}{\rho_nh(\un{x}_{n,j})}\int_{B_{\sqrt{\frac{\rho_nh(\un{x}_{n,j})}{8}}\delta  e^{\frac{\lambda_{n,j}}{2}}}(0)}\frac{|z|^2(1-|z|^2)}{(1+|z|^2)^3}\mathrm{d}z\\&+O(e^{-\lambda_{n,j}}\delta ^\sigma)+O(\lambda_{n,j}^2e^{-\frac{3}{2}\lambda_{n,j}})+O(e^{-(1+\frac{\sigma}{2})\lambda_{n,1}}),
 \end{align}}
{where in the last equality we used \eqref{g*}, \eqref{delr1}, \eqref{delr} and $\varphi_j(\un{x}_{n,j})=0$}. Therefore, by using \eqref{ap6}, \eqref{ap12}  and \eqref{ap16},  we conclude that,
\begin{equation}\label{ap17}\begin{aligned}
I_{n,j}^*= &-\frac{32\pi(\Delta \log h(\un{x}_{n,j})-2K(\un{x}_{n,j})+8\pi m)e^{-\lambda_{n,j}}}{\rho_nh(\un{x}_{n,j})}
\int_0^{\sqrt{\frac{\rho_nh(\un{x}_{n,j})}{8}}\delta  e^{\frac{\lambda_{n,j}}{2}}}\frac{r^3(1-r^2)}{(1+r^2)^3}\mathrm{d}r\\
&+O(e^{-\lambda_{n,j}}\delta ^\sigma)+O(\lambda_{n,j}^2e^{-\frac{3}{2}\lambda_{n,j}})+O(e^{-(1+\frac{\sigma}{2})\lambda_{n,1}})\\
= &\frac{16\pi(\Delta \log h(\un{x}_{n,j})-2K(\un{x}_{n,j})+8\pi m)e^{-\lambda_{n,j}}}{\rho_nh(\un{x}_{n,j})}\Big(\lambda_{n,j}+ \log {\frac{\rho_nh(\un{x}_{n,j})}{8}}\delta ^2-2\Big)\\& +O(e^{-\lambda_{n,j}}\delta ^\sigma)+O(\lambda_{n,j}^2e^{-\frac{3}{2}\lambda_{n,j}})+O(e^{-(1+\frac{\sigma}{2})\lambda_{n,1}}),
\end{aligned}\end{equation}
where we used the identity for large $R\gg1$,
\begin{align*}
\int_0^R\frac{r^3(1-r^2)}{(1+r^2)^3}\mathrm{d}r=~&\frac{2R^4+R^2}{2(R^2+1)^2}-\frac{1}{2}\log (R^2+1)
 =-\log R+1+O(R^{-2}).
\end{align*}
By using \eqref{relation_lambda} and \eqref{diffxandq} we can write this estimate in the following form,
\begin{equation}\label{ap18}\begin{aligned}
I^*_{n,j}=~&\Bigg\{ \frac{16\pi h(q_j)e^{G_j^*(q_j)}(\Delta \log h(q_j)-2K(q_j)+8\pi m)e^{-\lambda_{n,1}}}{\rho_nh^2(q_1)e^{G_1^*(q_1)}} \Big(\lambda_{n,1}+ \log \frac{\rho_n h^2(q_1)e^{G_1^*(q_1)}\delta ^2}{8h(q_j)e^{G_j^*(q_j)}}-2\Big)\Bigg\}\\&+O(e^{-\lambda_{n,j}}\delta ^\sigma)
+O(\lambda_{n,j}^2e^{-\frac{3}{2}\lambda_{n,j}})+O(e^{-(1+\frac{\sigma}{2})\lambda_{n,1}}),
\end{aligned}\end{equation}
and eventually use it with \eqref{ap2} and \eqref{0ap3} to obtain that,
\begin{equation}
\label{ap19}\begin{aligned}
\rho_{n,j}=~&8\pi-\frac{64\pi h(q_j)e^{G_j^*(q_j)} e^{-\lambda_{n,1}}}{\rho_n h^2(q_1)e^{G_1^*(q_1)}\delta ^2 }+O(e^{-\lambda_{n,j}}\delta ^\sigma)
+O(\lambda_{n,j}^2e^{-\frac{3}{2}\lambda_{n,j}})+O(e^{-(1+\frac{\sigma}{2})\lambda_{n,1}})\\
&+\Bigg\{ \frac{16\pi h(q_j)e^{G_j^*(q_j)}(\Delta \log h(q_j)-2K(q_j)+8\pi m)e^{-\lambda_{n,1}}}{\rho_nh^2(q_1)e^{G_1^*(q_1)}}
 \Big(\lambda_{n,1}+ \log \frac{\rho_n h^2(q_1)e^{G_1^*(q_1)}\delta ^2}{8h(q_j)e^{G_j^*(q_j)}}-2\Big)\Bigg\}.
\end{aligned}\end{equation}
\textit{Step 3.}  In view of \eqref{ap1}, \eqref{st3} and \eqref{ap19}, we find that,
\begin{equation}
\label{fap1}\begin{aligned}
\rho_n&=\sum_{j=1}^m\rho_{n,j}+\rho_n\int_{M\setminus \cup_{j=1}^m {U}^{\sscp M}_{\delta }(q_j)} he^{\tilde{u}_n}\mathrm{d}\mu
\\&=8\pi m+\sum_{j=1}^m
\Bigg\{ \frac{16\pi h(q_j)e^{G_j^*(q_j)}(\Delta \log h(q_j)-2K(q_j)+8\pi m)e^{-\lambda_{n,1}}}{\rho_nh^2(q_1)e^{G_1^*(q_1)}}
 \Big(\lambda_{n,1}+ \log \frac{\rho_n h^2(q_1)e^{G_1^*(q_1)}\delta ^2}{8h(q_j)e^{G_j^*(q_j)}}-2\Big)\Bigg\}
\\&+\sum_{j=1}^m\frac{64e^{-\lambda_{n,1}}h(q_j)e^{G_j^*(q_j)}}{h^2(q_1)e^{G_1^*(q_1)}\rho_n}\int_{M_j\setminus U^{\sscp M}_{\delta }(q_j)} e^{\Phi_j(x,\mathbf{q})}
\mathrm{d}\mu-\sum_{j=1}^m\frac{64\pi h(q_j)e^{G_j^*(q_j)} e^{-\lambda_{n,1}}}{\rho_n h^2(q_1)e^{G_1^*(q_1)}\delta ^2 }
\\&+O(e^{-\lambda_{n,1}}\delta ^\sigma)+O(\lambda_{n,1}^2e^{-\frac{3}{2}\lambda_{n,1}})+O(e^{-(1+\frac{\sigma}{2})\lambda_{n,1}}),
\end{aligned}\end{equation}
where we used \eqref{info_lambda}.

By using \eqref{fap1}, \eqref{rhon} and the definition of $\ell(\mathbf{q})$, we see that,
\begin{equation}
\label{apresult}\begin{aligned}
 \rho_n-8\pi m &=\frac{2 \ell(\mathbf{q})e^{-\lambda_{n,1}}}{m h^2(q_1)e^{G_1^*(q_1)}}
\Big(\lambda_{n,1}+ \log \rho_n h^2(q_1)e^{G_1^*(q_1)}\delta ^2-2\Big)
\\&-\sum_{j=1}^m
 \frac{2 h(q_j)e^{G_j^*(q_j)}(\Delta \log h(q_j)-2K(q_j)+8\pi m)e^{-\lambda_{n,1}}}{m h^2(q_1)e^{G_1^*(q_1)}}
 \Big( \log {8h(q_j)e^{G_j^*(q_j)}}\Big)
\\&+\sum_{j=1}^m\frac{8e^{-\lambda_{n,1}}h(q_j)e^{G_j^*(q_j)}}{h^2(q_1)e^{G_1^*(q_1)}\pi m}\Bigg(\int_{M_j\setminus U^{\sscp M}_{\delta }(q_j)} e^{\Phi_j(x,\mathbf{q})}
\mathrm{d}\mu-\frac{\pi }{\delta ^2 }+O(\delta^\sigma)\Bigg)
 +O(\lambda_{n,1}^2e^{-\frac{3}{2}\lambda_{n,1}})+O(e^{-(1+\frac{\sigma}{2})\lambda_{n,1}}).
\end{aligned}\end{equation}
For small $r>0$, let $r_j=r\sqrt{8h(q_j)e^{G_j^*(q_j)}} $ and observe that,
\begin{equation}
\label{apresult1}\begin{aligned}
 &\sum_{j=1}^mh(q_j)e^{G_j^*(q_j)}\Big(\int_{M_j\setminus U^{\sscp M}_{\delta }(q_j)} e^{\Phi_j(x,\mathbf{q})}
\mathrm{d}\mu\Big)
\\&=\sum_{j=1}^mh(q_j)e^{G_j^*(q_j)}\Big(\int_{M_j\setminus U^{\sscp M}_{r_j}(q_j)}e^{\Phi_j(x,\mathbf{q})}
\mathrm{d}\mu-\int_{B_{\delta }(\un{q}_j)\setminus B_{r_j}(\un{q}_j)}
\frac{e^{G_j^*(\un{x})-G_j^*(\un{q}_j)+\log h(\un{x})-\log h(\un{q}_j)+2\varphi_j(\un{x})}}{|\un{x}-\un{q}_j|^4}
\mathrm{d}\un{x}\Big).\end{aligned}\end{equation}
Since $\nabla f_m(\mathbf{q})=0$ and $\varphi_j(\un{x}_{n,j})=\nabla\varphi_j(\un{x}_{n,j})=0$,
 for $x\in B_{\delta }(\un{q}_j)\setminus B_{r_j}(\un{q}_j)$, we see from \eqref{diffxandq} that,
\begin{equation}\begin{aligned}
&G_j^*(\un{x})-G_j^*(\un{q}_j)+\log h(\un{x})-\log h(\un{q}_j)+2\varphi_j(\un{x})
\\&=2\varphi_j(\un{q}_j)+(\nabla_{x_j} f_m(\mathbf{q})+\nabla 2\varphi_j(\un{q}_j))
\cdot(\un{x}-\un{q}_j)+\frac{1}{2}\sum_{1\le j,k\le 2}(\nabla^2_{x_{j_k} x_{j_l}}f_m(\mathbf{q})+\nabla^2_{j k}
2\varphi_j(\un{q}_j))(\un{x}_{j_k}-\un{q}_{j_k})(\un{x}_{j_l}-\un{q}_{j_l})\\&+O(|\un{x}-\un{q}_j|^3)
=\frac{1}{2}\sum_{1\le j,k\le 2}(\nabla^2_{x_{j_k} x_{j_l}}f_m(\mathbf{q})+\nabla^2_{j k}2\varphi_j(\un{q}_{j}))(\un{x}_{j_k}-\un{q}_{j_k})(\un{x}_{j_l}-\un{q}_{j_l}) +O(|\un{x}-\un{q}_j|^3)
+O(\lambda_{n,1}e^{-\lambda_{n,1}}),\label{apresul2}
\end{aligned}\end{equation}
which implies that,
\begin{equation}\begin{aligned}\label{apresult3}
&\int_{B_{\delta }(\un{q}_j)\setminus B_{r_j}(\un{q}_j)}
\frac{e^{G_j^*(\un{x})-G_j^*(\un{q}_j)+\log h(\un{x})-\log h(\un{q}_j)+2\varphi_j(\un{x})}}{|\un{x}-\un{q}_j|^4}
\mathrm{d}\un{x}
\\&=\int_{B_{\delta }(\un{q}_j)\setminus B_{r_j}(\un{q}_j)} \frac{1+\frac{(\Delta_{x_j} f_m(\mathbf{q})+\Delta 2\varphi_j
(\un{q}_{j}))}{4}|\un{x}-\un{q}_{j}|^2+O(|\un{x}-\un{q}_j|^3)+O(\lambda_{n,1}e^{-\lambda_{n,1}})}{|\un{x}-\un{q}_j|^4}
\mathrm{d}\un{x}
\\&=\int_{B_{\delta }(\un{q}_j)\setminus B_{r_j}(\un{q}_j)} \frac{1+\frac{(\Delta \log(q_j)+8\pi m-2K(\un{q}_{j}))}{4}|\un{x}-\un{q}_j|^2}{|\un{x}-\un{q}_j|^4}
\mathrm{d}\un{x}
+O(\delta)+O(\lambda_{n,1}e^{-\lambda_{n,1}})
\\&=\frac{\pi}{r_j^2}-\frac{\pi}{\delta^2}+ \frac{\pi(\Delta \log(q_j)+8\pi m-2K(\un{q}_{j}))}{2}
(\log \delta-\log r_j)
 +O(\delta)+O(\lambda_{n,1}e^{-\lambda_{n,1}}),
\end{aligned}\end{equation}
where we used \eqref{delr1}, \eqref{delr}.

In view of \eqref{apresult}, \eqref{apresult1}, \eqref{apresult3}, we obtain
\begin{equation}
\label{apresult4}\begin{aligned}
&\rho_n-8\pi m
\\&=\frac{2 \ell(\mathbf{q})e^{-\lambda_{n,1}}}{m h^2(q_1)e^{G_1^*(q_1)}}
 \Big(\lambda_{n,1}+ \log  \rho_n h^2(q_1)e^{G_1^*(q_1)} r^2 -2\Big)
\\&+\sum_{j=1}^m\frac{8e^{-\lambda_{n,1}}h(q_j)e^{G_j^*(q_j)}}{h^2(q_1)e^{G_1^*(q_1)}\pi m}\Bigg(\int_{M_j\setminus U^{\sscp M}_{r_j}(q_j)} e^{\Phi_j(x,\mathbf{q})}
\mathrm{d}\mu-\frac{\pi }{r_j ^2 }+O(\delta^\sigma)\Bigg)
+O(\lambda_{n,1}^2e^{-\frac{3}{2}\lambda_{n,1}})+O(e^{-(1+\frac{\sigma}{2})\lambda_{n,1}})\  \textrm{for}\ \forall r>0,
\end{aligned}\end{equation}
where we used the explicit form of $r_j$ to cancel out the second line of \eqref{apresult4}.
Therefore, we conclude that
\begin{equation}
\label{apresult5}\begin{aligned}
\rho_n-8\pi m=~&\frac{2 \ell(\mathbf{q})e^{-\lambda_{n,1}}}{m h^2(q_1)e^{G_1^*(q_1)}}
 \Big(\lambda_{n,1}+ \log  \rho_n h^2(q_1)e^{G_1^*(q_1)} r^2 -2\Big)
 + \frac{8e^{-\lambda_{n,1}} }{h^2(q_1)e^{G_1^*(q_1)}\pi m}\Big( D(\mathbf{q})+O(\delta^\sigma)\Big)
\\&+O(\lambda_{n,1}^2e^{-\frac{3}{2}\lambda_{n,1}})+O(e^{-(1+\frac{\sigma}{2})\lambda_{n,1}}),
\end{aligned}\end{equation}
which is just the estimate \eqref{a5}.\hfill$\Box$

\section*{Acknowledgements}
The first author is partially supported by FIRB project "{\em Analysis and Beyond}",  by PRIN project 2015, ”Variational methods, with applications to problems in mathematical physics and geometry” and by the Consolidate the Foundations project 2015 (sponsored by Univ. of Rome "Tor Vergata"),
"{\em Nonlinear Differential Problems and their Applications}".


\begin{thebibliography}{99}

\bibitem{Aub} T. Aubin, "Nonlinear analysis on Manifolds Monge-Amp\'ere equations." Grundlehren der Mathematischen Wissenschaften {\bf 252},
Springer-Verlag, New York, 1982.

\bibitem{bp} S.  Baraket, F.  Pacard, {\em Construction of singular limits for a semilinear elliptic equation in dimension 2},
Calc.  Var.  Partial Differential Equations  6  (1998),  no.  1, 1-38.

\bibitem{barjga} D. Bartolucci, {\em On the best pinching constant of conformal metrics on $\mathbb{S}^2$ with
one and two conical singularities}, Jour. Geom. Analysis {\bf 23} (2013) 855-877.

\bibitem{B5} D. Bartolucci, "{\em Global bifurcation analysis of mean field equations and
the Onsager microcanonical description of two-dimensional turbulence}", Preprint 2016.

\bibitem{bcct} D. Bartolucci, C.C. Chen, C.S. Lin \& G. Tarantello,
{\em Profile of Blow Up Solutions To Mean Field Equations with Singular Data},
Comm. in P. D. E.  {\bf 29}(7-8) (2004), 1241-1265.



\bibitem{BJLY} D. Bartolucci, A. Jevnikar, Y. Lee, W. Yang,
{\em Non degeneracy, Mean Field Equations and the Onsager theory of 2D turbulence}, arXiv:1711.09970.


\bibitem{bl} D. Bartolucci, C.S. Lin, {\em Uniqueness Results for Mean Field Equations with Singular Data},
Comm. in P. D. E. {\bf 34}(7) (2009), 676-702.


\bibitem{BLin3} D. Bartolucci, C.S. Lin, {\em Existence and uniqueness for
Mean Field Equations on multiply connected domains at the critical parameter},
{Math. Ann.}, {\bf 359} (2014), 1-44; DOI 10.1007/s00208-013-0990-6.

\bibitem{BLT} D. Bartolucci, C.S. Lin, G. Tarantello, {\em Uniqueness and symmetry results for
solutions of a mean field equation on ${\mathbb{S}}^{2}$ via a new bubbling phenomenon},
{Comm. Pure Appl. Math.} {\bf 64}(12) (2011), 1677-1730.


\bibitem{BMal} D. Bartolucci, A. Malchiodi, {\em An improved geometric
inequality via vanishing moments, with applications to singular
Liouville equations}, {Comm. Math. Phys.} {\bf 322} (2013), 415-452.

\bibitem{BDeM}
D. Bartolucci, F. De Marchis, {\em On the Ambjorn-Olesen electroweak condensates},
{Jour. Math. Phys.} {\bf 53} 073704 (2012); doi: 10.1063/1.4731239.

\bibitem{BdM2} D. Bartolucci, F. De Marchis,
{\em Supercritical Mean Field Equations on convex domains and the Onsager's
statistical description of two-dimensional turbulence}, Archive for Rational Mechanics and Analysis, {\bf 217}/2 (2015), 525-570;
DOI: 10.1007/s00205-014-0836-8.

\bibitem{BdMM} D. Bartolucci, F. De Marchis, A. Malchiodi, {\em Supercritical conformal metrics on
surfaces with conical singularities}, Int. Math. Res. Not. 2011, (2011){\bf (24)}, 5625-5643; DOI: 10.1093/imrn/rnq285.

\bibitem{bt} D. Bartolucci, G. Tarantello, {\em Liouville type equations with
singular data and their applications to periodic multivortices for the
electroweak theory}, Comm. Math. Phys. {\bf 229} (2002), 3-47.

\bibitem{bt2} D. Bartolucci, G. Tarantello, {\em Asymptotic blow-up analysis for singular Liouville type equations with
applications}, J. Differential Equations. {\bf 262}/7 (2017), 3887-3931;

\bibitem{bm}
H. Brezis, F. Merle,
{\em Uniform estimates and blow-up behaviour for
solutions of $-\Delta u = V(x)e^{u}$ in two dimensions},
{Comm. in P.D.E.,}  {\bf 16}(8,9) (1991), 1223--1253.


\bibitem{clmp2} E. Caglioti, P.L. Lions, C. Marchioro \& M. Pulvirenti,
{\em A special class of stationary flows for two dimensional Euler equations: a
statistical mechanics description. II}, Comm. Math. Phys. {\bf 174} (1995),
229--260.

\bibitem{cama} A. Carlotto, A. Malchiodi, {\em
Weighted barycentric sets and singular Liouville equations on compact surfaces},
J. Funct. Anal. 262(2) (2012), 409-450


\bibitem{cLin14} C.C. Chai, C.S. Lin, C.L. Wang, {\em Mean field equations, hyperelliptic curves, and
modular forms: I},  Camb. J. Math. {\bf 3}(1-2) (2015),  127-274.

\bibitem{cfl} H.  Chan, C. C.  Fu,  C. S.   {\em Lin, Non-topological multi-vortex solutions to the self-dual Chern-Simons-Higgs equation},
Comm.  Math.  Phys.   231  (2002),  no.  2, 189-221.

\bibitem{CCL} S.Y.A. Chang, C.C. Chen, C.S. Lin, {\em Extremal functions for a mean field equation in two dimension},
Lecture on Partial Differential Equations, New Stud. Adv. Math. {\bf 2} Int. Press, Somerville, MA, 2003, 61-93.

\bibitem{CKLin} {Z. J. Chen,} T.J. Kuo and C.S. Lin, {\em Hamiltonian system for the elliptic form of Painlev\'{e} VI equation},
J. Math. Pure App., {\bf 106}(3) (2016),  546-581.


\bibitem{cli1} W. X. Chen \& C. Li, {\sl Classification of solutions of some nonlinear elliptic equations,}
Duke Math. J.  {\bf 63}(3) (1991), 615-622.

\bibitem{cl1} C.  C.  Chen,    C.  S.  Lin, {\em Sharp estimates for solutions of multi-bubbles in compact Riemann surface.}
{ Comm.  Pure Appl.  Math. } 55 (2002), 728-771.

\bibitem{cl2} C.  C.  Chen,   C.  S.  Lin, {\em Topological degree for a mean field equation on Riemann surfaces.}
{ Comm.  Pure Appl.  Math. } 56 (2003), 1667-1727.

\bibitem{cl3} C.  C.  Chen,   C.  S.  Lin, {\em Mean field equations of Liouville type with singular data: shaper estimates.}
{ Discrete Contin.  Dyn.  Syst. } 28 (2010), 3, 1237-1272.

\bibitem{cl4} C.  C.  Chen,    C.  S.  Lin, {\em Mean field equation of Liouville type with singular data: topological degree.}
{  Comm.  Pure Appl.  Math. }  68  (2015),    6, 887-947.


\bibitem{clw} C.  C.  Chen,    C.  S.  Lin, G. Wang, {\em Concentration phenomena of two-vortex solutions
in a Chern–Simons model.}  {Ann. Sc. Norm. Super. Pisa Cl. Sci. (5)}  3 (2004), 2, 367–397.



\bibitem{dem2} F. De Marchis, {\em Generic multiplicity for a scalar field equation on compact surfaces},
J. Funct. An. (259) (2010), 2165-2192.


\bibitem{DeMSR}  F. De Marchis, R. Lopez-Soriano, D. Ruiz, {\em Compactness, existence and multiplicity for the singular mean
field problem with sign-changing potentials}, Preprint (2016).
 ́

\bibitem{DJLW} W. Ding, J. Jost, J. Li, G. Wang, {\em Existence results for
mean field equations},  Ann. Inst. H. Poincar\'e Anal. Non Lin\'eaire {\bf 16}
(1999), 653--666.


\bibitem{dj} Djadli Z., {\em Existence result for the mean field problem
on Riemann surfaces of all genuses}, Comm. Contemp. Math.  10(2) (2008), 205-220.


\bibitem{EGP} P. Esposito, M. Grossi \& A. Pistoia, {\em On the existence of blowing-up solutions
for a mean field equation}, Ann. Inst. H. Poincar\'e Anal. Non Lin\'eaire {\bf 22}(2) (2005), 227-257.

\bibitem{FL} H. Fang, M. Lai, {\em On curvature pinching of conic 2-spheres}, Calc. Var.  \& P.D.E. {\bf (55)}(2016), 118.


\bibitem{GG} F. Gladiali. M. Grossi, {\em Some Results for the Gelfand's Problem}, Comm. P.D.E. {\bf 29}(9-10) (2004), 
1335-1364.

\bibitem{GM1} C. Gui, A. Moradifam, {\em The Sphere Covering Inequality and Its Applications}, Preprint (2016).

\bibitem{GM2} C. Gui, A. Moradifam, {\em Symmetry of solutions of a mean field equation on flat tori}, Preprint (2016).

\bibitem{GM3}{ C. Gui, A. Moradifam,} {\em Uniqueness of solutions of mean field equations in $\mathbb{R}^2$}, preprint (2016).

\bibitem{KW} J. L. Kazdan, F. W. Warner,
{\em Curvature functions for compact 2-manifolds}, Ann. Math. {\bf 99}  (1974), 14-74.

\bibitem{KMdP} M. Kowalczyk, M. Musso \& M. del Pino,  {\em Singular limits in
Liouville-type equations}, Calc. Var. \& P.D.E. {\bf 24}(1)  (2005), 47-81.

\bibitem{KLin} T.J. Kuo, C.S. Lin, {\em Estimates of the mean field equations with integer singular sources: non-simple blow up},
Jour. Diff. Geom. {\bf 103} (2016), 377-424.

\bibitem{coll-sing1} Y. Lee, C. S. Lin, \emph{Uniqueness for bubbling solutions with collapsing singularities}, Preprint 2017.

\bibitem{coll-sing2} Y. Lee, C. S. Lin, G. Tarantello, W. Yang, \emph{Sharp estimates for the solutions with collapsing singularity}, to appear in Comm. P.D.E.

\bibitem{Lin1} C.S. Lin, {\em Uniqueness of solutions to the mean field equation for the
spherical Onsager Vortex}, Arch. Rat. Mech. An. {\bf  153} (2000), 153-176.

\bibitem{coll-sing3}
C. S. Lin, G. Tarantello, \emph{When "blow-up" does not imply "concentration": A detour from Brezis-Merle's result.} C. R. Math. Acad. Sci. Paris 354 (2016), no. 5, 493-498.



\bibitem{yy} Y.Y. Li,  {\em Harnack type inequality: the method of moving planes},
Comm. Math. Phys.,  {\bf 200} (1999), 421--444.

\bibitem{ls} Y.Y. Li, I. Shafrir, {\em Blow-up analysis for Solutions of $-\Delta u = V(x)e^{u}$
in dimension two}, {Ind. Univ. Math. J.},  {\bf 43}(4) (1994), 1255--1270.



\bibitem{Lin7} C.S. Lin, M. Lucia, {\em Uniqueness of solutions for a
mean field equation on torus},  J. Diff. Eq.  {\bf 229}(1)  (2006), 172-185.

\bibitem{Lin8} C.S. Lin, M. Lucia, {\em One-dimensional symmetry of periodic minimizers
for a mean field equation}, Ann. Sc. Norm. Super. Pisa Cl. Sci. {\bf (6)}2 (2007), 269-290.


\bibitem{linwang} C.S. Lin, C.L. Wang, {\em Elliptic functions, Green functions
and the mean field equations on tori}, Ann. of Math. {\bf 172}(2) (2010), 911-954.



\bibitem{ly2} C.  S.  Lin, S.     Yan, On the Chern-Simons-Higgs equation: Part II, local
uniqueness and exact number of solutions, {preprint.}

\bibitem{Mal1} A. Malchiodi, {\em Topological methods for an elliptic equation with exponential nonlinearities},
Discr. Cont. Dyn. Syst. {\bf 21} (2008), 277--294.

\bibitem{Mal2} A. Malchiodi, {\em Morse theory and a scalar field equation on compact
surfaces}, Adv. Diff. Eq. {\bf 13} (2008), 1109-1129.


\bibitem{NT}
M. Nolasco, G. Tarantello, {\em On a sharp Sobolev-type Inequality on two-dimensional
compact manifold}, {Arch. Rat. Mech. An.} {\bf 145} (1998), 161-195.


\bibitem{pot} A. Poliakovsky, G. Tarantello, {\em On a planar Liouville-type problem in the study of
selfgravitating strings}, J. Differential Equations {\bf 252} (2012), 3668-3693.

\bibitem{PT} J.  Prajapat, G.  Tarantello, {\em On a class of elliptic problems in $%
\mathbb{R}^2$: Symmetry and uniqueness results}, { Proc.  Roy.  Soc.  Edinburgh
Sect.  A. } 131, (2001), 967-985.

\bibitem{sy2} Spruck J., Yang Y., {\em On Multivortices in the Electroweak Theory I:Existence of Periodic Solutions},
Comm. Math. Phys. {\bf 144} (1992), 1-16.

\bibitem{suz} T. Suzuki, {\em Global analysis for a two-dimensional elliptic eiqenvalue problem with the exponential
                nonlinearly}, Ann. Inst. H. Poincar\'e Anal. Non Lin\'eaire {\bf 9}(4) (1992), 367-398.

\bibitem{T0} G. Tarantello,
{\em Multiple condensate solutions for the Chern-Simons-Higgs theory},
{J. Math. Phys.} {\bf 37} (1996), 3769-3796.

\bibitem{tar} G. Tarantello, "{Self-Dual Gauge Field Vortices: An Analytical Approach}",
PNLDE {\bf 72}, Birkh\"auser Boston, Inc., Boston, MA, 2007.

\bibitem{Tar14} G. Tarantello,
{\em Blow-up analysis for a cosmic strings equation}, Jour. Funct. An. 272 (1) (2017) 255-338.

\bibitem{Troy} M. Troyanov, {\em Prescribing curvature on compact surfaces with
conical singularities}, Trans. Amer. Math. Soc. {\bf 324} (1991), 793-821.

\bibitem{w} G. Wolansky, {\em On steady distributions of self-attracting
clusters under friction and fluctuations}, Arch. Rational Mech. An.
{\bf 119} (1992), 355--391.

\bibitem{yang} Y. Yang, "Solitons in Field Theory and Nonlinear Analysis",
Springer Monographs in Mathematics, Springer, New York, 2001.

\bibitem{Za2} L. Zhang, {\em Asymptotic behavior of blowup solutions for elliptic equations
with exponential nonlinearity and singular data}, Commun. Contemp. Math. {\bf 11} (2009), 395-411.

\end{thebibliography}
\end{document}